\documentclass[preprint,10pt]{elsarticle}
\usepackage{amssymb}
\usepackage{atbegshi}
\AtBeginDocument{\AtBeginShipoutNext{\AtBeginShipoutDiscard}}
\usepackage[utf8]{inputenc} 
\usepackage{alltt}
\usepackage{amsfonts}
\usepackage{amsmath}
\usepackage{times}
\usepackage[T1]{fontenc}
\usepackage[english,french]{babel}
\usepackage{lmodern}
\usepackage[a4paper]{geometry}
\usepackage{epstopdf}
\usepackage{mathrsfs}
\usepackage{latexsym,amssymb,amsfonts}
\usepackage{enumerate}
\usepackage{relsize,exscale}
\usepackage{dsfont}
\usepackage{microtype}
\usepackage{xcolor}
\usepackage{enumitem} 
\usepackage{graphicx}
\usepackage{float}
\usepackage{float}
\usepackage{upgreek}
\usepackage{subfigure}

\usepackage[labelfont=bf]{caption}
\usepackage{hyperref}
\usepackage[noabbrev]{cleveref}
\usepackage{longtable}
\hypersetup{
    colorlinks,
    citecolor=red,
    filecolor=red,
    linkcolor=red,
    urlcolor=red
}

\numberwithin{equation}{section} 

\usepackage{amsthm}
\newtheorem{theo}{Theorem}[section]
\newtheorem{coro}{Corollary}[section]
\newtheorem{lemm}{Lemma}[section]
\newtheorem{defi}{Definition}[section]

\newtheorem{rema}{Remark}[section]
\theoremstyle{definition}
\newtheorem{ex}{Example}[section]

\usepackage{indentfirst}
\usepackage[toc,page]{appendix}
\journal{*}
\usepackage{geometry}
\begin{document}
\selectlanguage{english}
\begin{frontmatter}
\title{Advanced and comprehensive research on the dynamics of COVID-19 under mass communication outlets intervention and quarantine strategy: a deterministic and probabilistic approach}
\author{Driss Kiouach\footnote{Corresponding author.\\
\hspace*{0.4cm}E-mail addresses: \href{d.kiouach@uiz.ac.ma}{d.kiouach@uiz.ac.ma} (D. Kiouach), \href{salim.elazamielidrissi@usmba.ac.ma}{salim.elazamielidrissi@usmba.ac.ma} (S. El Azami El-idrissi),\\ \hspace*{2.8cm}\href{yassine.sabbar@usmba.ac.ma}{yassine.sabbar@usmba.ac.ma} (Y. Sabbar).},  Salim El Azami El-idrissi and Yassine Sabbar}
\address{LPAIS Laboratory, Faculty of Sciences Dhar El Mahraz, Sidi Mohamed Ben Abdellah University, Fez, Morocco.}   
\vspace*{1cm}
\begin{abstract}
The ongoing Coronavirus disease 2019 (COVID-19) is a major crisis that has significantly affected the healthcare sector and global economies, which made it the main subject of various fields in scientific and technical research. To properly understand and control this new epidemic, mathematical modelling is presented as a very effective tool that can illustrate the mechanisms of its propagation. In this regard, the use of compartmental models is the most prominent approach adopted in the literature to describe the dynamics of COVID-19. Along the same line, we aim during this study to generalize and ameliorate many existing works that consecrated to analyse the behaviour of this epidemic. Precisely, we propose an SQEAIHR (Susceptible-Quarantined-Exposed-Asymptomatically infective- Infected-Hospitalized-Recovered) epidemic system for Coronavirus. Our constructed model is enriched by taking into account the media intervention and vital dynamics. By the use of the next-generation matrix method, the theoretical basic reproductive number $\mathcal{R}_0$ is obtained for  COVID-19.  Based on some nonstandard and generalized analytical techniques, the local and global stability of the disease-free equilibrium are proven when $\mathcal{R}_0<1$. Moreover,  in the case of $\mathcal{R}_0>1$, the uniform persistence of  COVID-19 model is also shown. In order to better adapt our epidemic model to reality,  the randomness factor is taken into account by considering a proportional white noises, which leads to a well-posed stochastic model.  Under appropriate conditions, interesting asymptotic properties are proved, namely: extinction and persistence in the mean. The theoretical results show that the dynamics of the perturbed COVID-19 model are determined by parameters that are closely
related to the magnitude of the stochastic noise. Finally, we present some numerical illustrations to confirm our theoretical results and to show the impact of media intervention and quarantine strategies.\\[2mm]  
\hspace*{2.1cm} Itô's formula; Extinction; Persistence in the mean.\\[3pt]
\end{abstract}
\end{frontmatter}

\section{Introduction and model formulation}
\par The control of human and animal epidemics is principally based on modeling and simulation as the main decision-making tools \cite{brauer2013mathematical,capasso2008mathematical}. Anyhow, each epidemic is distinguished by its own biological characteristics, which impose the adaptation of the dynamical models describing their propagation mechanisms to any specific case, and this in order to deal with real situations \cite{safarishahrbijari2015particle,el2018mathematical,kiouach2020ergodic}. Presently, the whole world is under a tremendous threat due to the Coronavirus disease, which is a highly contagious virus that first appeared in China at the end of $2019$, and spread rapidly to cover almost the entire globe \cite{wang2020novel}. Several researchers have discovered that this infectious disease, commonly called COVID-19 or $2019$-nCoV, is caused by a new generation of beta-coronavirus \cite{wu2020outbreak}, which affects the lungs and leads to the severe acute respiratory syndrome, the reason why the World Health Organization (WHO) renamed it SARS-CoV-2 \cite{who}. Because of the exponential increase in cases and victims that it caused, the Coronavirus was declared to be an international public health emergency \cite{whoo}. However, the danger did not stop at all, and the disease area carried on expanding to cover more than a hundred countries in March $2020$, until the WHO had considered it as a pandemic in April of the same year when the statistics revealed that the number of infected populations and deaths surpassed respectively $800,000$ and $40,000$ \cite{ivorra2020mathematical}. Despite the existence of many suggested COVID-19 vaccines  for which some national  authorities have already permited the emergency use \cite{whoooo}, none is proven to be completly safe yet. According to the WHO, these vaccines are not rigoursly tested and still in the phase of large clinical trial (see \cite{cdc,whooo}). The absence of an WHO-officially authorized or recommended treatment \cite{whooo} remains a genuine challenge for all the governments, especially with the significant and noticeable repercussions that this epidemic presents in the economic and health realms. In the case of this new virus, the majority of transmissions is occurred by respiratory droplets that may be inhaled from close contact with an infected person when he exhales, sneezes or coughs \cite{zeb2020mathematical}. Moreover, these droplets fall quickly on the floors and surfaces which makes them also a possible source of infection. Currently, the adoption of suitable strategies to tackle Coronavirus transportation presents big defiance for all the decision-makers around the world, and to overcome it, a good understanding of the pandemic evolution dynamics is really required.\par Developing an appropriate mathematical model is a prominent method to purvey essential instructions and guidelines measures for disease mitigation. In this context, compartmental systems, like the simple SIR or the more advanced ones such as SIRS, SEIRS, SEIRQ and others, (see \cite{brauer2008compartmental,capasso2008mathematical}) can be an inspiring choice to deal with the critical situation that we are going through now. Since its discovery, a lot of models have been suggested for the study of COVID-19 dynamics \cite{kucharski2020early,roosa2020real,fanelli2020analysis}. 
 Some works like \cite{zhong2020early}, suggested a classical SIR system for predicting and analysing the novel Coronavirus, other ones such as \cite{ming2020breaking,nesteruk2020statistics} used a modified version of this system for the purpose of being more adapted to the dynamics transmission of SARS-CoV-2. With regards to this new virus characteristics, the last mentioned model was expanded to an SEIR one by taking into consideration a latency period during which the infected individuals are not infectious yet. Yang and Wang \cite{yang2020mathematical} used exactly this extended model to describe COVID-19 dissemination in Wuhan, China. They considered various transmission ways and incorporated the importance of the environmental reservoir. In the same context, Kucharski et al. \cite{kucharski2020early} adopted an SEIR model to study the transmission variation of COVID-19 at the beginning of 2020. For a more rigorous vision of deaths number evolution, the authors in \cite{rajagopal2020fractional} developed a new form of the SEIR model by adding a deceased individuals class denoted (D). They utilized a fractional-order formulation SEIRD and arrived to the fact that this latter is more suitable than the classical one introduced in \cite{unknown}, and has less root mean square deviation. Inspired by the studies presented in \cite{ivorra2015codis,ferrandez2019application} and \cite{ferrandez2020multi}, Ivorra et al. \cite{ivorra2020mathematical} treated a fractional SEIHRD model to give successful forecasts about future variation in cases and deaths statistics, bearing in mind the existence of undetected infectious individuals. They proposed a new strategy that considers a proportion of the detected cases over total infected ones and exhibited the impact of this percentage, usually denoted $\theta$, on the spread level of COVID-19. We mention that there is many interesting works related to the estimations of $\theta$ for the Coronavirus (see, e.g., \cite{li10substantial,russell2020using}). By considering the isolation strategy which played a significant role in controlling many diseases, Pal et al \cite{pal2020mathematical} proposed an SEQIR system to describe the outbreak of the Coronavirus. They determined the basic reproduction number $\mathcal{R}_0$ in terms of system parameters, and on the basis of its expression, they analysed the stability dynamics of their model. In the same regard, Hu et al. \cite{hu2020evaluation} showed the effect of the strategy stated above on the COVID-19 evolution by using an SEIRQ compartmental system that takes into account the input population. They discussed different scenarios of the epidemic spread on Guangdong province by using the officially published data. In order to be well adapted to the prevalence mechanisms of the current epidemic, Jia et al \cite{jia2020modeling} suggested a new mathematical model that counts on the isolation and treatment besides home quarantine as prominent strategies to reduce the contagion intensity, without neglecting the existence of an asymptomatic transmission. In the case of COVID-19, it is important to point out that there seems to be a myriad of asymptomatic infected individuals \cite{rothe2020transmission} with a significant case fatality ratio, but it stills noticeably lower than MERS-CoV or SARS-CoV \cite{guan2020clinical}. They also effectuated the parameters estimation building on officially published data and the well known Least-Squares approach before passing to calculate the control reproduction number $\mathcal{R}_c$ for most Chinese provinces. It should be noted that the aforementioned study supposed the existence of meteorological factors effects on the virus activity, justifying this by the high level of genetic similarity between the current SARS-COV-2 and SARS-COV, and the role of temperature elevation in the disappear of this latter in 2003 \cite{darnell2004inactivation}. In spite of its formulation complexity, the previously mentioned model has a realistic hypothetical framework, which made it a perfect basis for many interesting contributions \cite{wu2020quantifying,mohsen2020global}.
 \par Along the same line, in this work, we will propose and analyze a generalized form of the model presented in \cite{jia2020modeling}. Our new version will comprise two main extra hypotheses: the demographic variations and media intervention. The first addition's objective is to enhance this model and make it able to describe the behaviour of the current pandemic over a long period of time  by including natural birth and general mortality rates. For the second one, we consider the role of the various awareness campaigns and daily reports announced by different mass communication outlets like radio, television, newspaper, internet etc. In order to formulate the model mathematically under the previous assumptions, we consider a host population denoted at time $t$ by $N(t)$ which is partitioned into seven classes of susceptible, quarantined, exposed, infectious with symptoms, asymptomatically infected, hospitalized and recovered individuals, with densities respectively denoted by $S(t), Q(t), E(t), A(t), I(t), D(t)$ and $R(t)$. The overall interactions between these classes are described through the system of deterministic ordinary differential equations below:
\begin{equation}\label{detr}
\left\lbrace
 \begin{aligned}
	\dfrac{\mathrm{d}S}{\text{d}t}	&=\Lambda-\left(\beta_1-\beta_2\dfrac{I}{b+I}\right) S\left(I+\theta A\right)+\lambda Q -(\mu+q)S,\\
		\dfrac{\mathrm{d}Q}{\text{d}t}	&=qS-\left(\mu+\lambda\right)Q,\\
		\dfrac{\mathrm{d}E}{\text{d}t} &=\left(\beta_1-\beta_2\dfrac{I}{b+I}\right)S\left(I+\theta A\right)-(\mu+\sigma)E,\\
		\dfrac{\mathrm{d}A}{\text{d}t}	&=\left(1-p\right)\sigma E-\left(\mu+\varepsilon_A+\gamma_A+d_A\right)A,\\
\dfrac{\mathrm{d}I}{\text{d}t}&=\sigma p E-\left(\mu+\varepsilon_I+\gamma_I+d_I\right)I,\\	
\dfrac{\mathrm{d}H}{\text{d}t}& =\varepsilon_I I+\varepsilon_A A-\left(\mu+d_H+\gamma_H \right)H,\\
\dfrac{\mathrm{d}R}{\text{d}t} &=\gamma_H H+\gamma_I I+\gamma_A A-\mu R.
\end{aligned}
\right.
\end{equation}
In this system, it was assumed that exposed individuals $E(t)$ are not infectious, and being in reality low-level virus carriers \cite{ivorra2020mathematical} explains the adoption of this assumption. Due to the exerted efforts by local authorities, like the closure of educational institutions, traffic control, travel restrictions, extension of vacations and postponing the return to work, we accept in \eqref{detr} that there is no contact chance between home quarantined individuals $Q(t)$ and infectious population. For the hospitalized people $H(t)$, we suppose that they are isolated and in treatment process. Based on the studies presented in \cite{karimi2020vertical,lu2020asymptomatic,schwartz2020infections}, it is admitted in \eqref{detr} that there is no vertical transmission from mother to foetus, also, we presume that the total lockdown policy is applied. So, the recruitment rate $\Lambda$ of uninfected population corresponds only to natural births. In addition, and similarly to \cite{mohsen2020global,kiouach2019modeling}, a media-induced incidence function is used to depict the disease infection mechanism in our system. We take first $\beta_1$ as the usual contact rate before media intervention and we suppose that it undergoes a reduction expressed by $\dfrac{\beta_2I}{b+I}$, when infective people are reported in the media. Note that $\beta_2$ is none other than the maximum reduced contact rate due to the presence of infective individuals $I(t)$ $\left(\text{i.e.}~ \max\limits_{I\geqslant 0} \dfrac{\beta_2I}{b+I}\right)$ and $m$ is the half-saturation constant that presents the effect of media alert on the transmission rate (see \cite{liu2008impact} for more details). It should be clarified that media coverage is able to reduce the intensity of the disease propagation, but cannot completely prevent it, therefore, we assume that $\beta_1\geqslant \beta_2>0$. Logically speaking, symptomatic infected individuals should have a greater contagion rate than asymptomatic ones, and this is the clue why we have denoted the ration between these two rates as $\theta\in(0,1)$. Regarding the home confined population, we have used parameters $p$ and $\lambda$ in reference respectively to the quarantined and release rates. The constant $\sigma$ in \eqref{detr} is actually the passage rate from exposed individuals class to infected compartments $A$ and $I$, but with a probability $p\in(0,1)$ of becoming symptomatic and $1-p$ for being asymptomatic. The parameters $\varepsilon_A$ and $\varepsilon_I$ stand respectively for the hospitalization rates of asymptomatic and symptomatic infective, whereas $\gamma_A$,$\gamma_I$ and $\gamma_D$ are the recovery rates of classes $A,I,D$. Moreover, the constants $\mu$, $d_A$, $d_I$ and $d_H$ are representing, in this order, the natural death rate of the whole population and the disease-induced death rates affecting only from classes $D,A$ and $I$. Based on the abovementioned modelling assumptions, the proposed system is illustrated by the following schematic flow diagram: 
\begin{figure}[hbtp]
\centering
\includegraphics[scale=0.37]{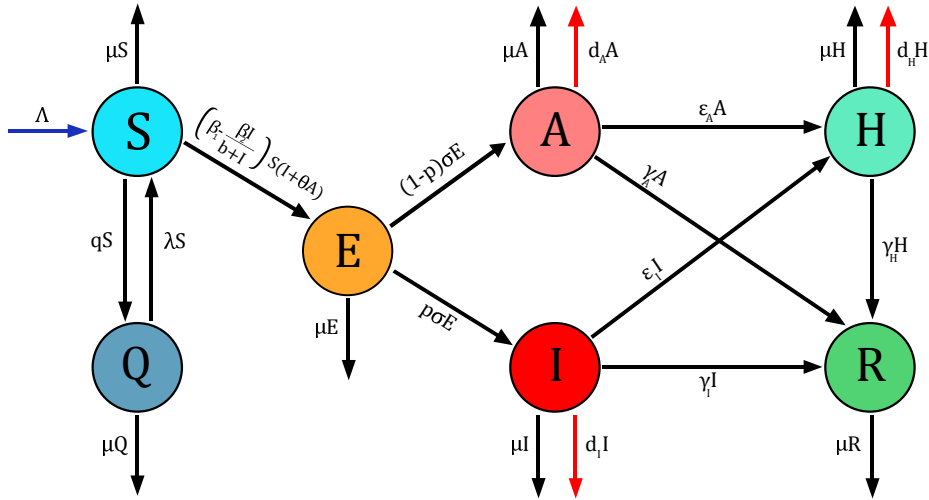}\label{figure}
\caption{Schematic diagram of the proposed COVID-19 epidemic model.}
\end{figure}
\par Generally speaking, the deterministic formulations analysis is very necessary and commonly used in the mathematical epidemiology, and it can be seen as a first tool for modelling new diseases spread and getting an overview of their asymptotic behaviour. However, the real phenomena are not always deterministic and may be subject to some uncertainties and randomness due to fluctuations in the natural environment \cite{kiouach2018stability,kiouach2019threshold,zhang2017threshold}. Therefore, an adapted version that considers this stochasticity is requisite in the case of COVID-19. For this purpose, and like several works \cite{liu2019dynamics,cai2017stochastic,kiouach2019modeling}, we extend the system \eqref{detr} to the following probabilistic version that incorporates proportional Gaussian white noises: 
\begin{equation}\label{systo} 
\left\lbrace
 \begin{aligned}
	\mathrm{d}S	&=\left[\Lambda-\left(\beta_1-\beta_2\dfrac{I}{b+I}\right) S\left(I+\theta A\right)+\lambda Q -(\mu+q)S\right]\mathrm{d}t+\sigma_1S~ \textup{d}B_{1}(t),\\
		\mathrm{d}Q	&=\left[qS-\left(\mu+\lambda\right)Q\right]\mathrm{d}t+\sigma_{2}Q~\mathrm{d}B_2(t),\\
		\mathrm{d}E &=\left[\left(\beta_1-\beta_2\dfrac{I}{b+I}\right)S\left(I+\theta A\right)-(\mu+\sigma)E\right]\text{d}t+\sigma_{3}E~ \mathrm{d}B_3(t),\\
		\mathrm{d} A	&=\left[\left(1-p\right)\sigma E-\left(\mu+\varepsilon_A+\gamma_A+d_A\right)A\right]\text{d}t+\sigma_{4}A~\mathrm{d}B_4(t),\\
\text{d}I&=\left[\sigma p E-\left(\mu+\varepsilon_I+\gamma_I+d_I\right)I\right]\text{d}t+\sigma_5I~\text{d}B_5(t),\\	
\text{d}H& =\left[\varepsilon_I I+\varepsilon_A A-\left(\mu+d_H+\gamma_H \right)H\right]\text{d}t+\sigma_6H~\text{d}B_6(t),\\
\text{d}R &=\left[\gamma_H H+\gamma_I I+\gamma_A A-\mu R\right] \text{d}t+\sigma_7 R~\text{d}B_7(t).
\end{aligned}
\right.\vspace*{4pt}
\end{equation}
Here and subsequently, $\sigma_i~(i=1,\cdots,7)$ are denoting the positive intensities of the mutually independent Brownian motions $B_i~(i=1,\cdots,7)$. These latter, and all the random variables that we will meet in this paper, are presumed to be defined on a complete probability space $\left(\Omega,\mathfrak{F},\mathbb{P}\right)$ with a filtration $(\mathfrak{F}_t)_{t\geqslant 0}$, that satisfies the usual conditions, i.e., it is increasing and right continuous while $\mathfrak{F}_0$ contains all $\mathbb{P}-$null sets. 
\par The primary purpose of this paper is to provide an overview of COVID-19 dynamics and examine its various properties over a relatively long period. Indeed, we have taken the necessary time to observe the disease evolution in various countries and to compare its behaviour with the theoretical epidemic models presented in the literature, which permitted us later to propose a system that appears suitable and very adapted to the evolution of this new disease. Since the asymptotic analysis of epidemic models provides an excellent insight into the future pandemic situation, we decide in this study to treat it for our suggested system of COVID-19 on both deterministic and probabilistic levels. In the deterministic framework, we highlight several methods which make it possible to obtain, and under certain conditions, of course, the local and global stability of the disease-free equilibrium as well as the uniform persistence of the epidemic, without resorting to Lyapunov theory \cite{khalil2002nonlinear}. For the perturbed version of our model, we check in detail its well-posedness. Furthermore, we study the extinction case and present by using some new techniques a sufficient condition for the persistence of Coronavirus.
\par The rest of this manuscript is arranged as follows: in Section \ref{sec1}, we demonstrate that the deterministic system \eqref{detr} has a unique global positive solution. Next, we show the existence and uniqueness of disease-free and endemic equilibria of this system. Moreover, we investigate besides the local and global stability of disease-free equilibrium (DFE) the uniform persistence of the epidemic. In Section \ref{sec2}, we will first prove the existence and uniqueness of the positive solution of the stochastic system \eqref{systo}. Then, we show the extinction and persistence in the mean of COVID-19 under some conditions. In Section \ref{sec3}, we perform our work with some examples of numerical simulations to illustrate our theoretical results and the effect of control strategies (coverage media and quarantine). Finally, results are discussed in Section \ref{sec4}.
\section{Analysis of the deterministic Coronavirus model}\label{sec1}
This section is devoted to the study of the deterministic  COVID-19 model presented by the aforementioned differential system \eqref{detr},  but before starting our analysis, we will first associate to \eqref{detr} the following initial-value problem:
\begin{equation}\label{detr2}
\begin{cases}
\overbrace{\dfrac{\text{d}\displaystyle{\mathlarger{x}}(t)}{\text{d}t}}^{7\times 1}=\overbrace{F\left(\mathlarger{x}(t)\right)}^{7\times 1},\\
\mathlarger{\mathlarger{x}}(0)=\displaystyle{\mathlarger{x}_0}\in \mathbb{R}_+^7, 
\end{cases}
\end{equation}
where
\begin{equation}
\mathlarger{x}(t)=\big(\mathlarger{x}_i(t)\big)_{1\leqslant i\leqslant 7}=\big(S(t),Q(t),E(t),A(t),I(t),H(t),R(t) \big),\nonumber\\[5pt] 
\end{equation}
and
\begin{equation}\label{F}
F\left( \mathlarger{x}(t)\right)=\big(F_i\left(\mathlarger{x}(t)\right)\big)_{1\leqslant i\leqslant 7}=\left( \begin{matrix}
\Lambda-\left(\beta_1-\beta_2\dfrac{I}{b+I}\right)S\left(I+\theta A\right)+\lambda Q -(\mu+q)S \\ 
qS-\left(\mu+\lambda\right)Q \\ 
\left(\beta_1-\beta_2\dfrac{I}{b+I}\right) S\left(I+\theta A\right)-(\mu+\sigma)E \\ 
\left(1-p\right)\sigma E-\left(\mu+\varepsilon_A+\gamma_A+d_A\right)A \\ 
\sigma p E-\left(\mu+\varepsilon_I+\gamma_I+d_I\right)I \\ 
\varepsilon_I I+\varepsilon_A A-\left(\mu+d_H+\gamma_H\right)H \\ 
\gamma_H H+\gamma_I I+\gamma_A A-\mu R
\end{matrix} \right).
\end{equation}
\subsection{Well-posedness}
In this subsection, we will show that the Coronavirus model presented by system \eqref{detr} is well posed, in the sense that if $S_0,Q_0,E_0,A_0,I_0,H_0$ and $R_0$ are positive, then  the initial-value problem, or the Cauchy problem \cite{vrabie2004differential}, \eqref{detr2} admits one and only one solution which is global in time, positive, and bounded.     
\begin{theo}\label{thm1*}
If the initial value $\displaystyle{\mathlarger{x}}_0$ is in the positive orthant $\mathbb{R}_+^7$, then there exists a unique solution $\displaystyle{\mathlarger{x}(t)}$ to the system \eqref{detr2}. Moreover, this solution remains bounded and positive for all $t\geqslant0$.     
\end{theo}
\begin{proof}
First of all, it is obvious that $F$ is a locally Lipschitz continuous function on the open connected domain $\mathbb{R}_+^7$. Therefore, the famous Picard-Lindl\"{o}f theorem \cite{platzer2018logical}, called also the Cauchy-Lipschitz theorem \cite{ding2007approaches}, guarantees the existence of a unique maximal solution to the initial-value problem \eqref{detr2}. This solution is defined on an interval $[0,\tau_e)$, where $\tau_e$ is the explosion time \cite{haddad2011nonlinear}. Moreover, it can be easily seen from \eqref{detr} that for all $\big(S,Q,E,A,I,H,R\big)\in\mathbb{R}_+^7,$
\begin{equation*}
\left\lbrace
\begin{aligned}
F_1\left(0,Q,E,A,I,H,R\right)&=\Lambda+\lambda Q\geqslant 0,\\
F_2\left(S,0,E,A,I,H,R\right)&=qS\geqslant 0,\\
F_3\left(S,Q,0,A,I,H,R\right)&=\left(\beta_1-\beta_2\dfrac{I}{b+I}\right) S\left(I+\theta A\right)\geqslant 0,\\
F_4\left(S,Q,E,0,I,H,R\right)&=\sigma(1-p)E\geqslant 0,\\
F_5\left(S,Q,E,A,0,H,R\right)&=\sigma pE\geqslant 0,\\
F_6\left(S,Q,E,A,I,0,R\right)&=\varepsilon_I I+\varepsilon_A A\geqslant 0,\\
F_7\left(S,Q,E,A,I,H,0\right)&=\gamma_H H+\gamma_I I+\gamma_A A\geqslant 0.
\end{aligned}
\right.
\end{equation*}
So, by using Proposition 4.1 of \cite{haddad2005stability}, it follows immediately that $\mathbb{R}_+^7$ is positively invariant for the dynamical system \eqref{detr}, which means that any solution trajectory starting from the cone $\mathbb{R}_+^7$ will still inside it for all the future time $t\in[0,\tau_e)$. On the other hand, for any initial value $\big(S_0,Q_0,E_0,A_0,I_0,H_0,R_0\big)\in \mathbb{R}_+^7$, the total population $N(t)$ satisfies the following equality for all $t\in[0,\tau_e)$,
\begin{equation}\label{equa}
\dfrac{\textup{d}N(t)}{\textup{d}t}=\Lambda-\mu N(t)-d_A A(t)-d_I I(t)-d_H H(t).
\end{equation} 
From the positivity of the solution, we obtain 
\begin{equation}\label{gron}
\dfrac{\textup{d}N(t)}{\textup{d}t}\leqslant\Lambda-\mu N(t).
\end{equation}
Applying the well-known Gronwall's inequality in its differential form (see \cite[page 20]{duan2015introduction}) to \eqref{gron} yields
\begin{equation}\label{equa1}
N(t)\leqslant \dfrac{\Lambda}{\mu}+\left(N_0-\dfrac{\Lambda}{\mu}\right)e^{-\mu t}\leqslant\dfrac{\Lambda}{\mu}+N_0.
\end{equation}
Therefore, the solution of system \eqref{detr2} is bounded on $[0,\tau_e)$. In addition, this last fact together with Proposition A.1 of \cite{maury2018crowds} lead to the globality in time of our solution. Thus, the initial-value problem \eqref{detr2} is well posed as required.   
\end{proof}
\begin{rema}
In addition to its interesting mathematical interpretation, the importance of the preceding theorem lies essentially in the fact that it makes our model logical, realistic and biologically meaningful. 
\end{rema}
\subsection{Invariant, absorbing and attracting regions}
In the theory of ordinary differential equations, the first step after proving the well-posedness of a dynamical system is to get a first glimpse of its asymptotic behaviour's nature. For example, to know if its trajectories are bounded on a neighbourhood of infinity or not. For that, we will present in this subsection some special regions of $\mathbb{R}_+^7$ to which the solutions of \eqref{detr} are belongs always, converge, or at least remain close as $t\to\infty$. This situation is often described in terms of invariant, absorbing and attracting subsets which will be defined in the following.
\begin{defi}[Invariant region \cite{slotine1991applied,zill2011advanced}]A set $\Delta\subset\mathbb{R}^7$ is said to be an invariant region for \eqref{detr}, or more precisely a positively invariant region with respect to \eqref{detr}, if $\displaystyle{\mathlarger{x}(0)}\in\Delta$ implies $\displaystyle{\mathlarger{x}(t)}\in \Delta$ for all $t\geqslant 0$. That is, whenever $\displaystyle{\mathlarger{x}_0}$ is in $\Delta$, the solution $\mathlarger{\mathlarger{x}}(t)$ of \eqref{detr} which starts from $\displaystyle{\mathlarger{x}}_0$ remains in $\Delta$ for all future time. 
\end{defi}
\begin{rema}
It is fairly easy to see from Theorem \ref{thm1*} that $\mathbb{R}_+^7$ is a positively invariant set for \eqref{detr} in the sense of the previous definition.
\end{rema}  
\begin{defi}[Absorbing region \cite{milani2004introduction}]
We say that a region $\Delta\subset\mathbb{R}^7_+$ is absorbing for \eqref{detr} if there is an open set $\mathcal{U}\subset\mathbb{R}^7_+$ containing it such that $\displaystyle{\mathlarger{x}(0)}\in\mathcal{U}$ implies the existence of some $T>0$ for which $\displaystyle{\mathlarger{x}(t)}\in \Delta$ whenever $t\geqslant T.$ 
In other words, any solution of \eqref{detr} that starts from $\mathcal{U}$, will necessarily enter and stay in $\Delta$ after a certain time. In the particular case of $~\mathcal{U}=\mathbb{R}^7_+$, the set $\Delta$ is called globally absorbent for \eqref{detr}. 
\end{defi}
\begin{rema}
The existence of a bounded absorbing set is considered as a specific and interesting property of the system, and generally
this property is called dissipativity \cite{milani2004introduction}.
\end{rema}
\begin{defi}[Attracting region \cite{milani2004introduction,khalil2002nonlinear}]
A set $\Delta\subset\mathbb{R}^7_+$ is called an attracting region for the system \eqref{detr} if there exists an open set $\mathcal{U}\subset\mathbb{R}^7_+$ containing it such that $$\lim\limits_{t\to\infty}\operatorname{dist}\left(\displaystyle{\mathlarger{x}}(t,\displaystyle{\mathlarger{x}}_0)\textbf{,}\Delta\right)=0~~\text{for all}~\displaystyle{\mathlarger{x}}_0\in\mathcal{U}.$$ 
 Here, $\displaystyle{\mathlarger{x}}(t,\displaystyle{\mathlarger{x}}_0)$ denotes the solution of \eqref{detr} that satisfies the initial condition $\displaystyle{\mathlarger{x}}(0)=\displaystyle{\mathlarger{x}_0}$ and $\operatorname{dist}\left(\textup{p}\textbf{,}\mathbb{B}\right)$ is the distance from a point $\textup{p}\in\mathbb{R}^7$ to a subset $\mathbb{B}\subset\mathbb{R}^7 $, that is, the smallest distance from $\textup{p}$ to any point in $\mathbb{B}$. More explicitly,
 $$\operatorname{dist}\left(\textup{p}\textbf{,}\mathbb{B}\right):=\inf\limits_{x\in \mathbb{B}}\|x-\textup{p}\|.$$ 
 If $~\mathcal{U}=\mathbb{R}_+^7$, the region $\Delta$ will be said to be globally attracting for \eqref{detr}.
\end{defi}
\begin{rema}
Clearly, an absorbent set is always attracting, but the converse is generally false (see \cite{milani2004introduction}).
\end{rema}
After having defined the notions of invariant, absorbing and attracting sets, we will turn now to the main results of this subsection. But before doing so, let us first introduce these notations which will be used from now on to simplify the writing: 
\begin{equation}\label{nota}
 S^o=\dfrac{\Lambda}{\mu}\times\dfrac{\mu+\lambda}{\mu+q+\lambda}~~\text{and}~~
 Q^o=\dfrac{\Lambda}{\mu}\times\dfrac{q}{\mu+q+\lambda}.
\end{equation}  
\begin{theo}\label{thinv}
Let $\eta,\eta^\prime\geqslant 0$ be two nonnegative numbers, and let $S^o,Q^o$ be defined by \eqref{nota}. The following regions:
\begin{enumerate}
\item $\mathcal{D}^\eta=\left\{\left(S,Q,E,A,I,H,R\right)\in\mathbb{R}_+^7\mid S+Q+E+A+I+H+R\leqslant\frac{\Lambda}{\mu}+\eta\right\},$
\item $\mathcal{D}_{\eta^\prime}=\left\{\left(S,Q,E,A,I,H,R\right)\in\mathbb{R}_+^7\mid S+Q+E+A+I+H+R\geqslant\frac{\Lambda}{\mu+d_I+d_A+d_H}-\eta^\prime\right\},$
\item $\mathcal{D}_{\eta^\prime}^\eta=\left\{\left(S,Q,E,A,I,H,R\right)\in\mathbb{R}_+^7\mid \frac{\Lambda}{\mu+d_I+d_A+d_H}-\eta^\prime\leqslant S+Q+E+A+I+H+R\leqslant \frac{\Lambda}{\mu}+\eta \right\},$
\item $\mathcal{S}=\left\{\left(S,Q,E,A,I,H,R\right)\in\mathbb{R}_+^7\mid S\leqslant S^o~~\text{and}~~ Q\leqslant Q^o\right\},$
\end{enumerate} 
are positively invariant with respect to the dynamical system \eqref{detr}. 
\end{theo}
\begin{proof}
For the reader's convenience, we will divide the proof into four parts, each of which will be devoted to showing the invariance of one of the sets appearing in the last theorem's statement.\\
Throughout this demonstration, $\displaystyle{\mathlarger{x}_0}=\big(S_0,Q_0,E_0,A_0,I_0,H_0,R_0\big)\in \mathbb{R}_+^7$ is a given value and $\displaystyle{\mathlarger{x}}(t)\hspace*{-1pt}=\hspace*{-1pt}\big(S(t),Q(t),E(t),A(t)\\,I(t),H(t), R(t)\big)$ is the solution of \eqref{detr} that starts from it.
\begin{enumerate}
\item Let $\eta\geqslant 0$ and suppose that $\displaystyle{\mathlarger{x}_0}\in \mathcal{D}^\eta$. From \eqref{equa}, we have for all $t\geqslant 0$
\begin{equation}
N(t)\leqslant \dfrac{\Lambda}{\mu}+\left(N_0-\dfrac{\Lambda}{\mu}\right)e^{-\mu t},
\end{equation}
and since $\displaystyle{\mathlarger{x}_0}\in \mathcal{D}^\eta$, it follows that $N_0\in \left(0,\dfrac{\Lambda}{\mu}+\eta\right)$, then
$$N(t)\leqslant \dfrac{\Lambda}{\mu}+\left(N_0-\dfrac{\Lambda}{\mu}\right)e^{-\mu t}\leqslant \dfrac{\Lambda}{\mu}+\eta\times e^{-\mu t}\leqslant \dfrac{\Lambda}{\mu}+\eta.$$
Therefore, $\displaystyle{\mathlarger{x}}(t)\in \mathcal{D}^\eta$ for all $t\geqslant 0$, which means that $\mathcal{D}^\eta$ is positively invariant for \eqref{detr}.  
\item According to \eqref{equa1}, we have 
$$\dfrac{\textup{d}N(t)}{\textup{d}t}=\Lambda-\mu N(t)-d_A A(t)-d_I I(t)-d_H H(t),\quad \forall t\geqslant 0,$$
and by using the positivity of solutions we get for all $t\geqslant 0$  
$$\dfrac{\textup{d}N(t)}{\textup{d}t}\geqslant \Lambda-\left(\mu+d_A+d_H+d_I\right) N(t).$$
Thus,
\begin{equation}\label{equa2}
N(t)\geqslant \dfrac{\Lambda}{\mu+d_A+d_H+d_I}+\left(N_0-\dfrac{\Lambda}{\mu+d_A+d_H+d_I}\right)e^{-\mu t}.
\end{equation}
Then, if $\displaystyle{\mathlarger{x}_0}\in\mathcal{D}_{\eta^\prime}$, we obtain
$$N(t)\geqslant \dfrac{\Lambda}{\mu+d_A+d_H+d_I}-\eta^\prime e^{-\mu t}\geqslant \dfrac{\Lambda}{\mu+d_A+d_H+d_I}-\eta^\prime,$$
which is equivalent to saying that $\mathcal{D}_{\eta^\prime}$ is positively invariant with respect to \eqref{detr}.
\item The invariance of region $\mathcal{D}_{\eta^\prime}^\eta= \mathcal{D}_{\eta^\prime} \cap{\mathcal{D}}^\eta $ with respect to system \eqref{detr} follows immediately from the fact that the intersection of two positively invariant sets is also a positively invariant set (see Proposition 2.31, \cite{haddad2005stability}).
\item Let $\displaystyle{\mathlarger{x}_0}\in\mathcal{S}$. From the system \eqref{detr}, we remark that
\begin{align*}
  \left.\frac{\text{d}S}{\text{d}t}\right|_{S=0}\hspace{-10pt}&=\Lambda+\lambda Q\geqslant 0, & \left.\frac{\text{d}Q}{\text{d}t}\right|_{Q=0}\hspace{-12pt}&=qS\geqslant 0,  & \left.\frac{\text{d}E}{\text{d}t}\right|_{E=0}\hspace{-10pt}&=\left(\beta_1-\dfrac{\beta_2 I}{b+I}\right) S\left(I+\theta A\right)\geqslant 0,\\
   \left.\frac{\text{d}A}{\text{d}t}\right|_{A=0}\hspace{-10pt}&=\sigma(1-p)E\geqslant 0, & \left.\frac{\text{d}I}{\text{d}t}\right|_{I=0}\hspace{-10pt}&=\sigma pE\geqslant 0,  & \left.\frac{\text{d}H}{\text{d}t}\right|_{H=0}\hspace{-12pt}&=\varepsilon_I I+\varepsilon_A A\geqslant 0,\\
 \left.\frac{\text{d}R}{\text{d}t}\right|_{R=0}\hspace{-10pt}&=\gamma_H H\hspace{-1pt}+\hspace{-1pt}\gamma_I I\hspace{-1pt}+\hspace{-1pt}\gamma_A A\geqslant\hspace{-1pt} 0, & \left.\frac{\text{d}Q}{\text{d}t}\right|_{Q=Q^o}\hspace{-18pt}&=q(S\hspace{-1pt}-\hspace{-1pt}S^o)\leqslant \hspace{-1pt} 0, &  \left.\frac{\text{d}S}{\text{d}t}\right|_{S=S^o}\hspace{-18pt}&= \lambda\left(Q-Q^o\right)\hspace{-1pt}-\hspace{-1pt}\beta S^o \left(I\hspace{-1pt}+\hspace{-1pt}\theta A\right)\leqslant 0. 
\end{align*}    
Thus, the vector field on each boundary 
point of the region $\mathcal{S}$ is pointing toward its interior, and this prevents any solution that starts in $\mathcal{S}$ from leaving it (see \cite{zill2011advanced}). Therefore, $\mathcal{S}$ is an invariant region for \eqref{detr}.
\end{enumerate}
Hence, the proof is completed.
\end{proof}
\begin{rema}
The method used to prove the last point of Theorem \ref{thinv} is known as the Bony-Brezis theorem, and for more details on this, we refer the reader to \cite{redheffer1972theorems}.
\end{rema}

From now on, we will keep the notations of Theorem \ref{thinv}\textbf{.}
\begin{theo}\label{inva}
For any positive numbers $\eta$ and $\eta^\prime$, the set $\mathcal{D}_{\eta^\prime}^\eta$ is globally absorbent with respect to the system \eqref{detr}.    
\end{theo}
\begin{proof}
Let $\eta,\eta^{\prime}>0$ be two positive numbers, and let $x_0\in\mathbb{R}_+^7$. Based on $\eqref{equa}$ and $\eqref{equa2}$, we observe that for all $t\geqslant0$
\begin{equation}\label{enca}
\dfrac{\Lambda}{\mu+d_A+d_H+d_I}+\left(N_0-\dfrac{\Lambda}{\mu+d_A+d_H+d_I}\right)e^{-\mu t} \leqslant N(t)\leqslant \dfrac{\Lambda}{\mu}+\left(N_0-\dfrac{\Lambda}{\mu}\right)e^{-\mu t}.
\end{equation} 
By letting $t$ tend to infinity in \eqref{enca}, we obtain
  $$\dfrac{\Lambda}{\mu+d_A+d_H+d_I}\leqslant \lim_{t\to\infty} N(t)\leqslant \dfrac{\Lambda}{\mu}.$$
So, necessarily there is a certain time, depending in $\eta$ and $\eta^\prime$, from which we have
$$\dfrac{\Lambda}{\mu+d_A+d_H+d_I}\leqslant N(t)\leqslant\dfrac{\Lambda}{\mu}.$$
Therefore, $\mathcal{D}_{\eta^{\prime}}^\eta$ is a globally absorbent region for \eqref{detr}. 
\end{proof}
\begin{rema}
Needless to say, the last theorem implies that the system \eqref{detr} is dissipative. Moreover $\mathcal{D}_{\eta^\prime}, \mathcal{D}^\eta$, or generally any region containing $\mathcal{D}_{\eta^\prime}^\eta$, is also going to be globally absorbent for \eqref{detr}. 
\end{rema}
Now, let us mention an important consequence of the previous theorem.
\begin{coro}
With respect to the dynamical system \eqref{detr}, the region
  $$\mathcal{D}_0^0=\left\{\left(S,Q,E,A,I,H,R\right)\in\mathbb{R}_+^7\mid \frac{\Lambda}{\mu+d_I+d_A+d_H}\leqslant S+Q+E+A+I+H+R\leqslant \frac{\Lambda}{\mu}\right\}$$ is globally attracting. 
\end{coro}
\begin{proof}
The attraction property of the set $\mathcal{D}_{0}^0$ with respect to \eqref{detr} comes easily from the following observation: 
$$\mathcal{D}_0^0=\bigcap_{\eta>0}\mathcal{D}_{\eta}^\eta~,$$
and the fact that the region $\mathcal{D}_{\eta}^\eta$ is globally absorbent with respect to \eqref{detr} for any positive number $\eta$ (see Theorem \ref{inva} in the particular case $\eta=\eta^\prime$). The last passage can be explained by saying that reducing the $\eta$ value, and bringing it close to zero leaves no choice for the solutions trajectories except converging to $\mathcal{D}_0^0$, which is the desired conclusion.
\end{proof}
\subsection{The Basic reproduction number \texorpdfstring{$\mathcal{R}_0$}{R0}}
Epidemiologically, the basic reproduction ratio $\mathcal{R}_0$ is the number of secondary cases produced by one infected individual in an entirely susceptible population during its period as an infective. In the literature, several techniques have been proposed for the calculation of $\mathcal{R}_0$ \cite[chapter 5]{martcheva2015introduction}, but the most known is that of the next generation approach introduced by van den Driessche and Watmough in \cite{van2002reproduction}. According to this method, which will be used in our case, $\mathcal{R}_0$ is the spectral radius $\rho$ of the next generation matrix defined by $M=FV^{-1}$, where $V$ and $F$ are respectively the matrices expressing the infections transition and the emergence of new infected cases in the different contaminated compartments of the model. By observing our system \eqref{detr}, we can easily notice that it has two essential and particular properties. The first is the fact that it admits one, and only one, disease-free equilibrium (DFE) $\mathcal{E}^o=\left(S^o,Q^o,0,0,0,0,0\right)$, where $S^o$ and $Q^o$ are defined by \eqref{nota}. The second is the possibility of rearranging its equations and rewriting it in the modified form
\begin{equation}\label{detr3}
\dfrac{\text{d}\displaystyle{\widetilde{\mathlarger{x}}(t)}}{\text{d}t}=\widetilde{F}\left(\widetilde{\mathlarger{x}}(t)\right),
\end{equation}  
where $\displaystyle{\widetilde{\mathlarger{x}}}=\big(\hspace{-23pt}\overbrace{E,A,I}^{\text{\footnotesize{infected components}}}\hspace{-23pt},S,Q,H,R\big)$ and $\widetilde{F}=\big(\hspace{-23pt}\overbrace{F_3,F_4,F_5,}^{\text{\footnotesize{equations of }} E,A ~\text{and}~I}\hspace{-20pt}F_1,F_2,F_6,F_7\big).$ 
Needless to say, the last arrangement makes $\widetilde{\mathcal{E}}^o=\big(0,0,0,S^o,Q^o,0,0\big)$ as the unique free equilibrium point of \eqref{detr3}. As stated in \cite{van2002reproduction}, we can split the right-hand side of \eqref{detr3} in the following way: 
 \begin{equation}
\dfrac{\text{d}\displaystyle{\widetilde{\mathlarger{x}}(t)}}{\text{d}t}=\widetilde{F}\left(\widetilde{\mathlarger{x}}(t)\right)=\mathcal{F}\left(\widetilde{\mathlarger{x}}(t)\right)-\mathcal{V}\left(\widetilde{\mathlarger{x}}(t)\right),   
\end{equation}
where $\mathcal{F}$ is the appearance rate of new infections vector and  $\mathcal{V}$ is the remaining transitional terms vector represented respectively in this case by   
\begin{equation*}
\mathcal{F}=\left(\begin{matrix}
\left(\beta_1-\beta_2\dfrac{I}{b+I}\right)S\left(I+\theta A\right)\\
0 \\ 
0 \\
0 \\
0 \\
0 \\
0 
\end{matrix}\right)
~~\text{and}~~\mathcal{V}=\left( \begin{matrix}
(\mu+\sigma)E \\ 
\left(\mu+\varepsilon_A+\gamma_A+d_A\right)A-\left(1-p\right)\sigma E \\ 
\left(\mu+\varepsilon_I+\gamma_I+d_I\right)I -\sigma p E\\
\left(\beta_1-\beta_2\dfrac{I}{b+I}\right)S\left(I+\theta A\right)+(\mu+q)S -\lambda Q -\Lambda\\ 
\left(\mu+\lambda\right)Q-qS \\ 
\left(\mu+d_H+\gamma_H\right)H-\varepsilon_A A-\varepsilon_I I \\ 
\mu R-\gamma_H H-\gamma_I I-\gamma_A A
\end{matrix}\right).
\end{equation*}
The corresponding Jacobian matrices evaluated around the DFE are respectively block decomposed as follows:
\begin{equation}
D\mathcal{F}(\widetilde{\mathcal{E}}^o)=\left(\begin{tabular}{c|c}
\rule[0ex]{0pt}{1.5ex} $F$ & $ \scalebox{1.2}{\textbf{0}}_{3,4}$\\ 
\hline 
\rule[0ex]{0pt}{0ex} $\scalebox{1.2}{\textbf{0}}_{4,3}$ & $\scalebox{1.2}{\textbf{0}}_{4,4}$ \\ 
\end{tabular} \right)~~\text{and}~~D\mathcal{V}(\widetilde{\mathcal{E}}^o)= \left(\begin{tabular}{c|c}
\rule[0ex]{0pt}{0ex} $V$ & $\scalebox{1.2}{\textbf{0}}_{3,4}$ \\ 
\hline 
\rule[0ex]{0pt}{0ex} $W_3$ & $W_4$  
\end{tabular} \right),
\end{equation}
where $\scalebox{1.2}{\text{\textbf{0}}} _{n,p}$ denotes the null matrix (i.e, a matrix filled with zeros) of dimension $n\times p$ and 
\begin{align*}
F&=\begin{pmatrix}
0 & \theta\beta_1S^o & \beta_1S^o \\ 
0 & 0 & 0 \\ 
0 & 0 & 0
\end{pmatrix},\quad & V&=\begin{pmatrix}
\mu+\sigma & 0 & 0 \\ 
-(1-p)\sigma & \mu+d_A+\varepsilon_A+\gamma_A & 0 \\ 
-\sigma p & 0 & \mu+d_I+\varepsilon_I+\gamma_I
\end{pmatrix},\\
\nonumber\\
\end{align*}
\begin{align*}
 W_3&=\begin{pmatrix}
0 & \theta\beta_1S^o & 0 \\ 
0 & 0 & 0 \\ 
0 & -\varepsilon_A & -\varepsilon_I \\
0 & -\gamma_A & -\gamma_I 
\end{pmatrix}, \qquad\text{and}& W_4&=\left(\begin{tabular}{c|c}
\rule[0ex]{0pt}{0ex} \hspace*{-5pt}$\textbf{w}_1$ &\hspace*{-5pt} $\scalebox{1.2}{\textbf{0}}_{2,2}$\hspace*{-5pt} \\ 
\hline 
\rule[0ex]{0pt}{0ex} \hspace*{-5pt}$\scalebox{1.2}{\textbf{0}}_{2,2}$ & $\textbf{w}_4$
\hspace*{-5pt}\end{tabular} \right)=\begin{pmatrix}
\mu\hspace{-0.5pt} +\hspace{-0.5pt}q  & -\lambda & 0 & 0 \\ 
-q & \mu\hspace{-0.5pt} +\hspace{-0.5pt}\lambda  & 0 & 0 \\ 
0 & 0 & \mu\hspace{-0.5pt} +\hspace{-0.5pt} d_H\hspace{-1pt} +\hspace{-1pt} \gamma_H & 0 \\ 
0 & 0 & -\gamma_H & \mu
\end{pmatrix}.
\end{align*}
So, the next generation matrix $FV^{-1}$ is
\begin{equation*}
FV^{-1} =\begin{pmatrix}
\left(\dfrac{\theta\left(1-p\right)}{\mu\hspace{-1pt}+d_A\hspace{-1pt}+\hspace{-1pt}\varepsilon_A\hspace{-1pt}+\hspace{-1pt}\gamma_A\hspace{-1pt}}\hspace{-1pt}+\hspace{-1pt}\dfrac{p}{\mu+d_I+\varepsilon_I+\gamma_I}\right)\times\dfrac{\sigma\beta S^o}{\mu+\sigma} & \dfrac{\theta\beta_1S^o}{\mu\hspace{-1pt}+\hspace{-1pt}d_A\hspace{-1pt}+\hspace{-1pt}\varepsilon_A\hspace{-1pt}+\hspace{-1pt}\gamma_A} & \dfrac{\beta_1S^o}{\mu\hspace{-1pt}+d_I\hspace{-1pt}+\varepsilon_I\hspace{-1pt}+\hspace{-1pt}\gamma_I\hspace{-1pt}} \\ 
0 & 0 & 0 \\ 
0 & 0 & 0
\end{pmatrix}.
\end{equation*}
Consequently, the reproduction number associated to the Coronavirus model \eqref{detr} is given by:
\begin{equation}
\mathcal{R}_0 \triangleq\rho\left(FV^{-1}\right)=\left(\dfrac{\theta\left(1-p\right)}{\mu+d_A+\varepsilon_A+\gamma_A}+\dfrac{p}{\mu+d_I+\varepsilon_I+\gamma_I}\right)\dfrac{\sigma\beta S^o}{\mu+\sigma}.
\end{equation} 
 \subsection{Stability analysis of the disease-free steady state  \texorpdfstring{$\mathcal{E}^o$}{E0}} 
 In this subsection, we will deal with the issue of disease-free equilibrium stability in its local and global levels, and we will show the effect of $\mathcal{R}_0$ value on each of them, but before doing so, let us first present these needed notations and terminologies to follow this part without ambiguity. 
 \begin{itemize}
\item[$\bullet$] $\textup{I}_n$  is the identity matrix of order $n$. 
\item[$\bullet$] The real part of a complex number $z\in\mathbb{C}$ is written as $\mathfrak{Re}(z)$.
\item[$\bullet$] The set of eigenvalues of a square matrix $A$, is denoted by the symbol $\Sigma(A)$. 
\item[$\bullet$] For a square matrix with complex entries $M$, the greatest real part of eigenvalues is called the spectral bound, or the stability modulus, of $M$ (see \cite{kamgang2008computation}) and we write it as $\upalpha\left(M\right)\triangleq \max\limits_{\omega\in\Sigma(M)} \mathfrak{Re}(\omega)$.
\item[$\bullet$] Finally, we remind that $\rho(M)$ refers to the maximum modulus of a square matrix $M$ eigenvalues, that is, $\rho(M)\triangleq \max\limits_{\omega\in\Sigma(M)}\left|\omega \right|$. 
\end{itemize}
\subsubsection{Local stability of the disease-free steady state \texorpdfstring{$\mathcal{E}^o$}{E0}}
\begin{theo}\label{stabi}
Under the condition $\mathcal{R}_0<1 $, the disease-free equilibrium of the system \eqref{detr} is locally asymptotically stable, but it becomes unstable if $\mathcal{R}_0>1$.
\end{theo}
\begin{proof}
 To show the local stability of $\mathcal{E}^o=\left(S^o,Q^o,0,0,0,0,0\right)$ for system \eqref{detr}, we will go through that of $\widetilde{\mathcal{E}}^o=\big(0,0,0,S^o,Q^o,0,0\big)$ for system \eqref{detr3} because, as we have already mentioned in the previous subsection, \eqref{detr3} is just a slightly modified version of \eqref{detr} obtained by an arrangement of its lines. Obviously, the Jacobian matrix of system \eqref{detr3} at $\widetilde{\mathcal{E}}^o$ is given by the following block representation:
\begin{equation}
 D\widetilde{F}({\widetilde{\mathcal{E}}^o})=D\left[\mathcal{F}-\mathcal{V}\right]({\widetilde{\mathcal{E}}^o})=D\mathcal{F}({\widetilde{\mathcal{E}}^o})-D\mathcal{V}({\widetilde{\mathcal{E}}^o})=\left(
 \begin{tabular}{c|c}
\rule[0ex]{0pt}{0ex} $F-V$ & $\scalebox{1.2}{\textbf{0}}_{3,4}$ \\ 
\hline 
\rule[0ex]{0pt}{0ex} $-W_3$ & $-W_4$  
\end{tabular} \right).
\end{equation}  
So, and with the help of Theorem 5.2.10 of \cite{watkins2004fundamentals}, one can easily conclude that the eigenvalues of $D\widetilde{F}({\widetilde{\mathcal{E}}^o})$ are exactly the combined characteristic roots of $F-V$ and $-W_4$. In other terms,
\begin{equation}\label{deco}
 \Sigma\left(D\widetilde{F}({\widetilde{\mathcal{E}}^o})\right)=\Sigma(F-V)\cup\Sigma(-W_4), 
\end{equation}
with $\Sigma$ is referring to the spectrum (i.e., the set of eigenvalues). By the same argument, we can draw the following equality for the matrix $-W_4$: 
\begin{equation*}
\Sigma(-W_4)=\Sigma(-\text{\textbf{w}}_1)\cup\Sigma(-\text{\textbf{w}}_4),
\end{equation*} 
 which together with \eqref{deco} implies that 
 \begin{equation*}
 \Sigma\left(D\widetilde{F}({\widetilde{\mathcal{E}}^o})\right)=\Sigma(F-V)\cup\Sigma(-\text{w}_1)\cup\Sigma(-\text{w}_4). 
\end{equation*}
Then 
 \begin{equation}\label{decom}
 \upalpha\left(D\widetilde{F}({\widetilde{\mathcal{E}}^o})\right)=\upalpha(F-V)\vee\upalpha(-\text{\textbf{w}}_1)\vee\upalpha(-\text{\textbf{w}}_4), 
\end{equation}
where
\begin{equation*}
-\text{\textbf{w}}_1=\begin{pmatrix}
-\mu\hspace{-0.5pt} -\hspace{-0.5pt}q  & \lambda\\ 
q & -\mu\hspace{-0.5pt} -\hspace{-0.5pt} \lambda  \end{pmatrix}~~\text{and}~~-\text{\textbf{w}}_4=\begin{pmatrix}
-\mu\hspace{-0.5pt} -\hspace{-0.5pt} \gamma_H\hspace{-1pt} -\hspace{-1pt}d_H  & 0\\
\gamma_H & -\mu
\end{pmatrix}.
\end{equation*}
The matrix $-\text{\textbf{w}}_4$ is upper triangular, then its eigenvalues are exactly the diagonal entries, and since they are all negative in this case we get \begin{equation}\label{w4}
\upalpha(-\text{\textbf{w}}_4)<0.
\end{equation}
On the other hand, it is easy to check that the characteristic polynomial of $-\text{\textbf{w}}_1$ is 
\begin{equation*}
P(w)=\textup{det}\left(-\text{w}_1-w\textup{I}_2\right)=w^2+\underbrace{\left(q+2\mu+\lambda\right)}_{a_1}w+\underbrace{\mu(\mu+\lambda
+q)}_{a_2}.
\end{equation*} 
So, and by observing the positivity of $a_1$ and $a_2$, we can immediately conclude that
\begin{equation}\label{w3}
\upalpha(-\text{\textbf{w}}_1)<0.
\end{equation}
 According to Lemma $1$ of \cite{van2002reproduction}, $V$ is a non-singular M-matrix (a real matrix with nonpositive off-diagonal entries and positive spectral bound \cite{plemmons1977m}), and since $F$ is nonnegative (i.e., all of whose entries are nonnegative), we can deduce by using Varga's theorem \cite{salletbook,kamgang2008computation,varga1999matrix} that 
\begin{equation}\label{eqiv}
\upalpha\left(F-V\right)<0 \iff \rho(FV^{-1})<1. 
\end{equation}
Combining \eqref{decom}, \eqref{w4} and \eqref{w3} with \eqref{eqiv} gives the following equivalence:
\begin{equation}
 \mathcal{R}_0<1 \iff \upalpha\left(D\widetilde{F}({\widetilde{\mathcal{E}}^o})\right)<0,
\end{equation}
 which leads us thanks to the famous Lyapunov's linearisation theorem \cite[page 139]{khalil2002nonlinear}, to the fact that $\mathcal{R}_0<1$ implies the local asymptotic stability of the disease-free equilibrium.
\par Now, the only remaining point is to check the instability of the DFE when $\mathcal{R}_0>1$. This time, the task is somewhat easy, especially when we use the second version of Varga's theorem \cite{martcheva2015introduction,brauer2013mathematical} which asserts that 
\begin{equation}\label{eqiv2}
\upalpha\left(F-V\right)>0 \iff \rho(FV^{-1})>1. 
\end{equation}
The last equivalence together with \eqref{decom} and the Lyapunov's linearisation theorem allows us to say that the DFE is unstable if $\mathcal{R}_0>1.$ Hence, the theorem is proved.    
\end{proof}
\begin{rema}
The adopted method in the previous proof is a little different from what we usually see in the literature, since it shows the local stability of the disease-free state without using the Routh-Hurwitz criterion \textup{(}see\cite[page 101]{martcheva2015introduction}\textup{)} to the characteristic polynomial $\chi(w)=\textup{det}\left(DF(\mathcal{E}^o)-w\textup{I}_7\right)$, and this enabled us to sidestep lengthy and laborious calculations.
\end{rema}
\begin{rema}
The fact that $\mathcal{R}_0<1$ implies the local asymptotic stability of the disease-free equilibrium, is wrong in general, and to obtain it we should ensure in addition the stability condition of the matrix $-W_4$ \textup{(}i.e., $\upalpha\left(-W_4\right)<0$\textup{)}. This condition was formulated implicitly in \cite{van2002reproduction} by the assumption \textup{(A5)}, and it plays an important role in the completion of the proof concerning the DFE local asymptotic stability. But despite this, we note that there is not any reason to use it in the demonstration of instability in the case of $\mathcal{R}_0>1$.  
\end{rema}   
\subsubsection{Global stability of the disease-free steady state \texorpdfstring{$\mathcal{E}^o$}{E0}}
 In the following, we seek to proof the global asymptotic stability of $\mathcal{E}^0$ on the biologically feasible region $\mathcal{S}$ (see Theorem \ref{thinv}) by applying an approach adapted from \cite{chavez2002computation}. But before this, let us first follow the notation of \cite{brauer2013mathematical} and write system \eqref{detr3} in the form:
 \begin{equation}\label{detr4}
\begin{cases}
\dfrac{\text{d}\displaystyle{\widetilde{\mathlarger{x}}_1(t)}}{\text{d}t}=\widetilde{F}_1\left(\displaystyle{\widetilde{\mathlarger{x}}}_1(t),\displaystyle{\widetilde{\mathlarger{x}}}_2(t)\right)=\textup{\textbf{M}}\hspace{1pt}\displaystyle{\widetilde{\mathlarger{x}}_1(t)}-\widehat{f}\left(\displaystyle{\widetilde{\mathlarger{x}}_1(t)},\displaystyle{\widetilde{\mathlarger{x}}_2(t)}\right),\\[5pt]
\dfrac{\text{d}\displaystyle{\widetilde{\mathlarger{x}}_2(t)}}{\text{d}t}=\widetilde{F}_2\left(\displaystyle{\widetilde{\mathlarger{x}}}_1(t),\displaystyle{\widetilde{\mathlarger{x}}}_2(t)\right),
\end{cases}
\end{equation}    \hspace{-1pt}
where $\textup{\textbf{M}}=F-V$\textbf{,} $\displaystyle{\widetilde{\mathlarger{x}}}\hspace{-1pt}=\hspace{-1pt}\big(\widetilde{\mathlarger{x}}_1,\widetilde{\mathlarger{x}}_2\big)\hspace{-1pt}=\hspace{-1pt}\big(\hspace{-30pt}\overbrace{E,A,I,}^{\text{\footnotesize{infected components}}~\widetilde{\mathlarger{x}}_1}\hspace{-58pt},\underbrace{S,Q,H,R}_{\text{\footnotesize{uninfected components}}~\widetilde{\mathlarger{x}}_2}\hspace{-28pt}\big)$\textbf{,} $\widetilde{F}\hspace{-1pt}=\hspace{-1pt}\big(\widetilde{F}_1,\widetilde{F}_2\big)\hspace{-1pt}=\hspace{-1pt}\big(\hspace{-23pt}\overbrace{F_3,F_4,F_5,}^{\text{\footnotesize{equations of }} E,A ~\text{and}~I}\hspace{-40pt}\underbrace{F_1,F_2,F_6,F_7}_{\text{\footnotesize{equations of}}~S,Q,H~\text{and}~R}\hspace{-20pt}\big)$ and $$\widehat{f}\left(\displaystyle{\widetilde{\mathlarger{x}}_1},\displaystyle{\widetilde{\mathlarger{x}}_2}\right)=\begin{pmatrix}
\left(\beta_1S^o-\left(\beta_1-\beta_2\dfrac{I}{b+I}\right)S\right)\left(I+\theta A\right)\\
0\\
0
\end{pmatrix}.$$
\begin{lemm}\label{gener}
Let $\mathcal{M}\subset \mathbb{R}_+^7$, if the following conditions are satisfied:
\begin{enumerate}[label=$(\textup{\alph*})$]
\item The set $\mathcal{M}$ is a positively invariant region with respect to system \eqref{detr}.\label{ca}
\item For the disease-free system $\dfrac{\textup{d}\displaystyle{\widetilde{\mathlarger{x}}_2(t)}}{\textup{d}t}=\widetilde{F}_2\left(0,\displaystyle{\widetilde{\mathlarger{x}}}_2(t)\right),$ the equilibrium ${\widetilde{\mathcal{E}}^o}=\big(0,0,0,S^o,Q^o,0,0\big)$  is globally asymptotically stable.\label{cb}
\item $\textup{\textbf{M}}$ is  a stable Metzler matrix  \textup{(}i.e.,$-\textup{\textbf{M}}$ is a non-singular M-matrix \cite{kamgang2008computation}\textup{)}.\label{cc} 
\item For all $\displaystyle{\mathlarger{x}\in\mathcal{M}},$ $\widehat{f}\left(\displaystyle{\widetilde{\mathlarger{x}}_1},\displaystyle{\widetilde{\mathlarger{x}}_2}\right)\geqslant 0.$\label{cd}
\end{enumerate}
Then, the disease-free equilibrium $\mathcal{E}^o:=\left(S^o,Q^o,0,0,0,0,0\right)$ is globally asymptotically stable on $\mathcal{M}$.
\end{lemm}
\begin{proof}
Suppose that the conditions of this lemma hold. Let $\displaystyle{\mathlarger{x}}(0)\in \mathcal{M}$ and $\left(\displaystyle{\widetilde{\mathlarger{x}}}_1(t),\displaystyle{\widetilde{\mathlarger{x}}}_2(t)\right)$ be respectively a given initial value and the solution of \eqref{detr4} that starts from it. By using the famous variation-of-constant formula \cite{ding2007approaches} for the first equation of system \eqref{detr4}, we get 
$$\displaystyle{\widetilde{\mathlarger{x}}}_1(t)=e^{t\textup{\textbf{M}}}\displaystyle{\widetilde{\mathlarger{x}}}_1(0)-\int_0^te^{(t-s)\textup{\textbf{M}}}\widehat{f}\left(\displaystyle{\widetilde{\mathlarger{x}}_1(s)},\displaystyle{\widetilde{\mathlarger{x}}_2(s)}\right)\textup{d}s.$$ 
From the third condition of this lemma, $\textup{\textbf{M}}$ is a stable Metzler matrix, then it can be expressed in the form $\textup{\textbf{M}}=\widehat{\textup{B}}-\displaystyle{\mathlarger{\kappa}}\textup{I}_7$ where $\widehat{\text{B}}\geqslant 0$ is a nonnegative matrix (i.e., all of whose entries are nonnegative)   and $\displaystyle{\mathlarger{\kappa}}>\rho(\textup{\textbf{M}})$ (see \cite{plemmons1977m}). Since $\widehat{\textup{B}}\geqslant 0$, we have $e^{\xi \widehat{\text{B}}}\geqslant 0$ for all $\xi\geqslant 0$, then $e^{\xi\textup{\textbf{M}}}\hspace{-1pt}=e^{\xi \widehat{\text{B}}}e^{-\xi \displaystyle{\mathlarger{\kappa}}\textup{I}_7}\hspace{-1pt}=e^{-\xi \displaystyle{\mathlarger{\kappa}}}e^{\xi \widehat{\text{B}}}\hspace{-1pt}\geqslant 0$. This together with the invariance of the region $\mathcal{M}$, and the fact that $\widehat{f}$ is a nonnegative function on this region (assumptions \ref{ca},\ref{cd}) implies that $$\displaystyle{\widetilde{\mathlarger{x}}}_1(t)=e^{t\textup{\textbf{M}}}\displaystyle{\widetilde{\mathlarger{x}}}_1(0)-\int_0^te^{(t-s)\textup{\textbf{M}}}\widehat{f}\left(\displaystyle{\widetilde{\mathlarger{x}}_1(s)},\displaystyle{\widetilde{\mathlarger{x}}_2(s)}\right)\textup{d}s\leqslant e^{t\textup{\textbf{M}}}\displaystyle{\widetilde{\mathlarger{x}}}_1(0).$$
The matrix $\textup{\textbf{M}}$ is stable (i.e, $\upalpha\left(\textbf{M}\right)<0$), so $\lim\limits_{t\to \infty} e^{t\textup{\textbf{M}}}\displaystyle{\widetilde{\mathlarger{x}}}_1(0)=0$. Combining this fact with the positivity of  the solution $\displaystyle{\widetilde{\mathlarger{x}}}(t)$ (Theorem \ref{thm1*}) gives $\lim\limits_{t\to \infty} \displaystyle{\widetilde{\mathlarger{x}}}_1(t)=0.$ By using the condition \ref{cb} and  an argument similar to the proof of Theorem 1 in \cite{chavez2002computation} (see also \cite[page 94]{castillo2002mathematical} for more details) we can conclude immediately that $\mathcal{E}^o:=\left(S^o,Q^o,0,0,0,0,0\right)$ is globally asymptotically  stable on $\mathcal{M}$, which proves the lemma.   
\end{proof}
\begin{rema}
Clearly, if we take $\mathcal{M}=\mathbb{R}_+^7$, the condition \ref{ca} in the last lemma becomes unnecessary because $\mathbb{R}_+^7$ is already positively invariant, and in this case, we recover exactly the statement of Theorem 9.2 of \cite{brauer2013mathematical}, which allows us to say that Lemma \ref{gener} is a generalization of this theorem. 
\end{rema}
\begin{rema}
According to the previously mentioned Varga's theorem \big{(}see \eqref{eqiv}\big{)}, the stability of the matrix $\textup{\textbf{M}}:=F-V$ \big{(}i.e., $\upalpha(F-V)<0$\big{)} is equivalent to $\mathcal{R}_0:=\rho (FV^{-1})<1$, so the last assumption of Theorem 9.2 in \cite{brauer2013mathematical} \big{(}$\mathcal{R}_0<1$\big{)} is redundant and can therefore be dropped.
\end{rema}
\begin{theo}\label{theodfe}
 If $\mathcal{R}_0<1$, then the disease-free equilibrium $\mathcal{E}^o$ is globally asymptotically stable on $\mathcal{S}$. 
\end{theo}
\begin{proof}
To demonstrate this theorem, we will use Lemma \ref{gener}, in other words, we are going to show the global stability of $\mathcal{E}^o$ by proving that all the conditions \ref{ca},\ref{cb},\ref{cc} and \ref{cd} are satisfied. As supposed in the statement of the theorem, let $\mathcal{R}_0<1$. 
From the equations of \eqref{detr}, we observe that the disease-free system is expressed by
 \begin{equation}\label{dfsys}
 \dfrac{\textup{d}\displaystyle{\widetilde{\mathlarger{x}}_2(t)}}{\textup{d}t}=\widetilde{F}_2\left(0,\displaystyle{\widetilde{\mathlarger{x}}}_2(t)\right)=-W_4\hspace{2pt}\widetilde{\mathlarger{x}}_2(t)+\begin{pmatrix}
\Lambda\\
\scalebox{1.2}{\textbf{0}}_{3,1}
\end{pmatrix}.
\end{equation}
Obviously, \eqref{dfsys} is a linear differential system with $\upalpha(-W_4)<0$ (see the proof of Theorem \ref{stabi}), so the fixed point  $\widetilde{\mathcal{E}}^o$ is a globally asymptotic stable equilibrium of \eqref{dfsys} \cite[Theorem  4.5]{khalil2002nonlinear}. Thus, the condition \ref{cb} is satisfied. On the other hand, for any $\displaystyle{\mathlarger{x}\in\mathcal{S}}=\left\{\left(S,Q,E,A,I,H,R\right)\in\mathbb{R}_+^7\mid S\leqslant S^o~~\text{and}~~ Q\leqslant Q^o\right\},$ we have \vspace*{-10pt} $$\left(\beta_1S^o-\left(\beta_1-\beta_2\dfrac{I}{b+I}\right)S\right)\left(I+\theta A\right)=\Big(\beta_1\overbrace{\left(S^o-S\right)}^{\geqslant 0}+\beta_2\dfrac{I}{b+I}S\Big)\left(I+\theta A\right)\geqslant 0.$$   
Hence, $\widehat{f}\left(\displaystyle{\widetilde{\mathlarger{x}}_1},\displaystyle{\widetilde{\mathlarger{x}}_2}\right)\geqslant 0,$ for all $\displaystyle{\mathlarger{x}\in\mathcal{S}},$ which means that the condition \ref{cd} is also satisfied. Apparently, $\textup{\textbf{M}}=F-V$ is a Metzler matrix, and since $\mathcal{R}_0<1$, one can say that it is also stable \big{(}see \eqref{eqiv}\big{)}, so the assumption \ref{cc} holds. According to Theorem \ref{thinv}, the set $\mathcal{S}$ is an invariant region for \eqref{detr}. Therefore, \ref{ca} is satisfied and the conclusion of the theorem follows immediately from Lemma \ref{gener}.   
\end{proof}
\subsection{Existence of endemic equilibrium and uniform persistence}
The endemic equilibria of the proposed  COVID-19 model are obtained by solving the system  $F(x)=0$ on the positive orthant $\mathbb{R}_+^7$, where $F$ is the vector-valued function given by \eqref{F}. In other words, a point $\mathcal{E}^{\star}=\big(S^{\star},Q^{\star},E^{\star},A^{\star},I^{\star},H^{\star},R^{\star}\big)$ is an endemic equilibrium of \eqref{detr} if and only if all its components are strictly greater than zero  and satisfy the following equalities:
\begin{align*}
 S^{\star}=\dfrac{\mu+\lambda}{q}\left(c-mc^{\prime}I^{\star}\right), ~~\quad  Q^{\star}&=c-mc^{\prime}I^{\star}, ~~\quad E^{\star}=c^{\prime}I^{\star}, ~~\quad  A^{\star}=c^{\prime}m^{\prime}I^{\star}, ~~\quad  H^{\star}=\dfrac{\left(\varepsilon_I
 +\varepsilon_A c^{\prime}m^{\prime}\right)I^{\star}}{\mu+d_H+\gamma_H},\\
R^{\star}&=\dfrac{1}{\mu}\left(\gamma_Ac^{\prime}m^{\prime}+\gamma_I+\dfrac{\gamma_H\left(\varepsilon_I+\varepsilon_Ac^{\prime}m^{\prime}\right)}{\mu+d_H+\gamma_H}\right)I^{\star},  
\end{align*}
where \begin{align*}
c&=\Lambda\left((\mu+q)\times\dfrac{\mu+\lambda}{q}-\lambda\right)^{-1},&  m&=(\mu+\sigma)\left((\mu+q)\times\dfrac{\mu+\lambda}{q}-\lambda\right)^{-1},\\
c^{\prime}&=\dfrac{\mu+\varepsilon_I+\gamma_I+d_I}{\sigma p},& m^{\prime}&=\dfrac{\sigma(1-p)}{\mu+\varepsilon_A+\gamma_A+d_A},
\end{align*}
and $I^{\star}$ is defined as the positive solution of this  equation derived from the fact that $F_3(\mathcal{E}^{\star})=0$:
\begin{align}\label{equation}
\left(\beta_1-\beta_2\dfrac{I^{\star}}{b+I^{\star}}\right) S^{\star}\left(I^{\star}+\theta A^{\star}\right)=(\mu+\sigma)E^{\star}.
\end{align}
Substituting the aforementioned expressions of $S^{\star}$, $A^{\star}$ and $E^{\star}$ into \eqref{equation} leads us, and after some simplifications, to the following equation:
$$\mathfrak{A}{I^{\star}}^{2}+\mathfrak{B}I^{\star}+\mathfrak{C}=0,$$
where
$$\mathfrak{A}=mc^{\prime}(\beta_1-\beta_2)>0,~
\mathfrak{B}=\beta_1bmc^{\prime}+\beta_1c\left(\dfrac{1}{\mathcal{R}_0}-1\right)+\beta_2c,
~\text{and}~\mathfrak{C}=\dfrac{q(\sigma+\mu c^{\prime}b)}{(\mu+\lambda)(1+\theta m^{\prime}c^{\prime})}\left(1-\mathcal{R}_0\right).$$
If $\mathcal{R}_0>1$, then $\mathfrak{C}<0$, and since $\mathfrak{A}>0$, the equation \eqref{equation} admits a unique strictly positive solution. Therefore \eqref{detr} admits in turn  a unique endemic equilibrium. On the other hand, when  $\mathcal{R}_0 \leqslant 1$, we get $\mathfrak{A}>0$,$\mathfrak{B}>0$ and $\mathfrak{C}\geqslant 0$, so the equation \eqref{equation}  does not have any endemic equilibrium in this case, which implies the non-existence of an endemic equilibrium for \eqref{equation}. 
\par Hence, we can summarize the above discussions in the following theorem.
\begin{theo}
The system \eqref{detr} has a unique endemic equilibrium in the case of $\mathcal{R}_0>1$, but when $\mathcal{R}_0\leqslant 1$ such an equilibrium can never exist. 
\end{theo} 
After having studied the existence of the endemic equilibrium, we will now explore  the uniform persistence of system \eqref{detr}.
\begin{defi}[Uniform persistence \cite{bai2018global}]
Let $\Delta\subset \mathbb{R }_+^7$. We say that the system \eqref{detr} is uniformly persistent in $\Delta$ if there is a positive constant $\varrho>0$ such that for any initial value $\big(S_0,Q_0,E_0,A_0,I_0,H_0,R_0\big)\in \Delta$, the solution of \eqref{detr} starting from this value satisfies \begin{align*}
\liminf_{t\to\infty} S(t)&>\varrho, &\liminf_{t\to\infty} Q(t)&>\varrho, &\liminf_{t\to\infty} E(t)&>\varrho,\\
\liminf_{t\to\infty} A(t)&>\varrho, &\liminf_{t\to \infty} I(t)&>\varrho, &\liminf_{t\to\infty} H(t)&>\varrho,\\
 &  &\liminf_{t\to \infty} R(t)&>\varrho.
\end{align*}
\end{defi}
\begin{theo}\label{persist}
If $\mathcal{R}_0>1$, then the system \eqref{detr} is uniformly persistent in $\mathcal{S}$.
\end{theo}
\begin{proof}
As assumed in the statement of the theorem, let $\mathcal{R}_0>1$. Plainly, the disease-free equilibrium $\mathcal{E}^o$ is on the boundary  of $\mathbb{R}_+^7$, and as already mentioned in Theorem \ref{stabi}, this  equilibrium is unstable when $\mathcal{R}_0>1$. Combining this fact with the dissipativity of system \eqref{detr} (Theorem \ref{inva}) leads directly by virtue of Theorem 4.3 in \cite{freedman1994uniform} to the uniform persistence of this system, which is the desired conclusion. 
\end{proof}
\par This result is interpreted by saying that when the basic  reproductive number $\mathcal{R}_0$ is strictly greater than one, then all the individuals appearing in \eqref{detr} and especially the infected ones, will stay above a certain positive threshold, which means that the Coronavirus disease will persist in the population in this case.

\section{Analysis of the stochastic Coronavirus model}\label{sec2}
The intent of this section is to deal with the perturbed version of the COVID-19 model expressed by the stochastic differential system \eqref{systo}. For simplicity of notation, we write from now on the initial-value problem associated with  \eqref{systo} in the following form: \begin{equation}\label{systo2}
\begin{cases}
\overbrace{\text{d}X(t)}^{7\times 1}=\overbrace{F\left(X(t)\right)\mathrm{d}t}^{7\times 1}+\overbrace{G\left(X(t)\right)}^{7\times 7}\overbrace{\text{d}B(t)}^{7\times 1},\\
X(0)=X_0\in \mathbb{R}_+^7\quad \text{almost surely}. 
\end{cases}
\end{equation}
Here, the function $F$  is the same as in \eqref{F}, with 
\begin{align}
X(t)&=\big(X_i(t)\big)_{1\leqslant i\leqslant 7}=\big(S(t),Q(t),E(t),A(t),I(t),H(t),R(t) \big),\nonumber\\[5pt] 
G\left(X(t)\right)&=\mathrm{diag}\left(\left(\sigma_iX_i(t)\right)_{1\leqslant i\leqslant 7}\right)= \begin{pmatrix}
 \sigma_1 X_1 & \cdots & 0 \\ 
 \vdots & \ddots & \vdots \\ 
 0 & \cdots & \sigma_7 X_7
\end{pmatrix},\nonumber
\end{align}
and  $$\hspace*{-4.2cm}\quad B(t)=\left(B_1(t),B_2(t),\cdots,B_7(t) \right) \textbf{.}$$
For the sake of simplicity, we will denote the temporary mean $\displaystyle{\dfrac{1}{t}\int_0^t\upvarphi(s)~\text{d}s}$ of a continuous function $\upvarphi$ by  $\langle\upvarphi(t)\rangle$ . 
\subsection{Existence and uniqueness of the global positive solution}
 To explore the dynamical properties of a population system, the first concern is to know if it admits a solution, and if this solution is unique, positive and global (in time). In what follows, we will give some conditions under which these four points above \big(existence, uniqueness, positivity and globality\big), are verified for the system \eqref{systo}. 
\begin{theo}\label{thm1}
For any initial value $X_0\in \mathbb{R}_+^7$, there is a unique solution $X(t)$ to the system \eqref{systo2} on $t\geqslant 0$, and it will remain in $\mathbb{R}_+^7$ with probability one, which means that, if $\big(S(0),Q(0),E(0),A(0),I(0),H(0), R(0)\big)$ is in $\mathbb{R}_+^7$, then $\big(S(t),Q(t),E(t),A(t),I(t),H(t), R(t)\big)\in \mathbb{R}_+^7$ for all $t\geqslant 0$ almost surely (a.s. for short).
\end{theo}
\begin{proof}
 In the system \eqref{systo2}, the coefficients $F$ and $G$ are continuously differentiables on their domains of definition, so they satisfy the local Lipschitz condition, and for this reason, there exists for any given initial value $X_0\in\mathbb{R}_+^7$, a unique maximal local solution $ X(t)$ on $t\in \left[0,\tau_{e}\right),$ where $\tau_{e}$ is the explosion time \cite{mao2007stochastic}. At this point, our goal will be to demonstrate that this solution is global, that is $\tau_{e}=\infty$ a.s. 
\par\noindent To this purpose, let $k_0\in\mathbb{N}$ be very large such that $X(0)\in [k_0^{-1},k_0]$, and define for each integer $k\geqslant k_0$ the stopping time $\tau_k$ as follows:
 \begin{align}
\tau_k &=\inf\left\lbrace t\in\left[0,\tau_e\right)\mid \left(\exists i\in\left\{1,\cdots,7\right\}\right): \: X_i(t)\not\in \left(\dfrac{1}{k},k\right) \right\rbrace=\inf\left\lbrace t\in\left[0,\tau_e\right)\mid X(t)\not\in \left(\dfrac{1}{k},k\right)^7\right\rbrace\nonumber\\ 
&=\inf\left\lbrace t\in\left[0,\tau_e\right)\mid \min_{1\leqslant i\leqslant 7} X_i(t)\leqslant \dfrac{1}{k}\: \mathrm{or} \: \max_{1\leqslant i\leqslant 7} X_i(t)\geqslant k \right\rbrace.\label{ligne}
\end{align}
Set $\tau_\infty=\lim\limits_{k\to \infty} \tau_k$, clearly, $\left(\tau_k\right)_{k \geqslant k_0}$ is increasing; hence, $\lim\limits_{k\to \infty} \tau_k=\sup\limits_{k\geqslant k_0} \tau_k$, and according to Lemma 2.11 of \cite{karatzas1998brownian} $\sup\limits_{k\geqslant k_0} \tau_k$ is a stopping time, then so is $\tau_\infty$.\\
By adopting the convention $\inf \emptyset=\infty$ for the rest of this paper, we can easily affirm that $\tau_\infty \leqslant\tau_e$ a.s., and this because
\begin{itemize}
 \item[-] if for all $k\geqslant k_0$ $$\left\lbrace t\in\left[0,\tau_e\right)\mid X(t)\not\in \left(\frac{1}{k},k\right)^7\right\rbrace\neq \emptyset, $$ then obviously $\tau_k\leqslant \tau_e$ for each $k\geqslant k_0$, which implies that $\lim\limits_{k\to \infty} \tau_k=\tau_\infty \leqslant\tau_e.$
 \item[-] On the other hand, if there is an integer $k_1\geqslant k_0$ such that 
\begin{equation}\label{m}
\left\lbrace t\in\left[0,\tau_e\right)\mid X(t)\not\in \left(\frac{1}{k_1},k_1\right)^7\right\rbrace=\emptyset,
\end{equation} 
then $\tau_{k_1}=\infty$, and as $\left(\tau_k\right)_{k \geqslant k_0}$ is increasing, we get $\tau_k=\infty$ for all $k\geqslant k_1$, which yields $~\tau_\infty=\infty$.
At the same time, it follows from \eqref{m} that the solution $X(t)$ is bounded and belongs to $\left(\frac{1}{k_1},k_1\right)^7$ for every $t\in\left[0,\tau_e\right)$, namely
\begin{equation}\label{bor}
X(t)\in \left(\frac{1}{k_1},k_1\right)^7 \text{a.s.},\quad \text{for all} ~t \in\left[0,\tau_e\right).
\end{equation}
Therefore, $\tau_e=\infty$ and so $\tau_\infty\leqslant\tau_e.$
\end{itemize}
Hence, $\tau_e=\infty$ a.s. will follow directly if we show that $\tau_\infty=\infty$ a.s., and that is exactly what we are going to do to finish the proof.
\par Assume that $\tau_\infty=\infty$ a.s. is untrue, then there exists a positive constant $T$ such that $\mathbb{P}\left(\tau_\infty\leqslant T\right)>0$.\\
Therefore, there exists an 
$\epsilon>0$ for which
\begin{equation}\label{eq1}
\mathbb{P}\left(\tau_k\leqslant T\right)> \epsilon~~\text{for all}~k\geqslant k_0.
\end{equation}
 Consider the $\mathcal{C}^2$function $V$ defined for $x=\left(x_1,\cdots,x_7\right)\in \mathbb{R}_+^7$ by
 $$V\left(x\right)=\left[x_1-a-a\ln\left(\dfrac{x_1}{a}\right) \right]+\displaystyle{\sum\limits_{i=2}^7 \left(x_i-1-\ln\left(x_i\right)\right)},$$   
where $a$ is a positive constant to be chosen suitably later. The nonnegativity of this function  can be deduced from the following inequality: $ x-1-\ln(x)>0,~\forall x>0$.\\
Applying the multi-dimensional It\^{o}'s formula \big(see \cite[page 36]{mao2007stochastic}\big) to $V(X(t))$, we obtain for all $k \geqslant k_0$ and $t\in\left[0,\tau_k \right)$
\begin{align*}
\text{d}V\big(S(t),Q(t),E(t),A(t),I(t),H(t),R(t)\big)&=\mathcal{L}V\big(S(t),Q(t),E(t),A(t),I(t),H(t),R(t)\big)~\text{d}t+\left(S(t)-a\right)\sigma_1~\mathrm{d}B_1(t)\\
&\quad+\left(Q(t)-1\right)\sigma_2~\mathrm{d}B_2(t)+\left(E(t)-1\right)\sigma_3~\mathrm{d}B_3(t)
+\left(A(t)-1\right)\sigma_4~\mathrm{d}B_4(t)\\
&\quad+\left(I(t)-1\right)\sigma_5~\mathrm{d}B_5(t)
+\left(H(t)-1\right)\sigma_6~\mathrm{d}B_6(t)
+\left(R(t)-1\right)\sigma_7~\mathrm{d}B_7(t),
\end{align*} 
where $\mathcal{L}V:\mathbb{R}_{+}^7\to\mathbb{R}$ is defined by
\begin{align*}
\mathcal{L}V\big(S,Q,E,A,I,H,R\big)&=\left(1-\dfrac{a}{S} \right)\times\left[\Lambda-\left(\beta_1-\beta_2\dfrac{I}{b+I}\right)S\left(I+\theta A\right)+\lambda Q -(\mu+q)S \right]\\
&\quad +\left(1-\dfrac{1}{Q}\right)\times\left[qS-\left(\mu+\lambda\right)Q \right]+\left(1-\dfrac{1}{E}\right)\times \left[\left(\beta_1-\beta_2\dfrac{I}{b+I}\right)S\left(I+\theta A\right)-(\mu+\sigma)E\right]\\
&\quad+\left(1-\dfrac{1}{A}\right)\times\left[\left(1-p\right)\sigma E-\left(\mu+\varepsilon_A+\gamma_A+d_A\right)A\right]\hspace{-0.5pt}+\hspace{-0.5pt}\left(1\hspace{-0.3pt}-\hspace{-0.3pt}\dfrac{1}{I}\right)\hspace{-0.5pt}\times \hspace{-0.5pt}\left[\sigma p E\hspace{-1pt}-\hspace{-1pt}\left(\mu\hspace{-1pt}+\hspace{-1pt}\varepsilon_I\hspace{-1pt}+\hspace{-1pt}\gamma_I\hspace{-1pt}+\hspace{-1pt}d_I\right)I\right]\\
&\quad+\left(1-\dfrac{1}{H}\right)\times\left[\varepsilon_I I+\varepsilon_A A-\left(\mu+d_H+\gamma_H\right)H\right]+\left(1-\dfrac{1}{R}\right)\times\left[\gamma_H H+\gamma_I I+\gamma_A A-\mu R\right]\\
&\quad+\frac{1}{2}\left[a\sigma_1^2+\sigma_2^2+\sigma_3^2+\sigma_4^2+\sigma_5^2+\sigma_6^2+\sigma_7^2\right]\\
&=\Lambda-\mu\left(S+Q+E+A+I+H+R\right)-d_A A-d_I I-d_H H\\
&\quad+\left[-\dfrac{a \Lambda}{S}+a\beta_1\left(I+\theta A\right)-a\beta_2\dfrac{I\left(I+\theta A\right)}{b+I}-a\dfrac{\lambda Q}{S}+a(\mu+q) \right]+\left[-q\frac{S}{Q}+\left(\mu+\lambda\right)\right]\\
&\quad+\left[-\left(\beta_1-\beta_2\dfrac{I}{b+I}\right)\dfrac{S}{E}\left(I+\theta A\right)+(\mu+\sigma)\right]+\left[-\left(1-p\right)\sigma \frac{E}{A}+\left(\mu+\varepsilon_A+\gamma_A+d_A\right)\right]\\
&\quad+\left[-\sigma p \frac{E}{I}+\left(\mu+\varepsilon_I+\gamma_I+d_I\right)\right]+\left[-\varepsilon_I\frac{I}{H}-\varepsilon_A\frac{A}{H}+(\mu+\gamma_H+d_H)\right]\\
&\quad+\left[-\gamma_H\frac{H}{R}-\gamma_I\frac{I}{R}-\gamma_A\frac{A}{R}+\mu\right]+\frac{1}{2}\left[a\sigma_1^2+\sigma_2^2+\sigma_3^2+\sigma_4^2+\sigma_5^2+\sigma_6^2+\sigma_7^2\right]\\
&\leqslant \Big[\Lambda+6\mu+\lambda+\sigma+\varepsilon_A+\gamma_A+d_A+\varepsilon_I+\gamma_I+d_I+d_H+\gamma_H+a(\mu+q)\\
&\quad+\frac{1}{2}\left(a\sigma_1^2+\sigma_2^2+\sigma_3^2+\sigma_4^2+\sigma_5^2+\sigma_6^2+\sigma_7^2 \right)\Big]-\mu(S+Q+E+R)-\theta\beta_1\left(\dfrac{\mu+d_A}{\theta\beta_1}-a\right)A\\
&\quad-\beta_1\left(\dfrac{\mu+d_I}{\beta_1}-a\right)I.
\end{align*}
By choosing $a=\dfrac{1}{2}\min\left\{\dfrac{\mu+d_A}{\theta \beta_1},\dfrac{\mu+d_I}{\beta_1}\right\}$, the coefficients of $A$ and $I$ will be negatives, therefore
$$\mathcal{L}V\left(S,Q,E,A,I,H,R\right)\leqslant\Lambda+6\mu+\lambda+\sigma+\varepsilon_A+\gamma_A+d_A+\varepsilon_I\hspace{-1pt}+\gamma_I+d_I+d_H+\gamma_H+a(q+\mu)\hspace{-1pt}+\hspace{-1pt}\dfrac{1}{2}\left(a\sigma_1^2+\sum\limits_{i=2}^{7}\sigma_i^2\right)\hspace{-1pt}\triangleq \hspace{-1pt}\mathcal{K}.$$
Hence, we get for all $k \geqslant k_0$ and $t\in\left[0,\tau_k \right)$
\begin{align*}
\text{d}V\big(S(t),Q(t),E(t),A(t),I(t),H(t),R(t)\big)&\leqslant \mathcal{K}~\text{d}t\hspace{-1pt}+\hspace{-1pt}\left(S(t)-a\right)\sigma_1~\mathrm{d}B_1(t)\hspace{-1pt}+\hspace{-1pt}\left(Q(t)-1\right)\sigma_2~\mathrm{d}B_2(t)
\hspace{-1pt}+\hspace{-1pt}\left(E(t)-1\right)\sigma_3~\mathrm{d}B_3(t)\\
&\quad+\left(A(t)-1\right)\sigma_4~\mathrm{d}B_4(t)
\hspace{-1pt}+\hspace{-1pt}\left(I(t)-1\right)\sigma_5~\mathrm{d}B_5(t)
\hspace{-1pt}+\hspace{-1pt}\left(H(t)-1\right)\sigma_6~\mathrm{d}B_6(t)\\
&\quad+\left(R(t)-1\right)\sigma_7~\mathrm{d}B_7(t).
\end{align*}
 Integrating from $0$ to $\tau_k\wedge T$ and then taking the expectation on both sides of the above inequality leads to
 \begin{equation}\label{eq2}
 \mathbb{E}\left[V(X(T\wedge\tau_k)\right]\leqslant V(X(0))+\mathcal{K}\mathbb{E}\left[\tau_k\wedge T\right]\leqslant  V(X(0))+\mathcal{K}T.
 \end{equation}
 We have $V(x)\geq 0$ for all $x>0$, then
\begin{equation}\label{eq3}
\mathbb{E}\left[V(X(T\wedge\tau_k)\right]=\mathbb{E}\left[V(X(T\wedge\tau_k)\times \mathds{1}_{\lbrace\tau_k \leqslant T \rbrace}\right]+\mathbb{E}\left[V(X(t\wedge\tau_k)\times \mathds{1}_{\lbrace \tau_k>T \rbrace}\right]\geqslant \mathbb{E}\left[V(X(\tau_k)\times \mathds{1}_{\lbrace\tau_k \leqslant T \rbrace}\right],
\end{equation}
where $\mathds{1}_A$ denotes the indicator function of a measurable set  $A\in\mathfrak{F}$. Note that for every $\omega\in\left\lbrace\omega\in\Omega\mid \tau_k(\omega)\leqslant T \right\rbrace$, there is some component of $V\left(X(\tau_k)\right)$ equals to $k$ or $\dfrac{1}{k}$ so
$$V\left(X(\tau_k)\right)\geqslant\left(k-a-a\ln\left(\frac{k}{a}\right)\right)\wedge\left(\frac{1}{k}-a-a\ln\left(\frac{1}{ka}\right)\right)\wedge\left(k-1-\ln\left(k\right)\right)\wedge\left(\frac{1}{k}-1-\ln\left(\frac{1}{k}\right)\right).$$
Therefore
\begin{align}\label{eq4} 
\mathbb{E}\left[V(X(\tau_k)\times \mathds{1}_{\lbrace\tau_k \leqslant T \rbrace}\right]\geqslant&\hspace{2pt} \mathbb{P}\left(\tau_k\leqslant T\right)\left(k-a-a\ln\left(\frac{k}{a}\right)\right)\wedge\left(\frac{1}{k}-a-a\ln\left(\frac{1}{ka}\right)\right)\wedge\left(k-1-\ln\left(k\right)\right)\nonumber\\
&\wedge\left(\frac{1}{k}-1-\ln\left(\frac{1}{k}\right)\right).
\end{align}
Combining \eqref{eq2}, \eqref{eq3} and \eqref{eq4} with \eqref{eq1}, we conclude that
$$V\left(X(0)\right)+\mathcal{K}T\geqslant \mathlarger\varepsilon\left(k-a-a\ln\left(\frac{k}{a}\right)\right)\wedge\left(\frac{1}{k}-a-a\ln\left(\frac{1}{ka}\right)\right)\wedge\left(k-1-\ln\left(k\right)\right)\wedge\left(\frac{1}{k}-1-\ln\left(\frac{1}{k}\right)\right).$$
Letting $k\to \infty$  leads to the contradiction $V\left(X(0)\right)+\mathcal{K}T=\infty$, which completes the proof.   
\end{proof}
\subsection{Stochastically ultimate boundedness and permanence}
After having demonstrated the positivity and the globality of our system's solution, it is now time to discuss in more detail how it behaves in the positive cone $\mathbb{R}^7_+$. In the following, we are going to define the notions of stochastically ultimate boundedness and permanence, and we will subsequently show that the solution of system  \eqref{systo} verifies these properties.    
\begin{defi}[Stochastically ultimate boundedness \cite{li2009population}]\label{sub}
We say that the system \eqref{systo} is stochastically ultimately bounded, or ultimately bounded in probability (see \cite[page 395]{mao2007stochastic}), if to every $\varepsilon>0$, there corresponds a $\delta_\varepsilon>0$ such that $$\limsup\limits_{t\to\infty}\mathbb{P}\left(\left\|X\left(t,X_0\right)\right\|>\delta_\varepsilon\right)\leqslant \varepsilon~~\text{for every}~X_0\in\mathbb{R}_+^7,$$
where $X\left(t,X_0\right)$ here denotes  the solution of \eqref{systo2} which satisfies the initial condition $X(0)=X_0$.
\end{defi}
\begin{rema}
Other authors, as in \cite{mao2006stochastic,shen2005stochastic}, and \cite{bahar2004stochastic} adopt a different definition  for this last notion (Stochastically ultimate boundedness), according to which the system \eqref{systo2} is stochastically bounded in probability, if for any $\varepsilon\in\left(0,1\right)$ there is some $\delta_\varepsilon>0$ for which the following inequality: 
\begin{equation}\label{ine}
\limsup\limits_{t\to\infty}\mathbb{P}\left(\left\|X\left(t,X_0\right)\right\|\leqslant\delta_\varepsilon\right)\geqslant 1-\varepsilon
\end{equation}
is satisfied for any initial value   $X_0\in\mathbb{R}_+^7$.We mention that the inequality \eqref{ine} can equivalently be  rewritten as: 
\begin{equation*}
\liminf\limits_{t\to\infty}\mathbb{P}\left(\left\|X\left(t,X_0\right)\right\|>\delta_\varepsilon\right)\leqslant \varepsilon,
\end{equation*}
and this allows us to say that if our system is stochastically ultimate bounded in the sense of Definition \ref{sub}, then it will also be so in the sense of \eqref{ine}.
\end{rema}
\begin{defi}[Stochastic persistence \cite{liu2017permanence}]
Let $X\left(t,X_0\right)$ be the solution of \eqref{systo} that verifies $X(0)=X_0$, we say that the system \eqref{systo} is stochastically persistent, if for each $\mathlarger\varepsilon>0$, there exists a positive constant $\eta_\varepsilon>0$ such that the following property is satisfied:
$$\liminf\limits_{t\to\infty}\mathbb{P}\left(\left\|X\left(t,X_0\right)\right\|\geqslant\eta_\varepsilon\right)\geqslant 1-\varepsilon,~~\text{for all}~X_0\in\mathbb{R}_+^7.$$
\end{defi}
\begin{defi}[Stochastic permanence \cite{cai2017stochastic}]
\label{perm}
The system \eqref{systo2} is called stochastically permanent, if it is both stochastically ultimate bounded and persistent.
\end{defi}
\begin{theo}\label{bounded}
The system \eqref{systo} is stochastically ultimate bounded, and even stochastically permanent.
\end{theo}
\begin{proof}
Let $X_0\in\mathbb{R}_+^7$, by summing up the seven equations in \eqref{systo2} and denoting $N=S+Q+E+A+I+H+R$, we obtain for all $t\geqslant0$ 
\begin{align*}
\text{d}N(t)&=\left(\Lambda-\mu N(t)-d_A A(t)-d_I I(t)-d_H H(t)\right)~\text{d}t+\sigma_1S(t)~\text{d}B_1(t)+\sigma_2Q(t)~\text{d}B_2(t)+\sigma_3E(t)~\text{d}B_3(t)\\
&\quad+\sigma_4A(t)~\text{d}B_4(t)+\sigma_5I(t)~\text{d}B_5(t)+\sigma_6H(t)~\text{d}B_6(t)+\sigma_7R(t)~\text{d}B_7(t).
\end{align*}
Define a $\mathcal{C}^2$ function $\widetilde{V}:\mathbb{R}_+\to [2,\infty)$ by $\widetilde{V}(x)=x+\dfrac{1}{x}$, the It\^{o}'s formula shows that
\begin{align*}
\text{d}\widetilde{V}(N(t))&=\mathcal{L}\widetilde{V}(N(t))\,\text{d}t+\left(1-\dfrac{1}{N(t)^2}\right)\big[\sigma_1S(t)~\text{d}B_1(t)+\sigma_2Q(t)~\text{d}B_2(t)+\sigma_3E(t)~\text{d}B_3(t)+\sigma_4A(t)~\text{d}B_4(t)\\
&\quad+\sigma_5I(t)~\text{d}B_5(t)+\sigma_6H(t)~\text{d}B_6(t)+\sigma_7R(t)~\text{d}B_7(t)\big],
\end{align*}
where $\mathcal{L}\widetilde{V}(N)$ is given by
\begin{align*}
\mathcal{L}\widetilde{V}(N)&=\big(\Lambda-\mu N-d_A A-d_I I-d_H H\big)-\dfrac{\Lambda-\mu N-d_A A-d_I I-d_H H}{N^2}+\dfrac{1}{N^3}\left(\sigma_1^2 S^2+\sigma_2^2 Q^2+\sigma_3^2 E^2\right.\\
&\quad\left.+\sigma_4^2 A^2+\sigma_5^2 I^2+\sigma_6^2 H^2+\sigma_7^2 R^2\right)\\
&\leqslant \big(\Lambda-\mu N\big)-\dfrac{\Lambda}{N^2}+\dfrac{\mu+d_A+d_I+d_H}{N}+\dfrac{1}{N}\sum\limits_{i=1}^7 \sigma_i^2\\
&\leqslant -\mu\left(N+\frac{1}{N}\right)+\Lambda+\dfrac{1}{N}\left(2\mu+d_A+d_I+d_H+\sum\limits_{i=1}^7 \sigma_i^2\right)-\dfrac{\Lambda}{N^2}\\
&\leqslant -\mu \widetilde{V}(N)\hspace{-2pt}-\hspace{-2pt}\left(\dfrac{\sqrt{\Lambda}}{N}\hspace{-2pt}-\hspace{-2pt}\dfrac{1}{2\sqrt{\Lambda}}\hspace{-2pt}\left(2\mu+d_A+d_I+d_H+\hspace{-2pt}\sum\limits_{i=1}^7 \sigma_i^2\right)\right)^2\hspace{-6pt}+\underbrace{\dfrac{1}{4\Lambda}\hspace{-2pt}\left(\hspace{-0.2pt}2\mu\hspace{-1pt}+\hspace{-1pt}d_A\hspace{-1pt}+\hspace{-1pt}d_I+\hspace{-1pt}d_H\hspace{-1pt}+\hspace{-2pt}\sum\limits_{i=1}^7 \sigma_i^2\hspace{-0.2pt}\right)^2\hspace{-6pt}+\hspace{-1pt}\Lambda}_{:=\mathcal{C}}\\[-10pt]
&\leqslant -\mu \widetilde{V}\left(N\right)+\mathcal{C}. 
\end{align*}
Applying the integration by parts formula (see \cite[page 37]{mao2007stochastic}) to $e^{\mu t}\widetilde{V}\left(N(t)\right)$ gives
\begin{align*}
\text{d}e^{\mu t}\widetilde{V}\left(N(t)\right)&=\mu e^{\mu t}\widetilde{V}\hspace{-0.5pt}\left(N(t)\right)\:\text{d}t+e^{\mu t}\,\text{d}\widetilde{V}\left(N(t)\right)\\
&\leqslant \mu e^{\mu t}\widetilde{V}\hspace{-0.5pt}\left(N(t)\right)\text{d}t\hspace{-1.5pt}+\hspace{-1.5pt}e^{\mu t}\hspace{-2pt}\Big[\left(\hspace{-1pt}-\mu\widetilde{V}\left(N(t)\right)\hspace{-1.5pt}+\hspace{-1.5pt}\mathcal{C}\right)\text{d}t\hspace{-1.5pt}+\hspace{-1.5pt}\left(\hspace{-2pt}1\hspace{-1pt}-\hspace{-1pt}\dfrac{1}{N(t)^2}\hspace{-2pt}\right)\left[\sigma_1S(t)~\text{d}B_1(t)+\sigma_2Q(t)~\text{d}B_2(t)\right.\\
&\quad \left.+\sigma_3E(t)~\text{d}B_3(t)+\sigma_4A(t)~\text{d}B_4(t)+\sigma_5I(t)~\text{d}B_5(t)+\sigma_6H(t)~\text{d}B_6(t)+\sigma_7R(t)~\text{d}B_7(t)\right]\Big]\\
&\leqslant \mathcal{C} e^{\mu t} \,\text{d}t+e^{\mu t}\left(\hspace{-2pt}1\hspace{-1pt}-\hspace{-1pt}\dfrac{1}{N(t)^2}\hspace{-2pt}\right)\left[\sigma_1S(t)~\text{d}B_1(t)+\sigma_2Q(t)~\text{d}B_2(t)+\sigma_3E(t)~\text{d}B_3(t)+\sigma_4A(t)~\text{d}B_4(t)\right.\\
&\quad \left.+\sigma_5I(t)~\text{d}B_5(t)+\sigma_6H(t)~\text{d}B_6(t)+\sigma_7R(t)~\text{d}B_7(t)\right].
\end{align*}
By integrating  from $0$ to $t\wedge\tau_k$ ($\tau_k$ is already given  in \eqref{ligne}), and then taking the expectation on both sides of this inequality, we get for all $t\geqslant 0$ and $k\geqslant k_0$
\begin{align}\label{ligne2}
\mathbb{E}\left[e^{\mu(t\wedge\tau_k)}\widetilde{V}\left(N(t\wedge \tau_k)\right)\right]\leqslant \widetilde{V}\left(N(0)\right)+\mathbb{E}\left[\int_0^{t\wedge\tau_k}\hspace{-15pt}\mathcal{C} \times e^{\mu s}\:\text{d}s\right]\leqslant\widetilde{V}\left(N(0)\right)+\dfrac{\mathcal{C}}{\mu}\left(e^{\mu t}-1\right).
\end{align}  
According to Theorem \ref{thm1}, $\tau_k\to\infty$ almost surely as $k\to\infty$, so extending $k$ to $\infty$ in \eqref{ligne2} leads to
\begin{align*}
\mathbb{E}\left[\widetilde{V}\left(N(t)\right)\right]\leqslant \widetilde{V}\left(N(0)\right)\times e^{-\mu t}+\dfrac{\mathcal{C}}{\mu}\left(1-e^{-\mu t}\right)\leqslant \widetilde{V}\left(N(0)\right)\times e^{-\mu t}+\dfrac{\mathcal{C}}{\mu},
\end{align*}
Let $\varepsilon>0$, and take $\delta_\varepsilon =\frac{\mathcal{C}}{\varepsilon\mu}$, by making use of the well-known Markov's inequality  (see for example \cite{cohen2015markov} and the related bibliography), we get 
\begin{equation*}
\mathbb{P}\left(\widetilde{V}\left(N(t)\right)>\delta_\varepsilon\right)\leqslant \dfrac{1}{\delta_\varepsilon}\times\mathbb{E}\big[\widetilde{V}\left(N(t)\right)\big]\leqslant \dfrac{1}{\delta_\varepsilon} \left(\widetilde{V}\left(N(0)\right)\times e^{-\mu t}\right)+\varepsilon.
\end{equation*}
Then
\begin{equation*}
\mathbb{P}\left(N(t)+\dfrac{1}{N(t)}\leqslant \delta_{\varepsilon}\right)\geqslant 1-\varepsilon-\dfrac{1}{\delta_\varepsilon} \left(\widetilde{V}\left(N(0)\right)\times e^{-\mu t}\right).
\end{equation*}
Therefore
\begin{equation*}
\mathbb{P}\left(\dfrac{1}{\delta_\varepsilon}\leqslant N(t)\leqslant \delta_{\varepsilon}\right) \geqslant \mathbb{P}\left(N(t)+\dfrac{1}{N(t)}\leqslant \delta_{\varepsilon}\right)\geqslant 1-\varepsilon-\dfrac{1}{\delta_\varepsilon} \left(\widetilde{V}\left(N(0)\right)\times e^{-\mu t}\right).
\end{equation*}
By noting that $$\frac{1}{\sqrt{7}}\underbrace{(S+Q+E+A+I+H+R)}_N\leqslant\underbrace{\sqrt{S^2+Q^2+E^2+A^2+I^2+H^2+R^2}}_{\left\|X\right\|}\leqslant \underbrace{S+Q+E+A+I+H+R}_N,$$ we obtain \vspace{-5pt}
$$1-\varepsilon-\dfrac{1}{\delta_\varepsilon} \left(\widetilde{V}\left(N(0)\right)\times e^{-\mu t}\right)\leqslant\mathbb{P}\left(\dfrac{1}{\delta_\varepsilon}\leqslant N(t)\leqslant \delta_{\varepsilon}\right) \leqslant \mathbb{P}\left(\dfrac{1}{\sqrt{7}\delta_\varepsilon}\leqslant\left\|X\left(t,X_0\right)\right\|\leqslant \delta_{\varepsilon}\right),$$
so \vspace{-9pt} $$\liminf\limits_{t\to\infty}\mathbb{P}\Big(\overbrace{\dfrac{1}{\sqrt{7}\delta_\varepsilon}}^{\displaystyle{:=\eta_\varepsilon}}\leqslant\left\|X\left(t,X_0\right)\right\|\leqslant \delta_{\varepsilon}\Big)\geqslant 1-\varepsilon.$$ 
Hence, the theorem has been proved.
\end{proof}
\subsection{Stochastic extinction of COVD-19}
In epidemiology, we are usually concerned about two things, the first, is  to know when the disease will die out, and the second, is when it well persist. In this subsection, we will try our best to find a condition for the extinction of the disease expressed in terms of system parameters and intensities of noises, and for the persistence, it will be dealt with in the next subsection.
\begin{defi}[Stochastic extinction \cite{cai2017stochastic}]
For system \eqref{systo}, the infected individuals $E(t),A(t)$ and $I(t)$ are said to be stochastically extinct, or extinctive, if
$\lim\limits_{t\to\infty} E(t)+I(t)+A(t)=0$ almost surely.    
\end{defi} 
Before stating the result to be proved, we must firstly give the following  useful lemma that was stated and proved as Lemma 3.1 in \cite{zhang2017threshold}. 
\begin{lemm}\label{lemme1}
For any initial value $X_0\in\mathbb{R}_+^7$, the solution $X(t)=\left(S(t),Q(t),E(t),A(t),I(t),H(t), R(t)\right)$ of system \eqref{systo2} verifies the following properties:
\begin{enumerate}[label=$(\textup{\alph*})$]
\item $\lim\limits_{t\to\infty}\dfrac{X_k(t)}{t}=0~\text{a.s.}\quad \forall k\in\left\lbrace 1,2,\cdots,7\right\rbrace.$\label{lemme1a}
\item Moreover, if $\mu>\frac{1}{2}\left(\sigma_1^2\vee \sigma_2^2 \vee \sigma_3^2\vee \sigma_4^2 \vee \sigma_5^2\vee \sigma_6^2 \vee \sigma_7^2\right)$, then\\[3pt]
\hspace*{3cm}$\lim\limits_{t\to\infty}\dfrac{\int_0^t X_k(s)\,\textup{d}B_k(s)}{t}=0\quad \text{a.s.}\quad \forall k\in\left\lbrace 1,2,\cdots,7\right\rbrace.$\label{lemme1b}
\end{enumerate}
\end{lemm}
\begin{proof}
The proof of this lemma is similar in spirit to that of lemmas 2.1 and 2.2 of \cite{zhao2014threshold} and therefore it is omitted here. 
\end{proof}
\begin{theo}\label{thm2}
Let us denote by $X(t)=\big(S(t),Q(t),E(t),A(t),I(t),H(t), R(t)\big)$ the solution of system \eqref{systo} that starts from a given value $X_0=\big(S(0),Q(0),E(0),A(0),I(0),H(0), R(0)\big)\in \mathbb{R}_+^7$.\\If $\mu>\frac{1}{2}\left(\sigma_1^2\vee \sigma_2^2 \vee \sigma_3^2\vee \sigma_4^2 \vee \sigma_5^2\vee \sigma_6^2 \vee \sigma_7^2\right)$ and $\sigma_3^2\wedge\sigma_4^2\wedge \sigma_5^2>6\times\left(\beta_1S^o-\mu\right)$, with $S^o=\frac{\Lambda}{\mu}\cdot\frac{\lambda+\mu}{\lambda+q+\mu}$, then
$$\limsup_{t\to \infty}\dfrac{\ln(E(t)+A(t)+I(t))}{t}\leqslant\beta_1S^o-\mu-\frac{\sigma_3^2\wedge\sigma_4^2\wedge \sigma_5^2}{6}<0 \quad \text{a.s.},$$
which means that the disease will die out exponentially with probability one.
\end{theo}
\begin{proof}
From It\^{o}'s formula and system \eqref{systo}, we have
\begin{align*}
\text{d}\ln\left(E+A+I\right)&=\left[\dfrac{1}{E+A+I}\Bigg(\Big(\beta_1-\beta_2\dfrac{I}{b+I}\Big)S\left(I+\theta A\right)-\left(\varepsilon_A+\gamma_A+d_A\right)A-\left(\varepsilon_I+\gamma_I+d_I\right)I\Bigg)-\mu\right.\\
&\quad\left.-\dfrac{\sigma_3^2~E^2+\sigma_4^2~A^2+\sigma_5^2~I^2}{2\left(E+A+I\right)^2}\right]\text{d}t+\dfrac{1}{E+A+I}\big(\sigma_3 E~\text{d}B_3(t)+\sigma_4 A~\text{d}B_4(t)+\sigma_5I~\text{d}B_5(t)\big).
\end{align*}
Thus
\begin{align*}
 \text{d}\ln\left(E+A+I\right)&\leqslant \left[\Big(\beta_1-\beta_2\dfrac{I}{b+I}\Big)S-\mu-\dfrac{\sigma_3^2\wedge\sigma_4^2\wedge \sigma_5^2}{2}\times \dfrac{E^2+A^2+I^2}{\left(E+A+I\right)^2}\right]\text{d}t\\
&\quad+\sigma_3\dfrac{E}{E+A+I}~\text{d}B_3(t)+\sigma_4\dfrac{A}{E+A+I}~\text{d}B_4(t)
+\sigma_5\dfrac{I}{E+A+I}~\text{d}B_5(t).
\end{align*}
By using the famous Cauchy-Schwartz inequality (see for instance \cite{yin2017new} and the references given there), we can assert that $$\frac{E^2+A^2+I^2}{\left(E+A+I\right)^2}\geqslant \frac{1}{3}.$$ 
Hence
\begin{align}\label{eq5}
 \text{d}\ln\left(E+A+I\right)&\leqslant \left[\beta_1 S-\mu-\dfrac{\sigma_3^2\wedge\sigma_4^2\wedge \sigma_5^2}{6}\right]\text{d}t+\sigma_3\dfrac{E}{E+A+I}~\text{d}B_3(t)+\sigma_4\dfrac{A}{E+A+I}~\text{d}B_4(t)\nonumber \\
&\quad+\sigma_5\dfrac{I}{E+A+I}~\text{d}B_5(t),
\end{align}
Integrating \eqref{eq5} from $0$ to $t$, and then dividing by $t$ on both sides, we get
\begin{align}\label{eq6}
\dfrac{\ln\left(E(t)+A(t)+I(t)\right)}{t}&\leqslant \dfrac{\ln\left(E(0)+A(0)+I(0)\right)}{t}+\beta_1 \langle S(t)\rangle-\mu-\dfrac{\sigma_3^2\wedge\sigma_4^2\wedge \sigma_5^2}{6}\nonumber\\
&\quad+\dfrac{\sigma_3}{t}\int_0^t\dfrac{E(s)}{E(s)+A(s)+I(s)}~\text{d}B_3(s)+\dfrac{\sigma_4}{t}\int_0^t\dfrac{A(s)}{E(s)+A(s)+I(s)}~\text{d}B_4(s)\nonumber\\
&\quad +\dfrac{\sigma_5}{t}\int_0^t\dfrac{I(s)}{E(s)+A(s)+I(s)}~\text{d}B_5(s).
\end{align}
On the other hand, the first equation of \eqref{systo} gives
\begin{align*}
S(t)-S(0)&=\Lambda t -\int_0^t \left(\beta_1-\beta_2\dfrac{I(s)}{b+I(s)}\right) S(s)\left(I(s)+\theta A(s)\right)~\mathrm{d}s+\lambda \int_0^t Q(s)~\mathrm{d}s -(q+\mu)\int_0^t S(s)~\mathrm{d}s\\
&\quad+\sigma_1 \int_0^t S(s)~\textup{d}B_{1}(s)\\
&\leqslant \Lambda t+\lambda \int_0^t Q(s)~\mathrm{d}s-(q+\mu)\int_0^t S(s)~\mathrm{d}s+\sigma_1 \int_0^t S(s)~\textup{d}B_{1}(s).
\end{align*}
Therefore
\begin{align}
\langle S(t)\rangle=\dfrac{1}{t}\int_0^t S(s)~\mathrm{d}s &\leqslant \dfrac{1}{q+\mu}\left(\Lambda+\dfrac{\lambda}{t}\int_0^tQ(s)~\text{d}s+\dfrac{S(0)}{t}+\dfrac{\sigma_1}{t} \int_0^t S(s)~\textup{d}B_{1}(s)-\dfrac{S(t)}{t}\right) \nonumber \\ 
&\leqslant \dfrac{1}{q+\mu}\left(\Lambda+\lambda \langle Q(t)\rangle+\dfrac{S(0)}{t}+\dfrac{\sigma_1}{t} \int_0^t S(s)~\textup{d}B_{1}(s)\right).\label{ligne3}
\end{align}
Also, the second one gives
\begin{align*}
Q(t)-Q(0)=q \int_0^t S(s)~\text{d}s -\left(\mu+\lambda\right)\int_{0}^t Q~\text{d}s +\sigma_2\int_0^t Q(s)~\text{d}B_2(s),
\end{align*}
which shows that
\begin{align}
\langle Q(t)\rangle=\dfrac{1}{t}\hspace{-1pt}\int_0^t\hspace{-3pt}Q(s)\,\text{d}s&=\dfrac{1}{\lambda+\mu}\left(\dfrac{Q(0)-Q(t)}{t}\hspace{-1pt}+\hspace{-1pt}\frac{q}{t}\hspace{-1pt}\int_0^t \hspace{-3pt}S(s)\,\text{d}s\hspace{-1pt}+\hspace{-1pt}\dfrac{\sigma_2}{t}\hspace{-1pt}\int_0^t\hspace{-3pt} Q(s)\,\text{d}B_2(s)\right)\label{ligne4}\\
&\leqslant  \dfrac{1}{\lambda+\mu}\times\dfrac{Q(0)}{t}+\dfrac{q}{\lambda+\mu}\langle S(t)\rangle+\dfrac{\sigma_2}{\lambda+\mu}\times \dfrac{1}{t}\int_0^t Q(s)\,\text{d}B_2(s).\label{ligne5}
\end{align}
Combining \eqref{ligne3} with \eqref{ligne5} yields
\begin{align*}
\langle S(t) \rangle\hspace{-2pt}\leqslant\hspace{-2pt}\dfrac{1}{q+\mu}\hspace{-0.9pt}\left(\Lambda\hspace{-1pt}+\hspace{-1pt}\lambda\left(\hspace{-1pt}\dfrac{Q(0)}{\left(\lambda+\mu\right)t}\hspace{-1pt}+\hspace{-1pt}\dfrac{q}{\lambda+\mu}\langle S(t)\rangle\hspace{-1pt}+\hspace{-1pt}\dfrac{\sigma_2}{\lambda+\mu}\hspace{-3pt}\times\hspace{-3pt}\dfrac{1}{t}\int_0^t \hspace{-3pt} Q(s)\,\text{d}B_2(s)\hspace{-1pt}\right)\hspace{-1pt}+\hspace{-1pt}\dfrac{S(0)}{t}\hspace{-1pt}+\hspace{-1pt}\dfrac{\sigma_1}{t} \int_0^t \hspace{-3pt}S(s)\,\textup{d}B_{1}(s)\right)\hspace{-1pt}.
\end{align*}
Hence,
\begin{align}\label{ineq}
\langle S(t)\rangle\hspace{-1pt}&\leqslant\hspace{-1pt} \dfrac{\Lambda(\lambda+\mu)}{\mu(q+\mu+\lambda)}\hspace{-1pt}+\hspace{-1pt}\dfrac{\lambda}{\mu(q+\mu+\lambda)}\dfrac{Q(0)}{t}\hspace{-1pt}+\hspace{-1pt}\dfrac{\lambda+\mu}{\mu(q+\mu+\lambda)}\dfrac{S(0)}{t}\hspace{-1pt}+\hspace{-1pt}\dfrac{\lambda\sigma_2}{\mu(q+\mu+\lambda)}\hspace{-1pt}\times\hspace{-1pt} \frac{1}{t}\int_0^t\hspace{-3pt} Q(s)\,\text{d}B_2(s)\hspace{-1pt}\nonumber\\
&\quad+\hspace{-1pt}\sigma_1 \dfrac{(\lambda+\mu)}{\mu(q+\mu+\lambda)}\times \frac{1}{t}\int_0^t\hspace{-3pt} S(s)\,\text{d}B_1(s).
\end{align}
Since $\mu>\frac{1}{2}\left(\sigma_1^2\vee \sigma_2^2 \vee \sigma_3^2\vee \sigma_4^2 \vee \sigma_5^2\vee \sigma_6^2 \vee \sigma_7^2\right)$, we can conclude by virtue of Lemma \ref{lemme1}\ref{lemme1b} and inequality \eqref{ineq}  that
\begin{equation}\label{S}
\lim_{t\to\infty}\langle S(t)\rangle\leqslant \dfrac{\Lambda}{\mu}\times\dfrac{\lambda+\mu}{q+\mu+\lambda}=S^o.
\end{equation}
According to the strong law of large numbers for local martingales \big(see \cite[page 12]{mao2007stochastic}\big), we have
\begin{equation}\label{eq7}
\left\{\begin{aligned}
\lim_{t\to \infty}\dfrac{\sigma_3}{t}\int_0^t\dfrac{E(s)}{E(s)+A(s)+I(s)}~\text{d}B_3(s)&=0\quad \text{a.s.} ,\\
\lim_{t\to \infty}\dfrac{\sigma_4}{t}\int_0^t\dfrac{A(s)}{E(s)+A(s)+I(s)}~\text{d}B_4(s)&=0\quad \text{a.s.} ,\\
\lim_{t\to \infty}\dfrac{\sigma_5}{t}\int_0^t\dfrac{I(s)}{E(s)+A(s)+I(s)}~\text{d}B_5(s)&=0\quad \text{a.s.}
\end{aligned} \right.
\end{equation}
From \eqref{eq6}, \eqref{S} and \eqref{eq7} we get 
$$\limsup_{t\to \infty}\dfrac{\ln(E(t)+A(t)+I(t))}{t}\leqslant\beta_1S^o-\mu-\frac{\sigma_3^2\wedge\sigma_4^2\wedge \sigma_5^2}{6}<0 \quad \text{a.s.} ,$$
which is  exactly  the desired conclusion.
\end{proof}
\begin{rema} Unlike Theorem 4.1 of \cite{liu2019dynamics}, the above proof uses just the second  assertion of Lemma \ref{lemme1}, and makes no appeal to the first one.
\end{rema}
\begin{rema}\label{rema}
Needless to say, the preceding theorem implies the stochastic extinction of infected individuals, which implies in turn (by the positivity of the solution) that $\lim\limits_{t\to\infty}\hspace{-3pt}E(t)=0$,\hspace{-3pt} $\lim\limits_{t\to\infty}\hspace{-3pt}A(t)=0$ and \hspace{-3pt} $\lim\limits_{t\to\infty}\hspace{-3pt}I(t)=0$ a.s. (see for   example \cite{kiouach2011global,song2018extinction} and \cite[page 5071]{ji2014threshold}).        
\end{rema}
\begin{coro}\label{coro}
Under the same notations and hypotheses as in Theorem  \ref{thm2}, we have
$$\lim_{t\to\infty}\langle S(t)\rangle=S^{o}~\text{a.s.},\quad \lim_{t\to\infty}\langle Q(t)\rangle=Q^o~\text{a.s.},\quad \lim_{t\to\infty}\langle R(t)\rangle=0~\text{a.s.},\quad \text{and}~~\lim_{t\to\infty}\langle H(t)\rangle=0~\text{a.s.}$$
\end{coro}
\begin{proof}
Our proof starts with the observation that for all $t\geqslant 0$, we have
\begin{equation}\label{eq8}
\text{d}\big(S(t)+E(t)\big)=\left[\Lambda +\lambda Q(t)-\left(q+\mu\right)S(t)-\left(\mu+\sigma\right)E(t)\right]\text{d}t+\sigma_1S(t)~\text{d}B_1(t)+\sigma_3E(t)~\text{d}B_3(t).
\end{equation}
Integrating \eqref{eq8} from $0$ to $t$, and then dividing by $t$ on both sides gives
\begin{equation*}
\begin{aligned}
\dfrac{S(t)+E(t)}{t}-\dfrac{S(0)+E(0)}{t}&=\Lambda+\lambda\langle Q(t)\rangle-\left(q+\mu\right)\langle S(t)\rangle-\left(\mu+\sigma\right)\langle E(t)\rangle+\frac{\sigma_1}{t}\int_0^t S(s)~\text{d}B_1(s)\\
&\quad+\frac{\sigma_3}{t}\int_0^t E(s)~\text{d}B_3(s). 
\end{aligned}
\end{equation*}
 Replacing in the last equality $\langle Q(t)\rangle$ by its expression from \eqref{ligne4} yields 
\begin{align*}
\dfrac{S(t)+E(t)}{t}&=\dfrac{S(0)+E(0)}{t}\hspace{-1pt}+\hspace{-1pt}\Lambda\hspace{-1pt}+\hspace{-1pt}\dfrac{\lambda}{\lambda+\mu}\left(q\langle S(t)\rangle\hspace{-1pt}-\hspace{-1pt}\dfrac{Q(t)-Q(0)}{t}\hspace{-1pt}+\hspace{-1pt}\dfrac{\sigma_2}{t}\hspace{-1pt}\int_0^t\hspace{-3pt} Q(s)\,\text{d}B_2(s)\right)-\left(q+\mu\right)\langle S(t)\rangle\\
&\quad-\left(\mu+\sigma\right)\langle E(t)\rangle+\frac{\sigma_1}{t}\int_0^t S(s)~\text{d}B_1(s)+\frac{\sigma_3}{t}\int_0^t E(s)~\text{d}B_3(s).
\end{align*}
Hence
\begin{align}\label{eq9}
\left(-\frac{\lambda\times q}{\lambda+\mu}+q+\mu\right)\langle S(t)\rangle&=\Lambda+\dfrac{S(0)+E(0)}{t}+\dfrac{\lambda}{\lambda+\mu}\times \frac{Q(0)-Q(t)}{t}-\dfrac{S(t)+E(t)}{t}-\left(\mu+\sigma\right)\langle E(t)\rangle\nonumber\\
&\quad +\dfrac{\lambda}{\lambda+\mu}\dfrac{\sigma_2}{t}\int_0^t\hspace{-3pt}Q(s)\,\text{d}B_2(s)+\dfrac{\sigma_1}{t}\int_0^t\hspace{-3pt}S(s)\,\text{d}B_1(s)+\dfrac{\sigma_3}{t}\int_0^t \hspace{-3pt} E(s)\,\text{d}B_3(s).
\end{align}
Letting $t$ go to infinity on both sides of  \eqref{eq9}, then using \ref{lemme1a} and \ref{lemme1b} of Lemma \ref{lemme1}, we obtain
\begin{equation}
\lim_{t\to \infty}\left(-\frac{\lambda\times q}{\lambda+\mu}+q+\mu\right)\langle S(t)\rangle=\Lambda-\left(\mu+\sigma\right)\lim_{t\to\infty}\langle E(t)\rangle\quad \text{a.s.},
\end{equation}
On the account of  Remark \ref{rema}, we have$$\lim\limits_{t\to \infty} E(t)=0\quad\text{a.s.},$$
 which implies by the continuous version of Ces\`{a}ro's theorem \cite[page 3]{albanese2015continuous} that $$\lim\limits_{t\to\infty} \langle E(t)\rangle=0\quad \text{a.s.}$$
So
$$\lim_{t \to \infty}\left(-\frac{\lambda\times q}{\lambda+\mu}+q+\mu \right)\langle S(t)\rangle=\Lambda.$$
Therefore
\begin{equation}
\lim_{t\to \infty}\langle S(t)\rangle=\Lambda\left(-\frac{\lambda\times q}{\lambda+\mu}+q+\mu\right)^{-1}=\Lambda\times\dfrac{\lambda+\mu}{\mu\left(\lambda+q+\mu\right)}=S^o.
\end{equation} 
At the same time, we have
 $$\mathrm{d}Q(t)=\left[qS(t)-\left(\lambda+\mu\right)Q(t)\right]\mathrm{d}t+\sigma_{2}Q(t)~\mathrm{d}B_2(t).$$
Then $$\dfrac{Q(t)-Q(0)}{t}=q\langle S(t)\rangle -\left(\lambda+\mu\right)\langle Q(t)\rangle+\frac{\sigma_2}{t}\int_0^t Q(s)\,\text{d}B_2(s).$$
Hence
$$\langle Q(t)\rangle=\dfrac{q}{\lambda+\mu}\langle S(t)\rangle-\dfrac{Q(t)}{(\lambda+\mu)t}+\dfrac{Q(0)}{(\lambda+\mu)t}+\dfrac{\sigma_2}{\lambda+\mu}\times\dfrac{1}{t}\int_0^t Q(s)\,\text{d}B_2(s).$$
Consequently, and by a passage to the limit similar to the above, we get
$$\lim_{t\to\infty}\langle Q(t)\rangle=\dfrac{q}{q+\mu}\times S^o=Q^o.$$
The same reasoning remains valid for the two last assertions of our theorem $\big(\langle R(t)\rangle,\langle H(t)\rangle\underset{n\to\infty}{\longrightarrow} 0~~\text{a.s.} \big)$, and this finishes the proof, the detailed verification being  left to the reader.  
\end{proof}
\begin{rema}\label{remmm}
Obviously, if we keep the same notations and assumptions as in Theorem \ref{thm2}, the last result can be rewritten as follows:\\
\hspace*{7cm} $\lim\limits_{t\to\infty}\langle X(t)\rangle=\mathcal{E}_o\quad {a.s.}$,\\[6pt] 
which implies by  Ces\`{a}ro's theorem \cite[page 3]{albanese2015continuous}, that 
 if the solution $X(t)$ has a limit (finite or infinite)  as $t$ approaches infinity, almost everywhere (a.e.for brevity), then necessarily this limit will be equal to $\mathcal{E}_o$.
\end{rema}
\subsection{Persistence in the mean of COVID-19}
In the following, we give a condition for the persistence in the mean of the disease, but before stating the main result, we shall first recall the concept of persistence in the mean. 
\begin{defi}[Persistence in the mean \cite{song2018extinction,sun2020dynamics}]
For system \eqref{systo}, the infectious individuals $A(t)$ and $I(t)$ are said to be strongly persistent in the mean, or just persistent in the mean, if $\liminf\limits_{t\to\infty}\langle A(t)+I(t)\rangle>0$ almost surely. 
\end{defi} 
\begin{rema}
In the last definition, some authors use the word "Permanent" instead of "Persistent" \cite{han2018dynamics}, but others prefer to avoid this terminology so as not to confuse the notion of Persistence in the mean, with that of stochastic permanence which is completely different (see Definition \ref{perm}).
\end{rema}
For brevity and simplicity in writing the next results, it will be convenient to adopt the following notations:
\begin{itemize}
\item[$\bullet$] $\rho_1(\alpha)=3\times\sqrt[3]{\Lambda \left(\beta_1-\beta_2\right)\sigma}\times\left(\sqrt[3]{\theta(1-p)\times\alpha}+\sqrt[3]{p\times(1-\alpha) }\right),\quad \forall \alpha\in (0,1),$
\item[$\bullet$] $\rho_2=7\mu+\sigma+\left(\varepsilon_A+\gamma_A+d_A\right)+\left(\varepsilon_I+\gamma_I+d_I\right)+\left(d_H+\gamma_H\right)+\left|\lambda-q\right|+\dfrac{1}{2}\sum\limits_{i=1}^7\sigma_i^2,$
\item[$\bullet$] $\widehat{\alpha}=\dfrac{\sqrt{\theta(1-p)}}{\sqrt{\theta(1-p)}+\sqrt{p}}\in\left(0,1\right).$
\end{itemize}
\begin{lemm}\label{lass}
For any $\alpha\in(0,1)$, the following inequality is satisfied
\begin{equation}
\rho_1(\alpha)\leqslant \rho_1 (\widehat{\alpha}). 
\end{equation}
In other terms, $\rho_1 (\widehat{\alpha})$  is the maximum value of $\rho_1(\alpha)$ on the open interval $(0,1)$.
\end{lemm}
\begin{proof}
We start our proof by observing that the function $\rho_1( \alpha)$ is differentiable on $(0,1)$, with a first derivative given by:
\begin{align*}
\rho_1^{\prime}( \alpha):=\dfrac{\textup{d}\rho_1(\alpha)}{\textup{d}\alpha}&=\sqrt[3]{\Lambda \left(\beta_1-\beta_2\right)\sigma} \left(\dfrac{\sqrt[3]{\theta(1-p)}}{\sqrt[3]{\alpha^2}}-\dfrac{\sqrt[3]{p}}{\sqrt[3]{(1-\alpha)^2}}\right)\\
&=\dfrac{\sqrt[3]{\Lambda \left(\beta_1-\beta_2\right)\sigma}}{\sqrt[3]{\left(\alpha\times(1-\alpha)\right)^2}}\times\dfrac{\theta(1-p)(1-\alpha)^2-p\alpha^2}{\left(\sqrt[3]{\theta(1-p)\left(1-\alpha\right)^2}\right)^2+\sqrt[3]{\theta(1-p)p(1-\alpha)^2\alpha^2}+\left(\sqrt[3]{p\alpha^2}\right)^2}\\
&=\dfrac{\sqrt[3]{\Lambda \left(\beta_1-\beta_2\right)\sigma}\left(\sqrt{\theta(1-p)}(1-\alpha)+\sqrt{p}\alpha\right)\left(\sqrt{\theta(1-p)}+\sqrt{p}\right)}{\left(\left(1-\alpha\right)\sqrt[3]{\theta(1-p)\alpha}\right)^2+\sqrt[3]{\theta p(1-p)\alpha^4(1-\alpha)^4}+\left(\alpha\sqrt[3]{p(1-\alpha)}\right)^2}\times \left(\widehat{\alpha}-\alpha\right).
\end{align*}
As it can be seen, the derivative $\rho_1^{\prime}(\alpha)$ and the linear function $L(\alpha)=\widehat{\alpha}-\alpha$ have the same sign, so the function $\rho_1(\alpha)$ decreases for $\alpha\in(0,\widehat{\alpha})$ and increases for $\alpha\in(\widehat{\alpha},1)$. Therefore, the highest value of $\rho_1$ in the interval $(0,1)$ is $\rho_1 (\widehat{\alpha})$, and this is precisely the assertion of the lemma.
\end{proof}
 \begin{theo}\label{las}
If $\rho_1(\widehat{\alpha})>\rho_2$, then for any $X_0\hspace{-2pt}\in\hspace{-2pt} \mathbb{R}_+^7$,\hspace{-1pt} the solution $X\hspace{-0.5pt}(t)\hspace{-2pt}=\hspace{-2pt}\big(\hspace{-1pt}S(t),Q(t),E(t),A(t),I(t),H(t),R(t)\hspace{-1pt}\big)$ of the initial-value problem \eqref{systo2} verifies the following property:
$$\liminf\limits_{t\to \infty} \langle I(t)+A(t)\rangle\geqslant\dfrac{1}{\beta_1}\left(\rho_1(\widehat{\alpha})-\rho_2\right)>0\quad \text{a.s.,}$$
which is to say that  the infectious individuals $A(t)$ and $I(t)$ are persistent in the mean.
\end{theo}
\begin{proof}
Consider the function 
$$\begin{array}{crcl}
\widehat{V}:&\mathbb{R}^7_{+}&\longrightarrow &\mathbb{R}\\ 
 &x&\longmapsto&\sum\limits_{i=1}^7 \ln\left(x_i\right).
\end{array}$$ 
From It\^{o}'s formula and system \eqref{systo}, we have
\begin{align*}
\text{d}\widehat{V}(X(t))&=\Bigg(\left[\dfrac{\Lambda}{S}-\left(\beta_1-\beta_2\dfrac{I}{b+I}\right)\left(I+\theta A\right)+\lambda \dfrac{Q}{S}-(q+\mu)\right]+\left[q\dfrac{S}{Q}-\left(\lambda+\mu\right)\right]\\
&\quad+\left[\left(\beta_1-\beta_2\dfrac{I}{b+I}\right) \dfrac{S}{E}\left(I+\theta A\right)-(\mu+\sigma)\right]+\left[\left(1-p\right)\sigma \dfrac{E}{A}-\left(\mu+\varepsilon_A+\gamma_A+d_A\right)\right]\\
&\left.\quad+\left[\sigma p \dfrac{E}{I}-\left(\mu+\varepsilon_I+\gamma_I+d_I\right)\right]+\left[\varepsilon_I \dfrac{I}{H}+\varepsilon_A \dfrac{A}{H}-\left(d_H+\gamma_H+\mu \right)\right]\right.\\
&\quad+\left[\gamma_H \dfrac{H}{R}+\gamma_I \dfrac{I}{R}+\gamma_A\dfrac{A}{R} -\mu \right]-\dfrac{1}{2}\sum_{i=1}^7\sigma_i^2\Bigg)~\text{d}t+\sum_{i=1}^7 \sigma_i~\text{d}B_i(t)\\
&\geqslant \Bigg( \dfrac{\Lambda}{S}-\beta_1\left(I+\theta A\right)+\left(\lambda\wedge q\right)\left(\dfrac{S}{Q}+\dfrac{Q}{S}\right)+\left(\beta_1-\beta_2\right)\dfrac{S}{E}\left(I+\theta A\right)+\left(1-p\right)\sigma \dfrac{E}{A}+\sigma p \dfrac{E}{I}\\
&\quad-\left[7\mu+\lambda+q+\sigma+\left(\varepsilon_A+\gamma_A+d_A\right)+\left(\varepsilon_I+\gamma_I+d_I\right)+\left(d_H+\gamma_H\right)+\dfrac{1}{2}\sum\limits_{i=1}^7\sigma_i^2\right]\Bigg)\text{d}t\\
&\quad+\sum_{i=1}^7 \sigma_i~\text{d}B_i(t).
\end{align*}
 Noticing that $\lambda\wedge q=\dfrac{\lambda+q-\left|\lambda-q\right|}{2}$ and $\left(\dfrac{S}{Q}+\dfrac{Q}{S}\right)\geqslant 2$, we get for all $t\geqslant 0$
 \begin{align*}
\text{d}\widehat{V}(X(t))\hspace{-1pt}&\geqslant\hspace{-0.7pt} \Bigg(\hspace{-1pt}\left[\dfrac{(1-\widehat{\alpha})\Lambda}{S}+\left(\beta_1-\beta_2\right)\dfrac{SI}{E}+\sigma p \dfrac{E}{I}\right]\hspace{-2pt}+\hspace{-2pt}\left[\dfrac{\widehat{\alpha}\Lambda}{S}+\theta\left(\beta_1-\beta_2\right)\dfrac{SA}{E}+\left(1-p\right)\sigma \dfrac{E}{A} \right]\hspace{-2pt}\\
&\quad-\beta_1\left(I+\theta A\right)-\hspace{-1pt}\rho_2\hspace{-1pt}\Bigg)\text{d}t
+\sum_{i=1}^7 \sigma_i~\text{d}B_i(t),
\end{align*}
and from the relation between arithmetic and geometric means \big(the first is greater than or equal to the second, see \cite{nicholson2014concise}\big), it results that
\begin{align}\label{z}
\text{d}\widehat{V}(X(t))&\geqslant \left(3\hspace{-2pt}\times\hspace{-2pt}\sqrt[3]{(1\hspace{-1pt}-\hspace{-1pt}\widehat{\alpha})\Lambda\left(\beta_1\hspace{-1pt}-\hspace{-1pt}\beta_2\right)\sigma p}+3\hspace{-2pt}\times\hspace{-2pt}\sqrt[3]{\widehat{\alpha}\Lambda\left(\beta_1\hspace{-1pt}-\hspace{-1pt}\beta_2\right)\theta\sigma(1\hspace{-1pt}-\hspace{-1pt}p)}\hspace{-1pt}-\hspace{-1pt}\beta_1\left(I\hspace{-1pt}+\hspace{-1pt}\theta A\right)\hspace{-1pt}-\hspace{-1pt}\rho_2\right)\text{d}t+\hspace{-2pt}\sum_{i=1}^7 \sigma_i~\text{d}B_i(t)\nonumber\\
&\geqslant \big(\left(\rho_1(\widehat{\alpha})-\rho_2\right)-\beta_1\left(I\hspace{-1pt}+\hspace{-1pt}\theta A\right)\big)\,\text{d}t+\hspace{-2pt}\sum_{i=1}^7 \sigma_i~\text{d}B_i(t).
\end{align}
Integrating from $0$ to $t$ and dividing by $t$ on both sides of \eqref{z} gives
\begin{align}
&\dfrac{\widehat{V}(X(t))-\widehat{V}(X(0))}{t}\geqslant \left(\rho_1(\widehat{\alpha})-\rho_2\right)-\beta_1 \langle I(t)+\theta A(t) \rangle +\sum_{i=1}^7 \sigma_i \dfrac{B_i(t)}{t}.\nonumber
\end{align}
Hence
\begin{align}\label{a}
\langle I(t)+A(t) \rangle\geqslant \langle I(t)+\theta A(t) \rangle\geqslant\dfrac{1}{\beta_1}\left(\dfrac{\widehat{V}(X(0))-\widehat{V}(X(t))}{t}+\left(\rho_1(\widehat{\alpha})-\rho_2\right)\right)+\sum_{i=1}^7 \dfrac{\sigma_i}{\beta_1}\ \dfrac{B_i(t)}{t}.
\end{align}
Since $\ln(y)\leqslant y-1\leqslant y$ for all $y>0$, one can assert that $\widehat{V}(x)\leqslant \sum\limits_{i=1}^7x_i$ for any $x\in\mathbb{R}_+^7$.\\
Combining the last inequality with \eqref{a} yields
\begin{equation}
\langle I(t)+A(t) \rangle\geqslant\dfrac{1}{\beta_1}\left(\dfrac{\widehat{V}(X(0))}{t}-\dfrac{1}{t}\sum_{i=1}^7X_i(t)+\left(\rho_1(\widehat{\alpha})-\rho_2\right)\right)+\sum_{i=1}^7\dfrac{\sigma_i}{\beta_1}\ \dfrac{B_i(t)}{t}\nonumber.
\end{equation}
By using the strong law of large numbers for local martingales and the first assertion of Lemma \ref{lemme1}, we obtain
$$\liminf\limits_{t\to \infty} \langle I(t)+A(t)\rangle\geqslant\dfrac{1}{\beta_1}\left(\rho_1(\widehat{\alpha})-\rho_2\right)>0\quad \text{a.s.,}$$
which is the required assertion.
\end{proof}
\begin{rema}
In the last proof, we can notice that any constant $\alpha\in(0,1)$ can play the role of $\widehat{\alpha}$, but the peculiarity of the latter lies essentially in its capacity to weaken the hypothesis of Theorem \ref{las}. Indeed, according to Lemma \ref{lass}, if $\rho_1(\alpha)>\rho_2$ for some $\alpha\in(0,1)$ then necessarily $\rho_1(\widehat{\alpha})>\rho_2$.
\end{rema}
\section{Numerical simulation examples}\label{sec3}
In this section, and using the parameter values  as shown in Table \ref{tab}, we present some numerical simulations to validate the various results proved in this paper. Most of the parametric values appearing in this table (Table \ref{tab}) are selected from real data available in existing literature (\cite{jia2020modeling,tang2020estimation,wu2020quantifying,PHO} more precisely) and the rest of them are just assumed for numerical calculations. The solution of our COVID-19 model, in its both stochastic and deterministic forms, is simulated in our case with the initial state given by $S(0)=1.8\times 10^6,~ Q(0)=0,~ E(0)=10,~ A(0)=15,~ I(0)=8,~ H(0)=5$ and $R(0)=0$ (see \cite{tang2020estimation}). In what follows, the unity of time is one day and the number of individuals is expressed in one million population. 
\vspace*{1cm}
\begin{table}[H]
\begin{center}
\begin{tabular}{c|l|c}
\hline \hline
Parameter &  Description & Nominal value  
 \\ 
\hline 
$\Lambda$ & Recruitment rate & $108.63$ \\ 
$\beta_1$ & Contact rate in absence of media coverage & $\left(1.7  \times 10^{-9}\hspace{1pt}\textbf{,}\hspace{2pt}5.2 \hspace{-1.5pt}\times 10^{-3}\right)$ \\ 
$\beta_2$ & Awareness rate (or also response intensity) & $\left[0\hspace{1pt}\textbf{,}\hspace{1pt}\beta_1\right]$ \\ 
$b$ & Constant of media's half saturation   & $70$ \\ 
$\theta$ & Modification ratio of asymptomatic infectiousness & $0.0494$ \\ 
$q$ &  Quarantine rate & $0.071$ \\ 
$\lambda$ & Rate of release from quarantine & $0.1003$ \\
$\mu$ & Natural death rate & $0.00029$ \\ 
$\sigma$ & The transition rate of exposed individuals to the infective classes & $0.2$ \\
$p$ & Probability of having symptoms among infected individuals & $\left(0\textbf{,}1\right)$ \\
$\varepsilon_A$ & The hospitalization rate of asymptomatic infected individuals &  $0.1$\\
$\gamma_A$ & Recovery rate of asymptomatic infected individuals & $0.15$ \\
$d_A$ & Disease-induced death rate for asymptomatic infected individuals & $0.005$ \\
$\varepsilon_I$ & The hospitalization rate of symptomatic infected individuals & $0.33$ \\
$\gamma_I$ & Recovery rate of symptomatic infected individuals & $0.1001$ \\ 
$d_I$ & Disease-induced death rate for symptomatic infected individuals & $0.008$ \\
$\gamma_H$ & Recovery rate of hospitalized individuals & $0.14$ \\ 
$d_H$ & Disease-induced death rate for hospitalized individuals & $0.004$ \\
\hline \hline 
\end{tabular}
\end{center}
\caption{Definitions and values (per day) of  COVID-19 model parameters used in the simulation.}\label{tab}
\end{table}
\pagebreak

\begin{ex}[\textbf{Deterministic case}]\label{axample1}
In this case, and adopting the parameter values listed in Table \ref{tab}, we will illustrate the theoretical results of the first section. Figures \ref{Fig1} and \ref{Fig2} present the dynamical behaviour of the COVID-19 deterministic model  when $\beta_1$, $\beta_2$ and $p$ are conveniently fixed in their admissible ranges. In Figure \ref{Fig1}, we take $\beta_1=3.97\times 10^{-6}$, $\beta_2=0.6\hspace{-2pt}\times\hspace{-2pt}\beta_1$,  $p=0.6201$, and we get $\mathcal{R}_0=0.9180<1$. From the curves appearing in this figure, it is clear that the disease is dying out, and in addition to that, the solution $\big(S(t),Q(t),E(t),A(t),I(t),H(t),R(t) \big)$  converges to the free-disease state $\mathcal{E}^o=\left(1.5563\times 10^5,2.0896\times 10^5,0,0,0,0,0\right)$ which supports the Theorem \ref{theodfe}\textbf{.} On the other hand, and changing $\beta_1$ to $5\times 10^{-6}$, we obtain a basic reproductive number $\mathcal{R}_0$ greater than one ($\mathcal{R}_0=1.1562>1$). From Figure \ref{Fig2}, we observe the COVID-19 persistence in this case, which agree well with Theorem \ref{persist}.
\end{ex}
\vspace*{1cm}
\begin{ex}[\textbf{Stochastic case}]\label{axample2}
In order to exhibit the random fluctuations effect on COVID-19 dynamics, we present in Figures \ref{Fig3} and \ref{Fig4} a collection of numerical simulations. In the first instance, we take $\beta_1=2.08\times10^{-9}$, $\beta_2=0.6\hspace{-2pt}\times\hspace{-2pt}\beta_1$, $p=0.6201$, and we choose the stochastic  intensities as follows: $\sigma_1=0.024$, $\sigma_2=0.0235$, $\sigma_3=0.015$, $\sigma_4=0.0174$, $\sigma_5=0.019$, $\sigma_6=0.0213$, and $\sigma_7=0.0238$. Then, $$\frac{1}{2}\left(\sigma_1^2\vee \sigma_2^2 \vee \sigma_3^2\vee \sigma_4^2 \vee \sigma_5^2\vee \sigma_6^2 \vee \sigma_7^2\right)=0.000288<0.000290=\mu,$$ and $$\sigma_3^2\wedge\sigma_4^2\wedge \sigma_5^2=0.000225>0.000202=6\left(\beta_1S^o-\mu\right). $$
Hence, the assumptions of Theorem \ref{thm2} are verified, and consequently $$\limsup_{t\to \infty}\dfrac{\ln(E(t)+A(t)+I(t))}{t}\leqslant\beta_1S^o-\mu-\frac{\sigma_3^2\wedge\sigma_4^2\wedge \sigma_5^2}{6}=-3.83\times 10^{-6} <0 \quad \text{a.s.}$$
That is to say that the COVID-19 dies out exponentially almost surely. Moreover, by Corollary \ref{coro} and  Remark \ref{remmm}, the mean time of the solution  converges to the deterministic free-disease equilibrium $\mathcal{E}^o$. These two last results are confirmed by the curves depicted in Figure \ref{Fig3}. To make the condition $\rho_1(\widehat{\alpha})>\rho_2$ true, we take $\beta_1=4.1\times 10^{-3}$, $\beta_2=0.1\hspace{-2pt}\times\hspace{-2pt}\beta_1$ and we select new values of stochastic intensities as follows: $\sigma_1=0.019$, $\sigma_2=0.0185$, $\sigma_3=0.014$, $\sigma_4=0.017$, $\sigma_5=0.0158$, $\sigma_6=0.0136$, and $\sigma_7=0.0182$. Thus, the main result of Theorem \ref{las} is satisfied and this time, the COVID-19 persists  in the mean as  shown in Figure \ref{Fig4}.
\end{ex}
\vspace*{1cm}
\begin{ex}[\textbf{The effectiveness of media intervention and quarantine strategies}]
We aim during this example to examine numerically the impact of media intrusion and quarantine strategies on the COVID-19 spread. To this end, we simulate the progression of the total infected population number with various values of $\beta_2$,   $\lambda$ and $q$. Through Figure \ref{Fig5}, we can perceive that the increase of the quarantine rate and duration can delay the arrival of infection peak, reduce remarkably the impact of the disease, and even lead it to the extinction sometimes (see for example the last two curves presented in Figure \ref{Fig5}). On the other hand, and as it can be seen from figures \ref{Fig6} and \ref{Fig7}, the  media alert strategy is able also to diminish the  severity of the COVID-19 spread, but it can not make it disappear, and we explain this theoretically by  the absence of the  parameters $\beta_2$  and $b$ in the persistence and extinction conditions (for example $\mathcal{R}_0$ does not involve these parameters). Roughly speaking, the role of the quarantine and the information intervention about COVID-19 is critically important, particularly in its beginnings. The growth of the positive response in susceptible individuals leads to reduce the gravity of the infection and creates a conscious public able to overcome this new pandemic by respecting social distancing and self-isolation procedures.
\end{ex}
\pagebreak
\begin{figure}[H]
\centering
\subfigure{
    \includegraphics[width=.4\linewidth]{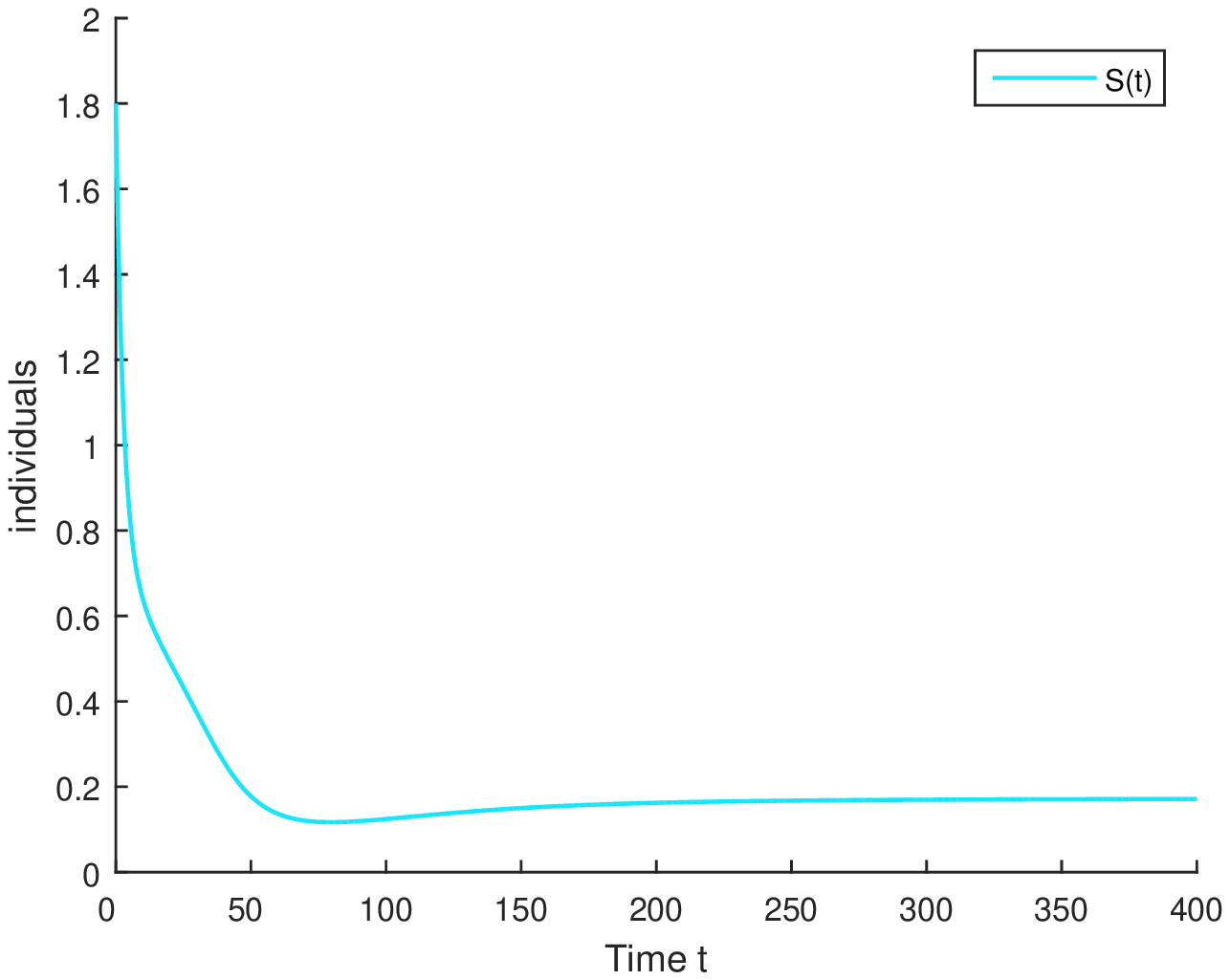}
  }%
\subfigure{
    \includegraphics[width=.4\linewidth]{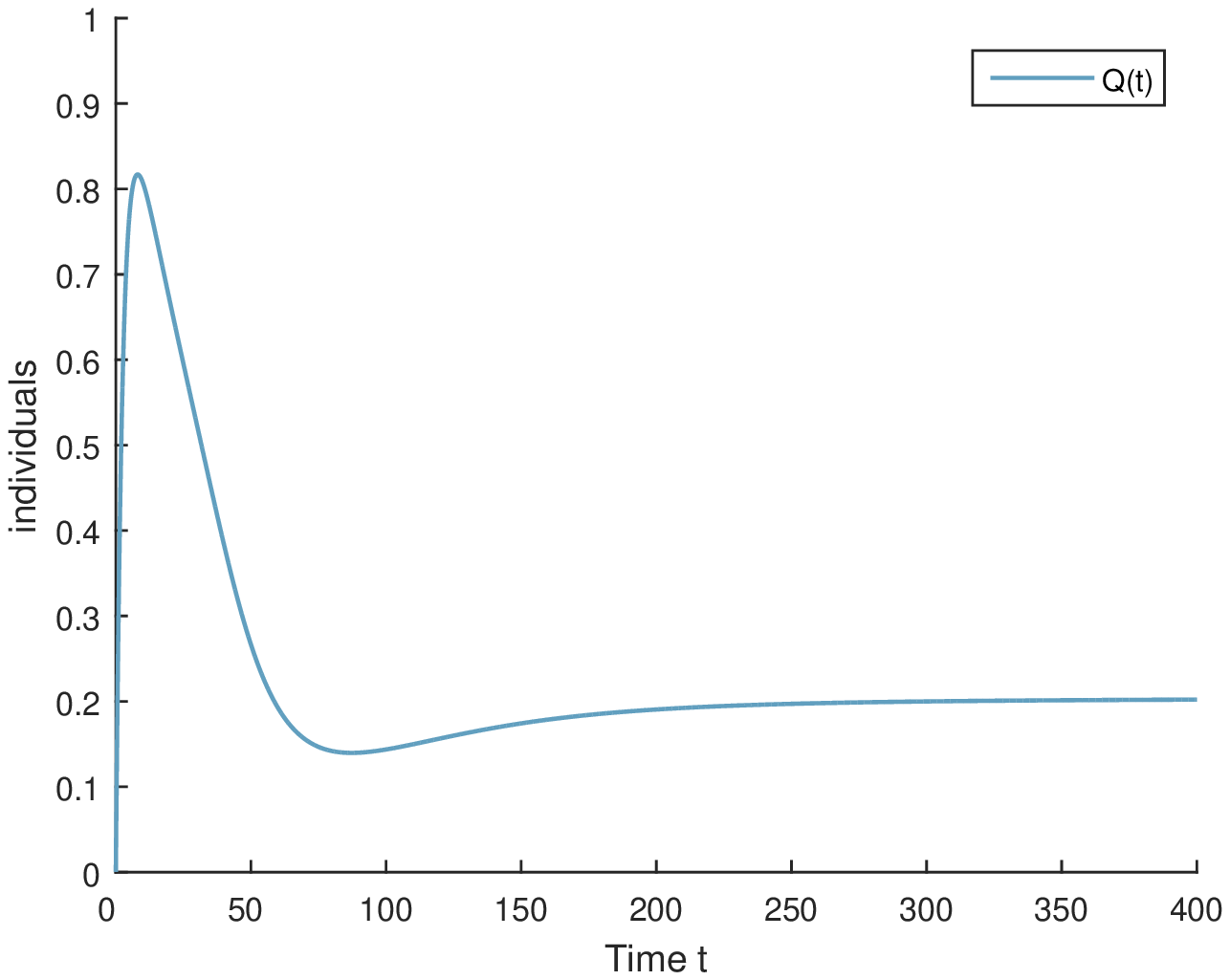}
  }\\[-2pt]
  \subfigure{
    \includegraphics[width=.4\linewidth]{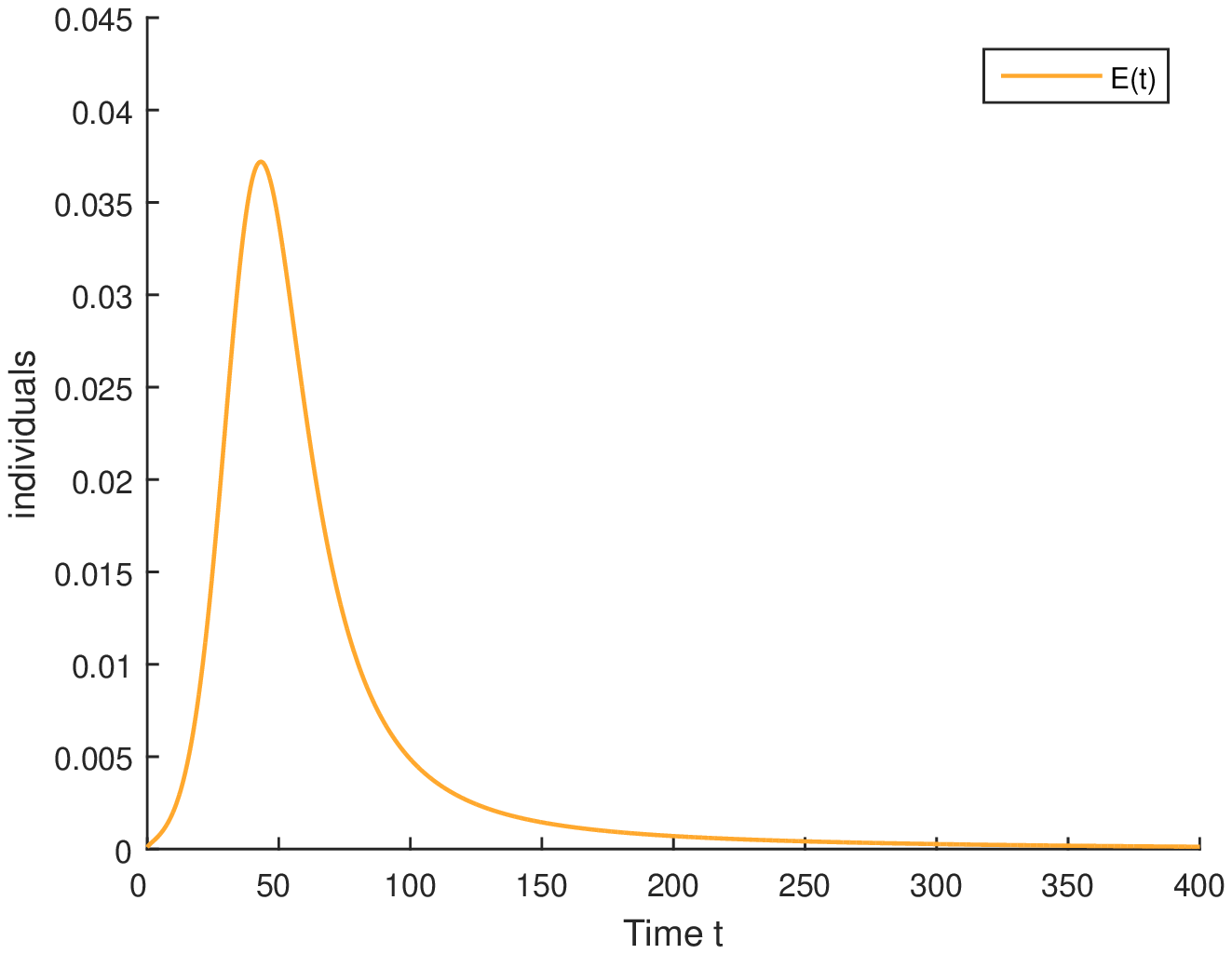}
  }%
\subfigure{
    \includegraphics[width=.4\linewidth]{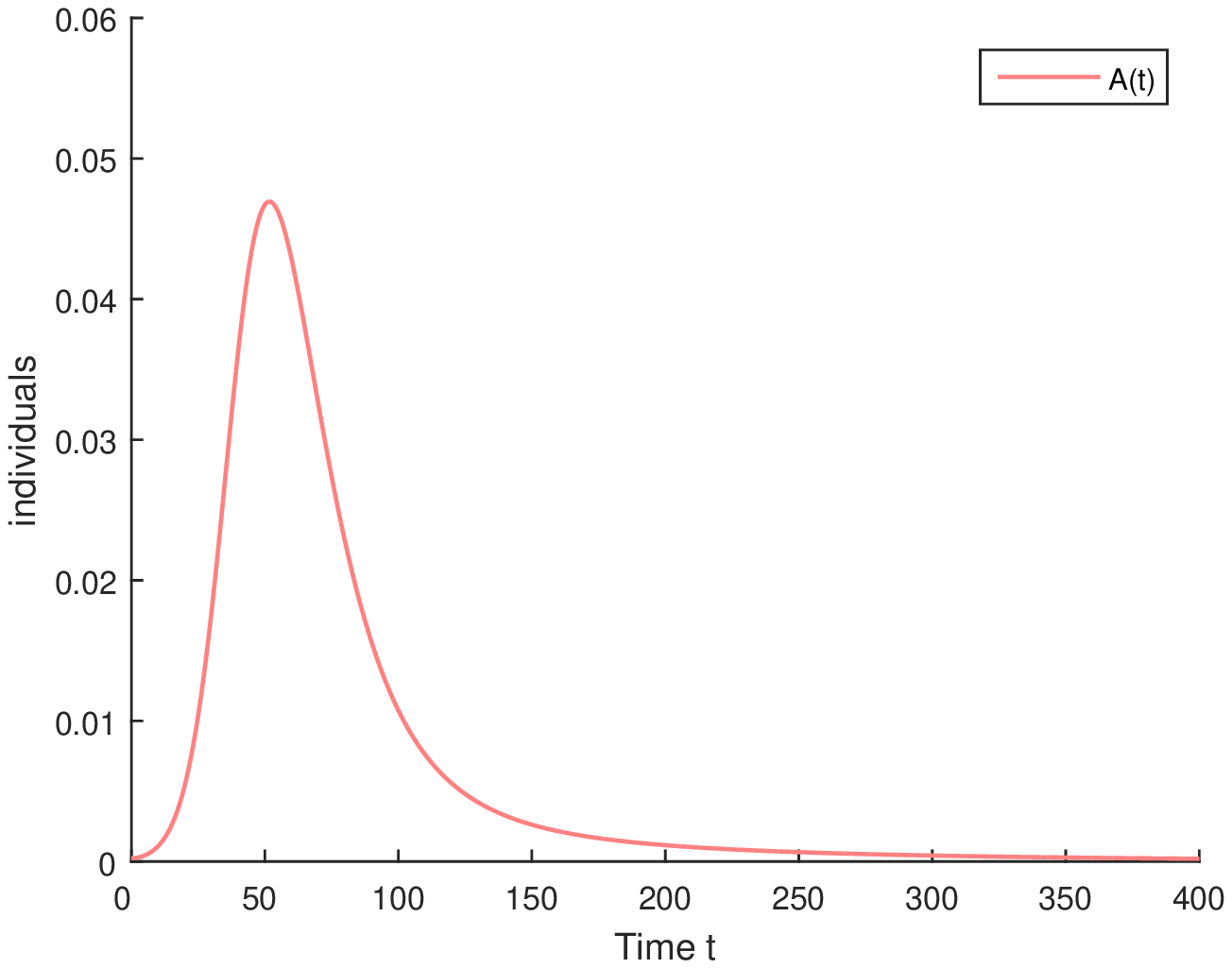}
  \centering
  }\\[-2pt]
    \subfigure{
    \includegraphics[width=.4\linewidth]{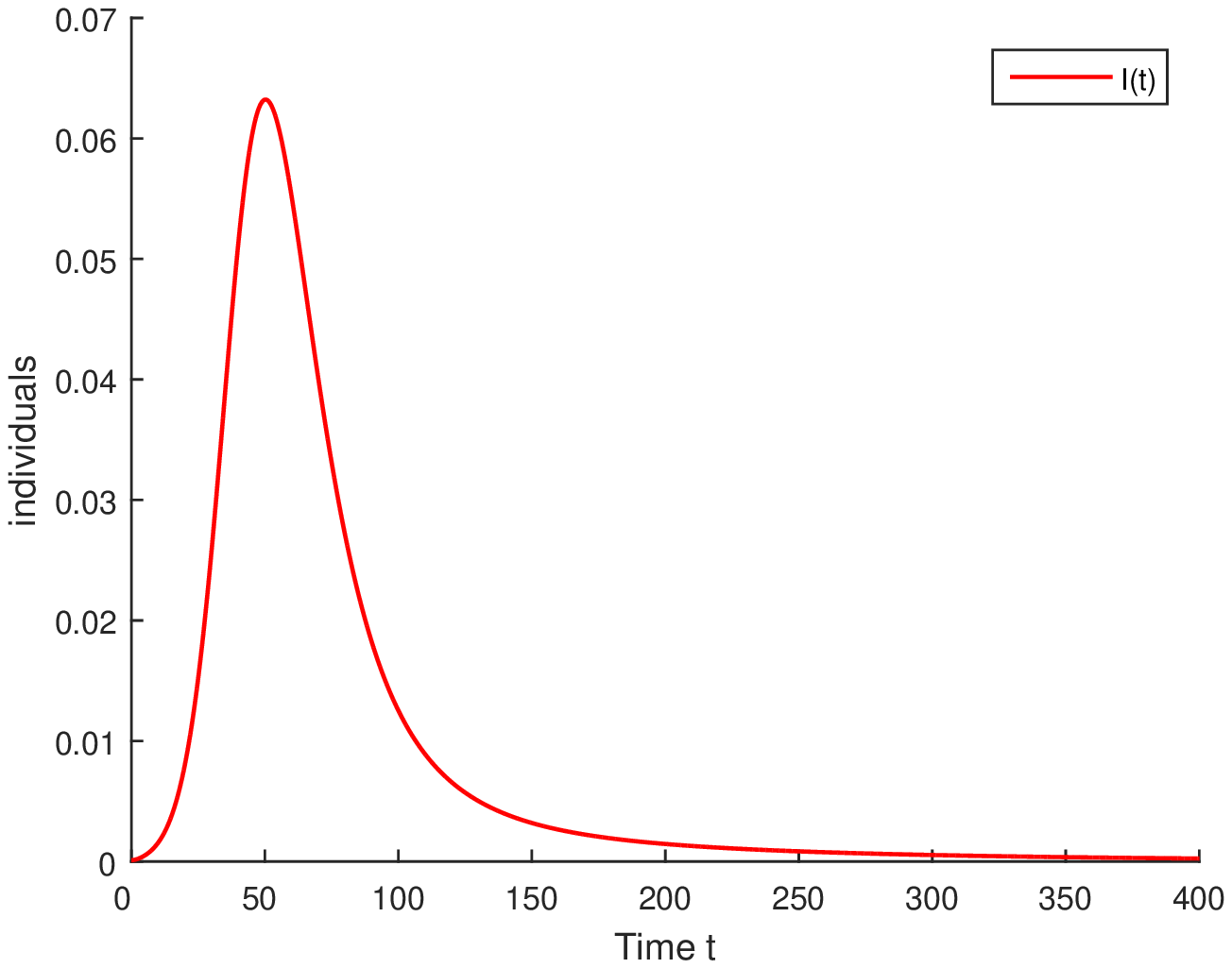}
  }%
\subfigure{
    \includegraphics[width=.4\linewidth]{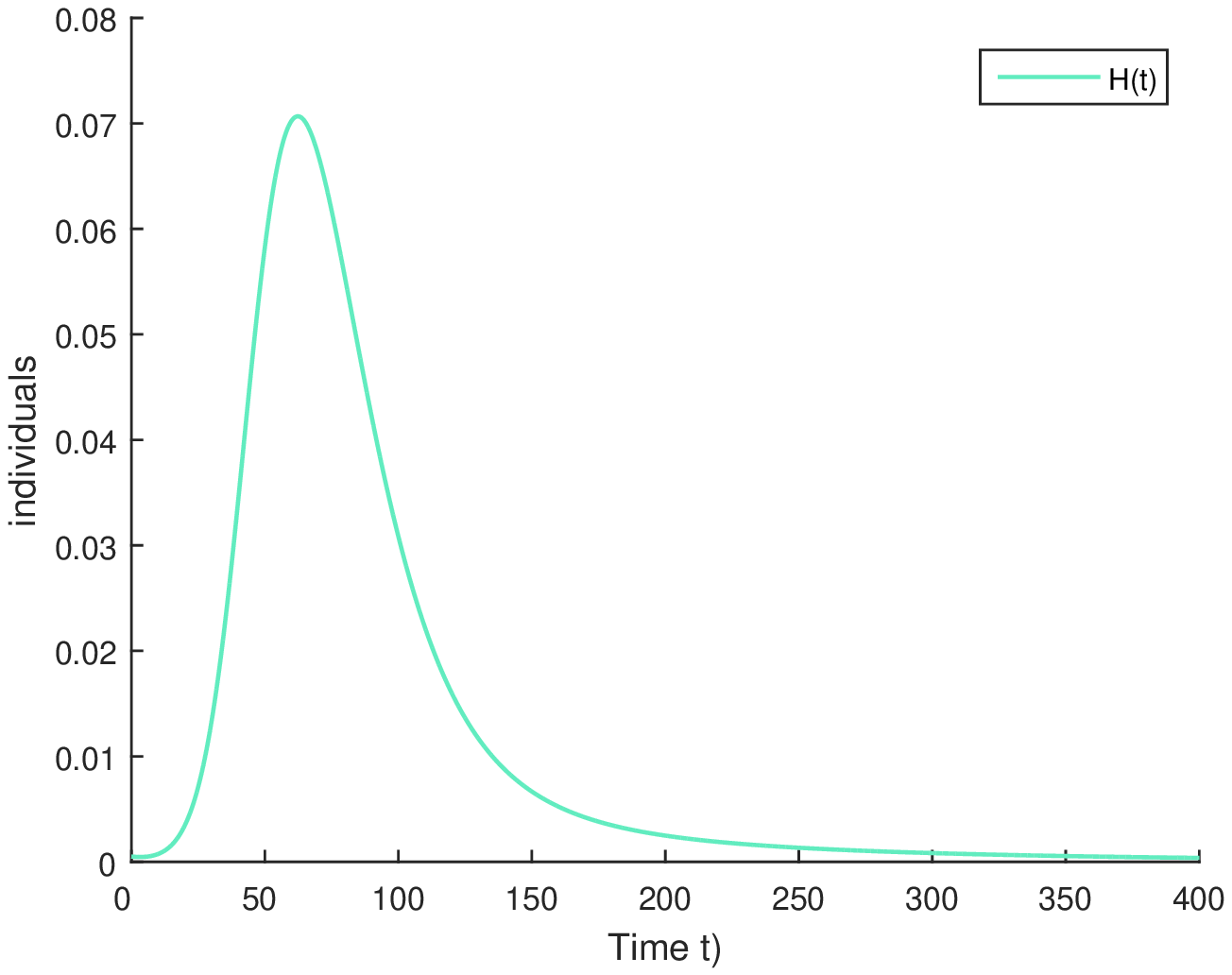}
  }\\[-2pt]
  \hspace*{2.8cm}\subfigure{
    \includegraphics[width=.4\linewidth]{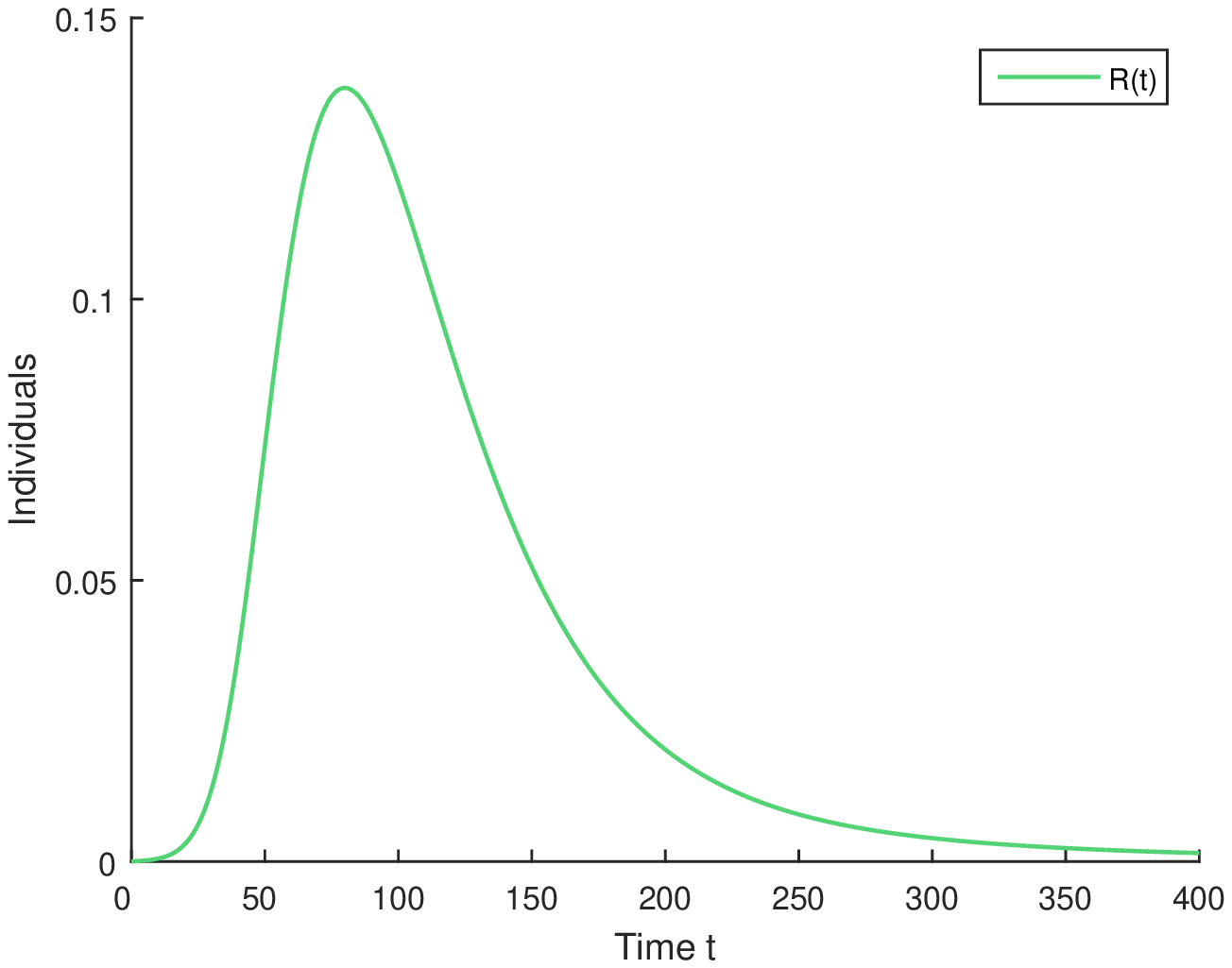}
  }\vspace*{-6pt}
 \caption{Solutions of COVID-19 deterministic model \eqref{detr} taking $\beta_1=3.97\times 10^{-6}$, $\beta_2=0.6\hspace{-2pt}\times\hspace{-2pt}\beta_1$  and $p=0.6201$ ($\mathcal{R}_0=0.9180<1$).}\label{Fig1}
\end{figure}
\vspace*{-1cm}
\begin{figure}[H]
\centering
\subfigure{
    \includegraphics[width=.4\linewidth]{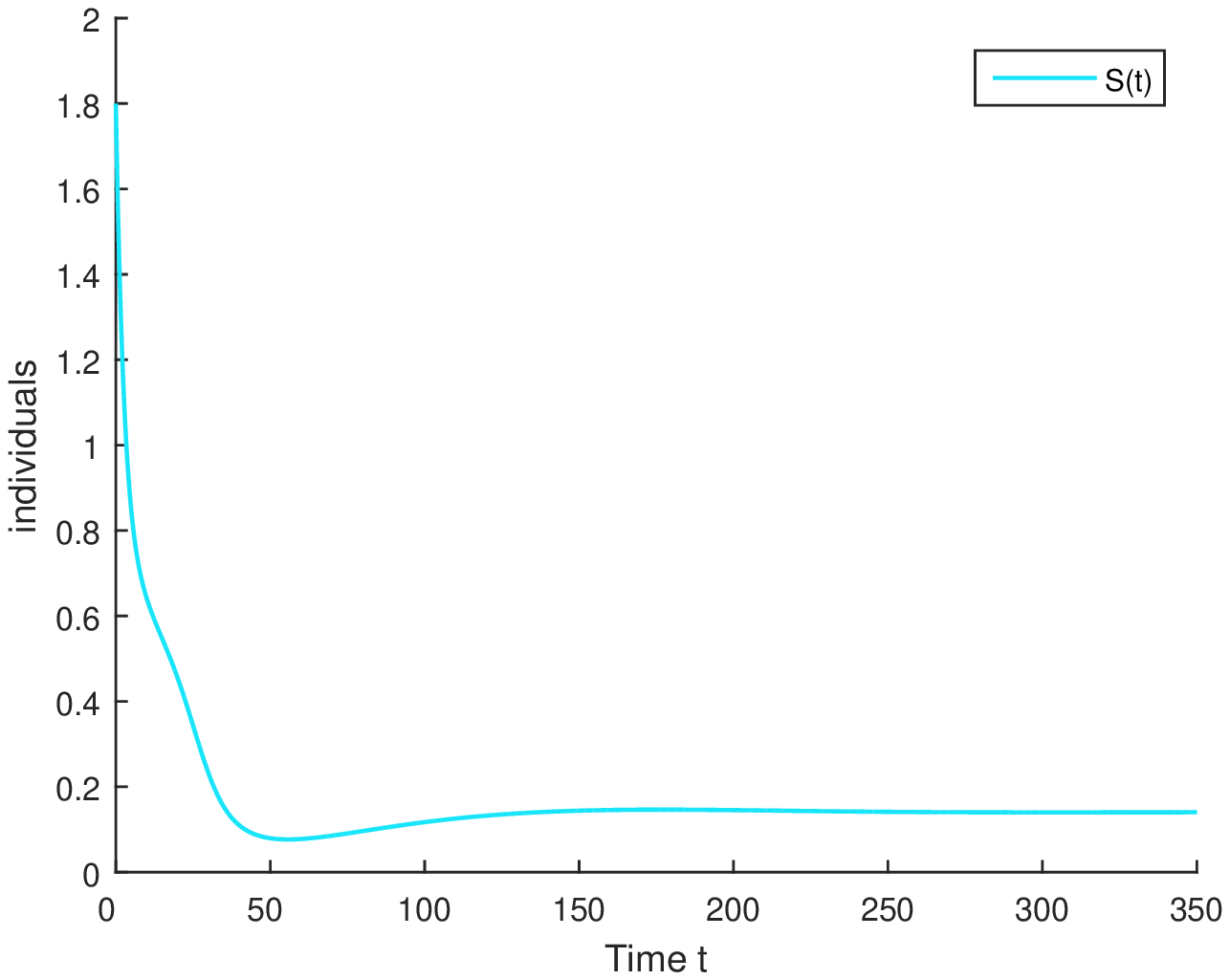}
  }%
\subfigure{
    \includegraphics[width=.4\linewidth]{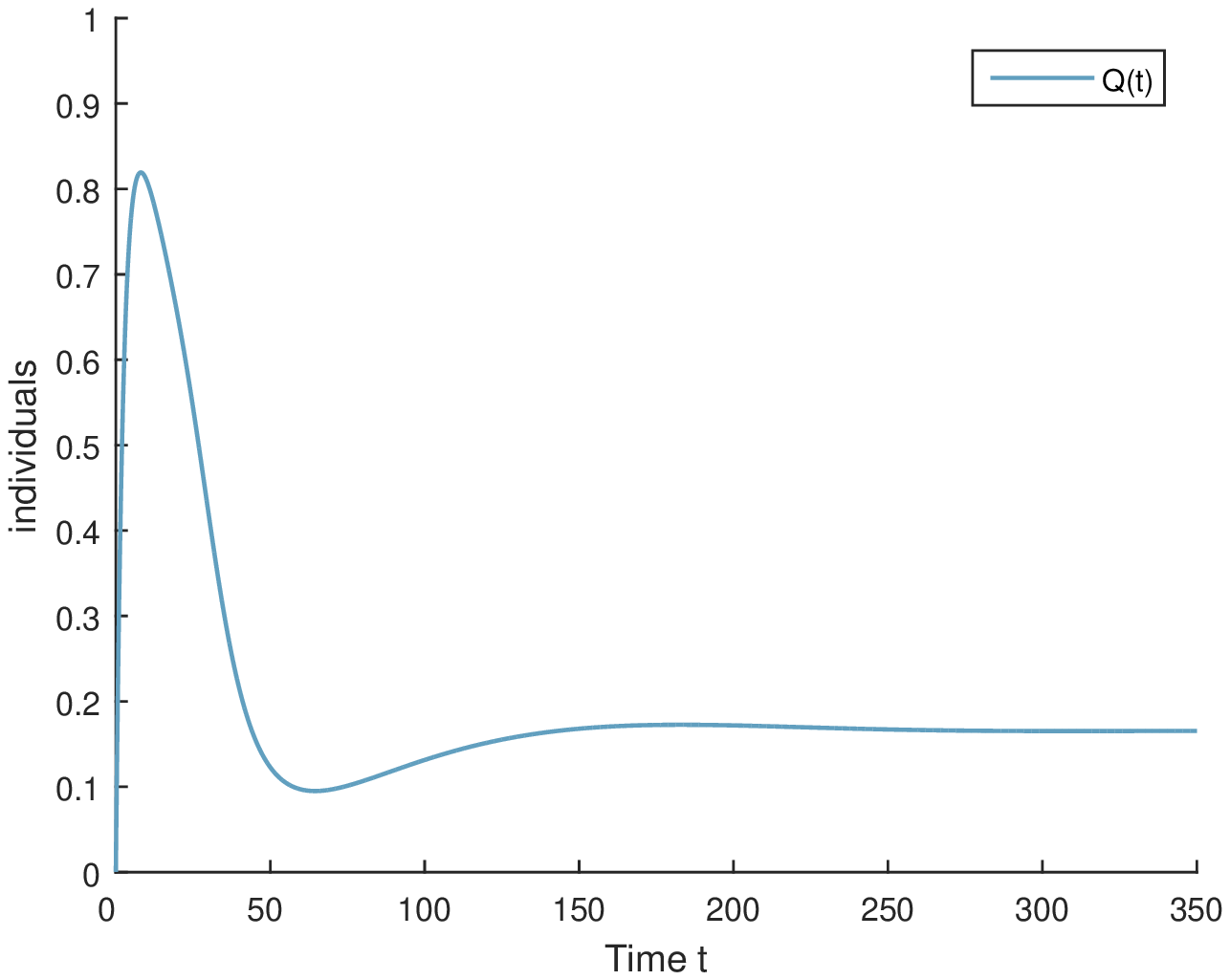}
  }\\[-2pt]
  \subfigure{
    \includegraphics[width=.4\linewidth]{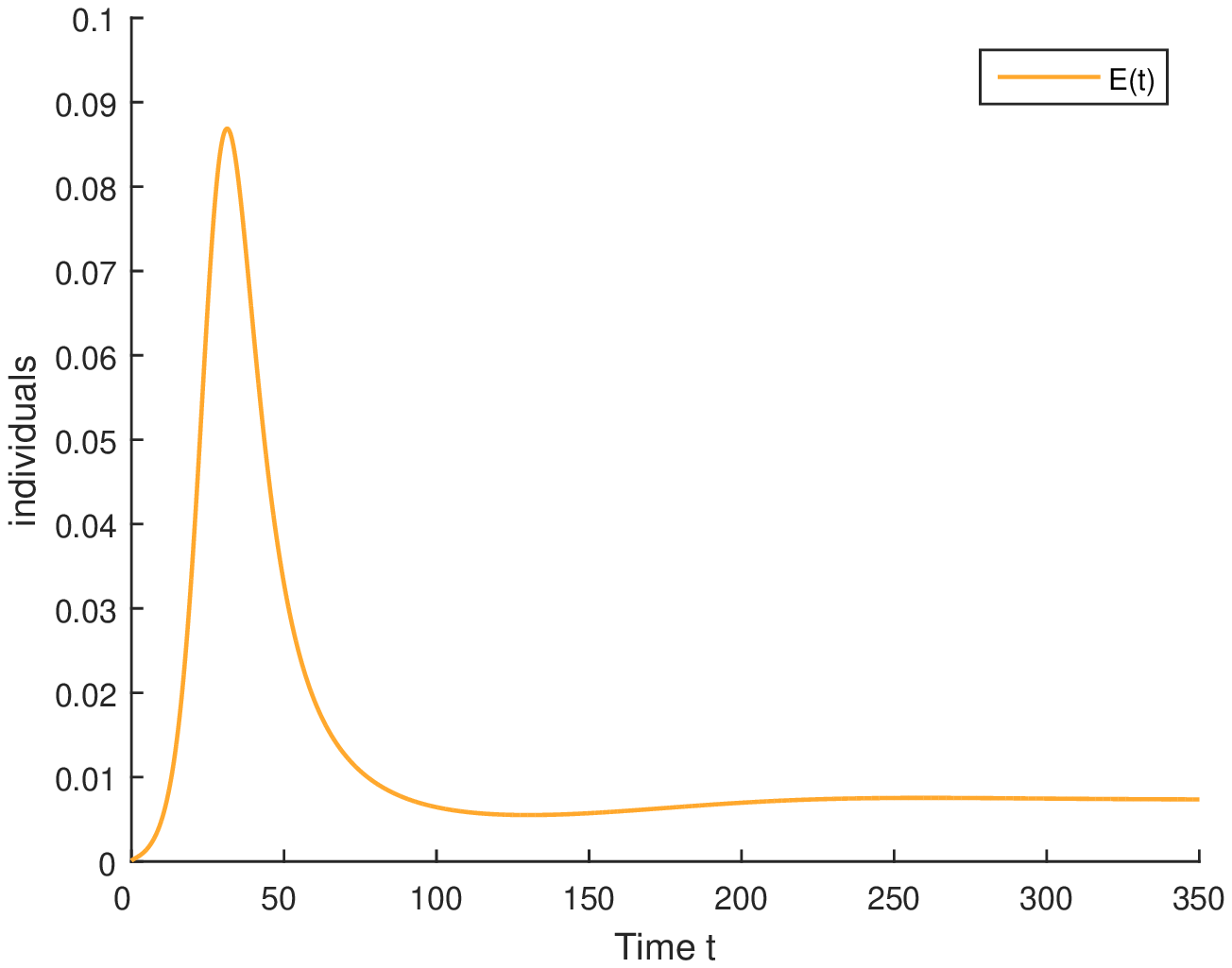}
  }%
\subfigure{
    \includegraphics[width=.4\linewidth]{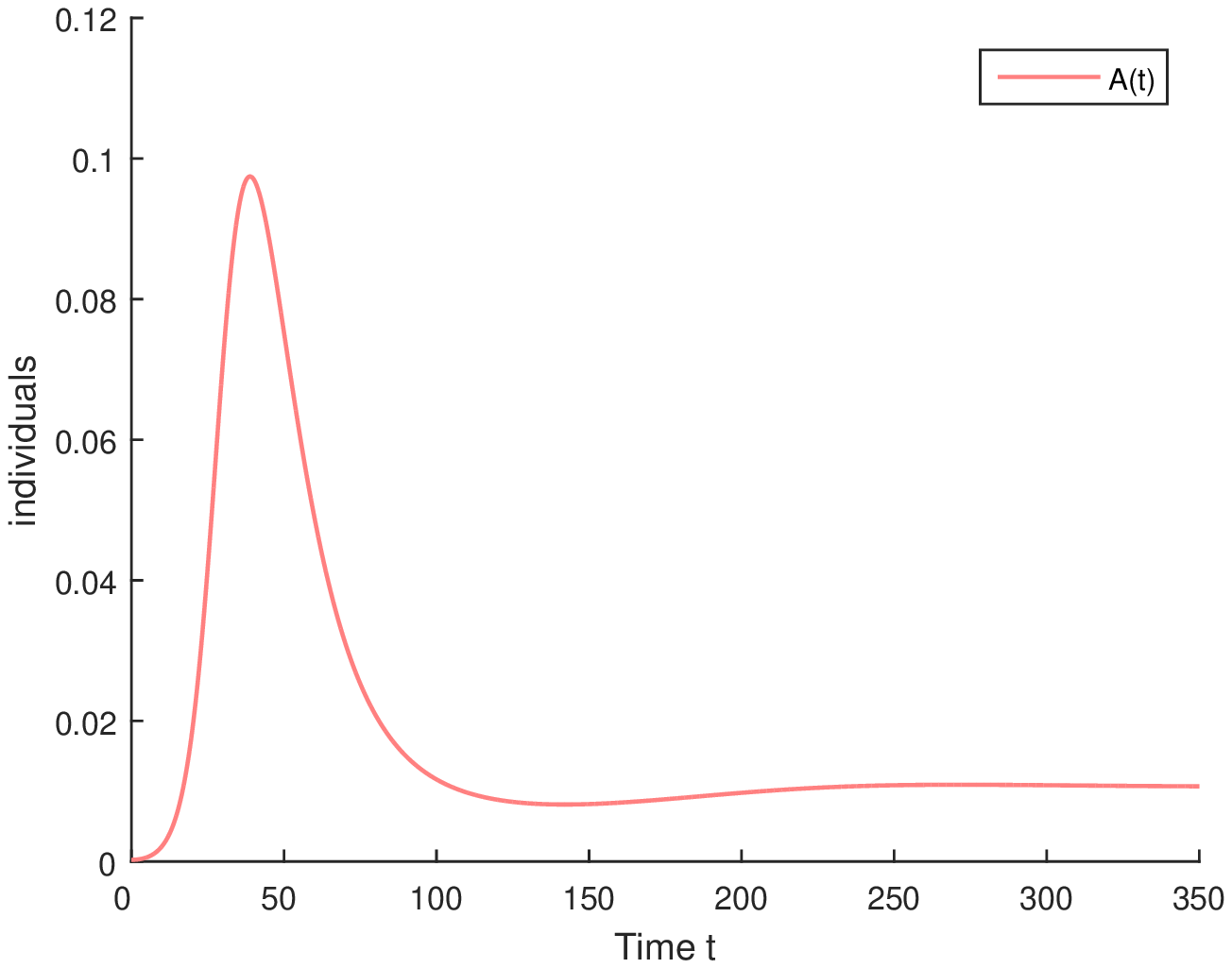}
  \centering
  }\\[-2pt]
    \subfigure{
    \includegraphics[width=.4\linewidth]{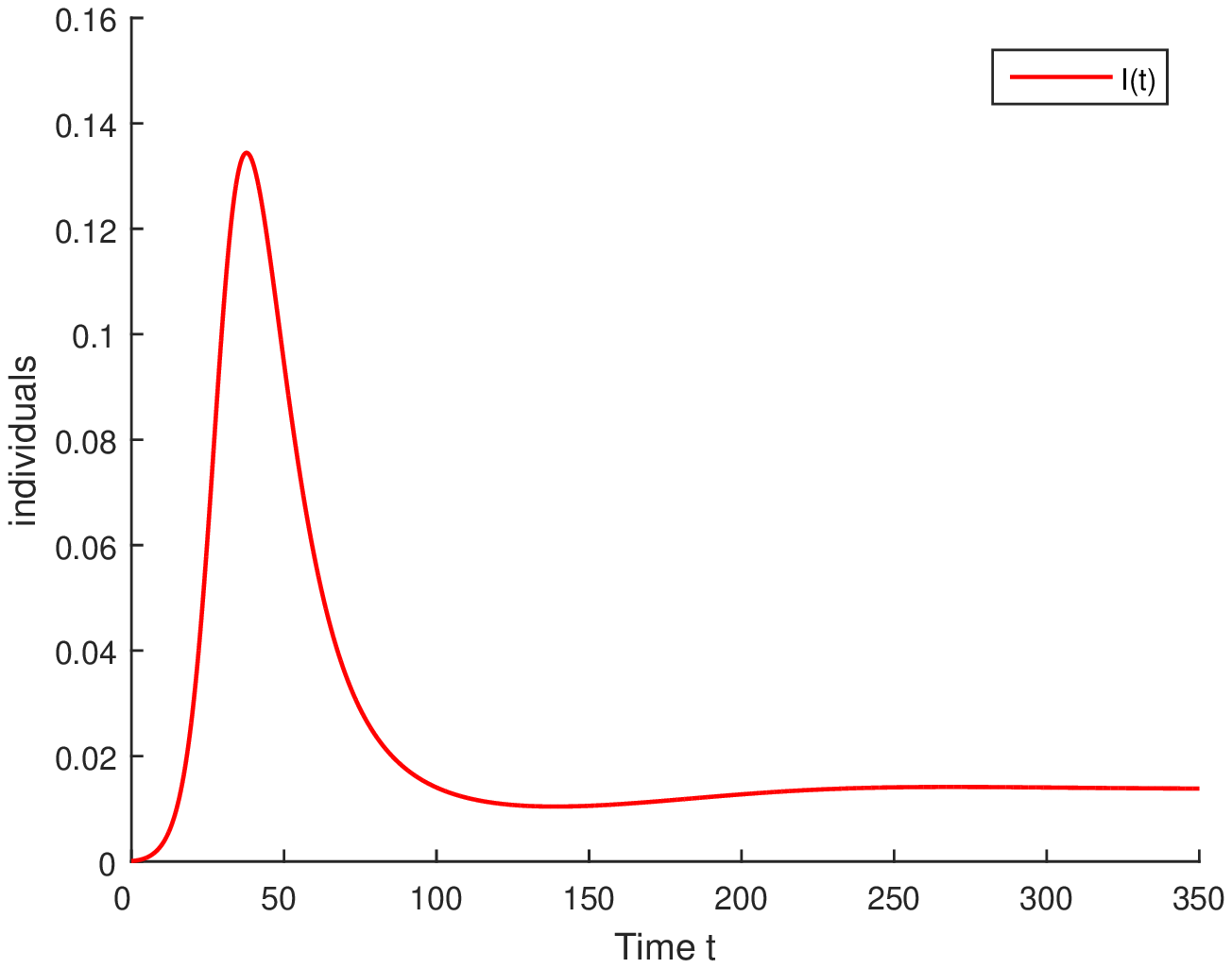}
  }%
\subfigure{
    \includegraphics[width=.4\linewidth]{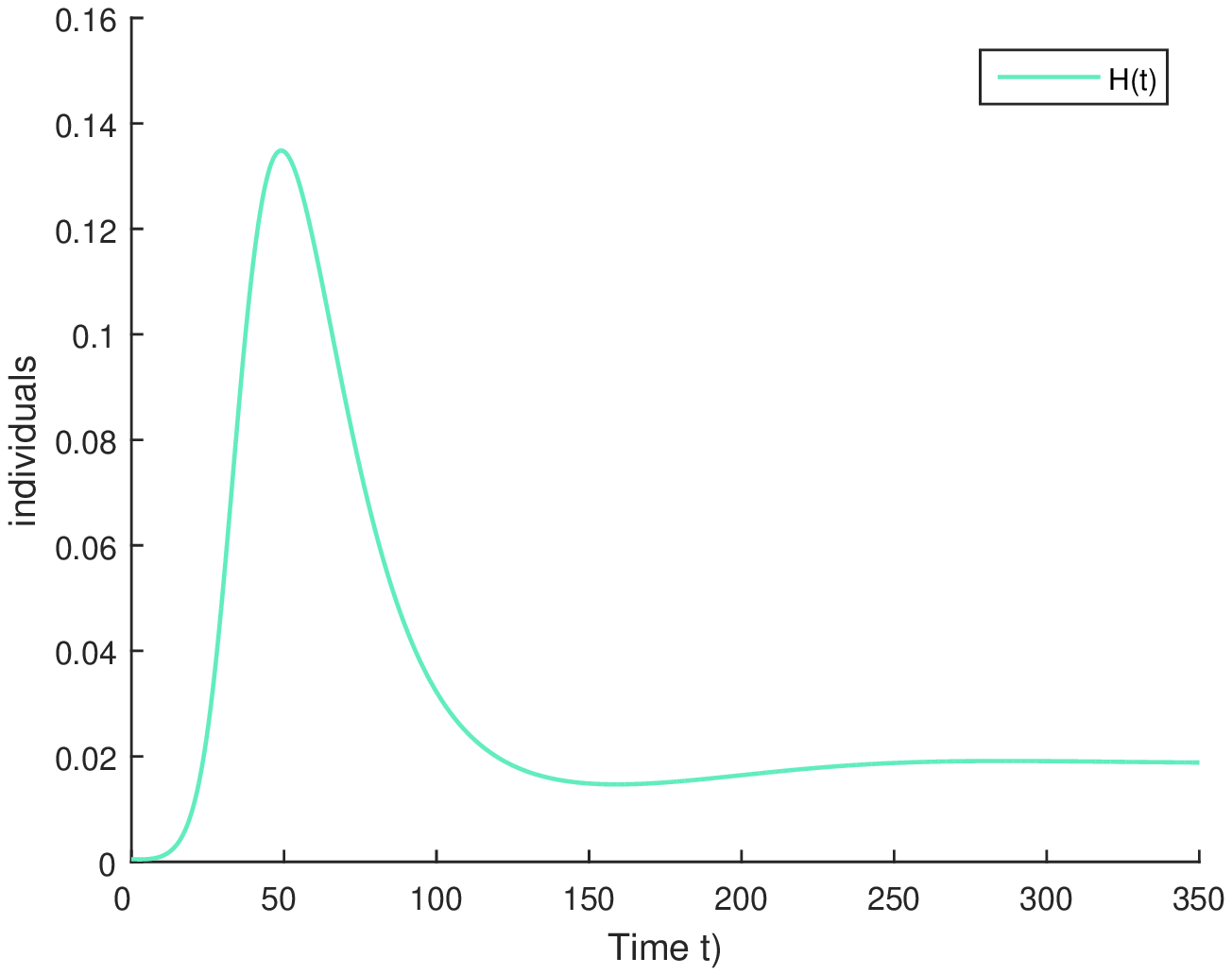}
  }\\[-2pt]
  \hspace*{2.8cm}\subfigure{
    \includegraphics[width=.4\linewidth]{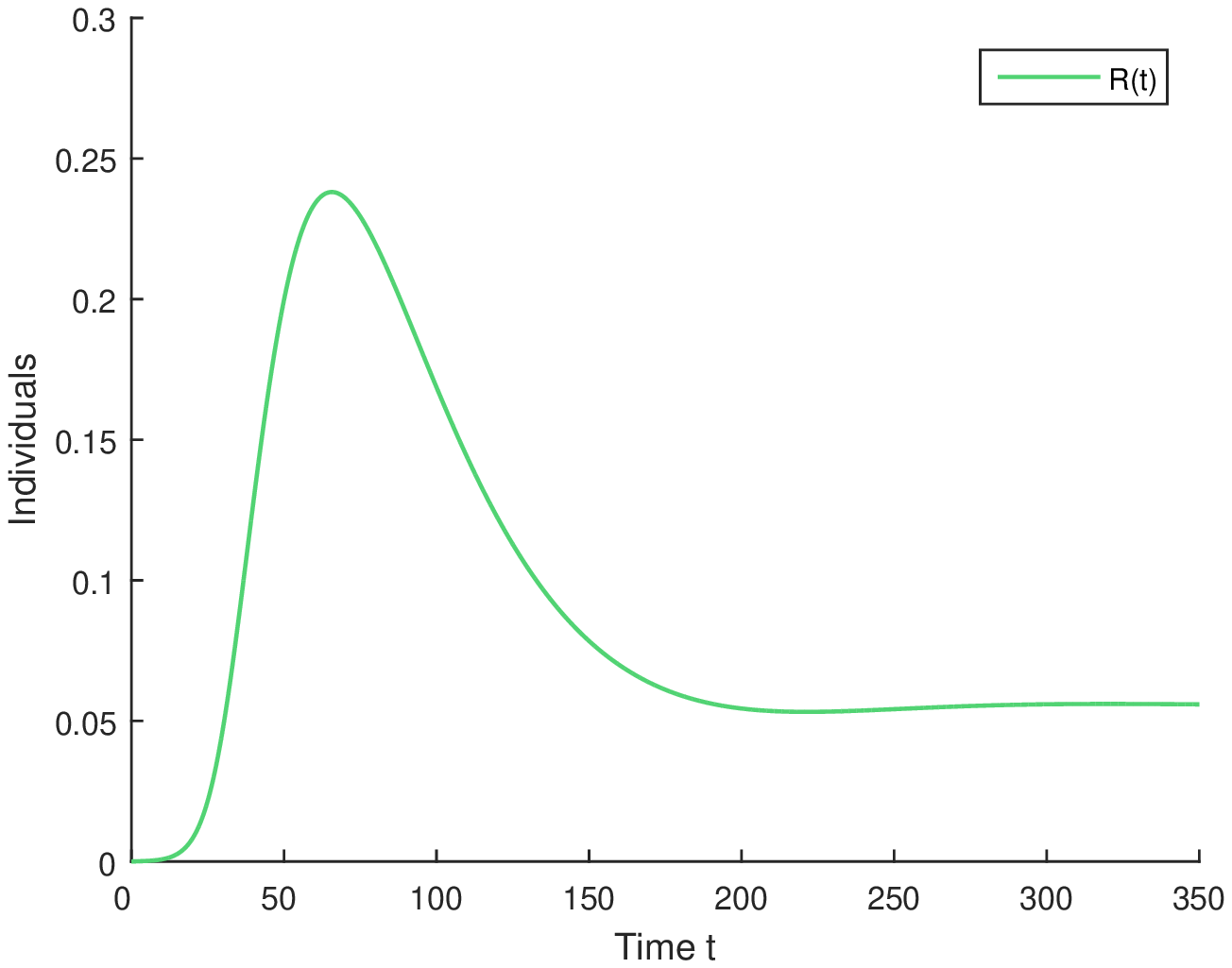}
  }\vspace*{-6pt}
 \caption{Trajectories of COVID-19 deterministic model \eqref{detr} taking $\beta_1=5\times 10^{-6}$, $\beta_2=0.6\hspace{-2pt}\times\hspace{-2pt}\beta_1$ and $p=0.6201$ ($\mathcal{R}_0=1.1562>1$).}\label{Fig2}
\end{figure}
\begin{figure}[H]
\centering
\subfigure{
    \includegraphics[width=.4\linewidth]{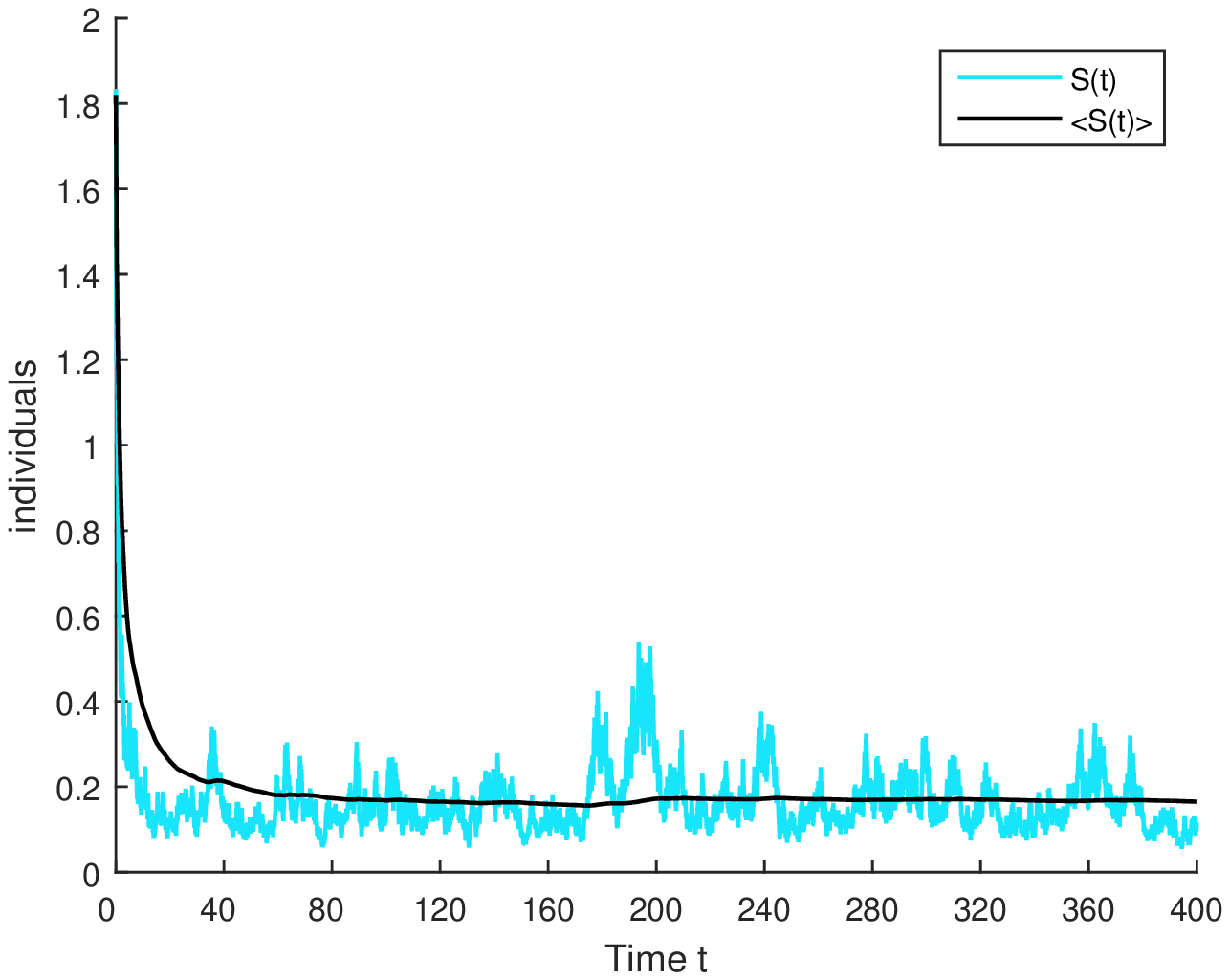}
  }%
\subfigure{
    \includegraphics[width=.4\linewidth]{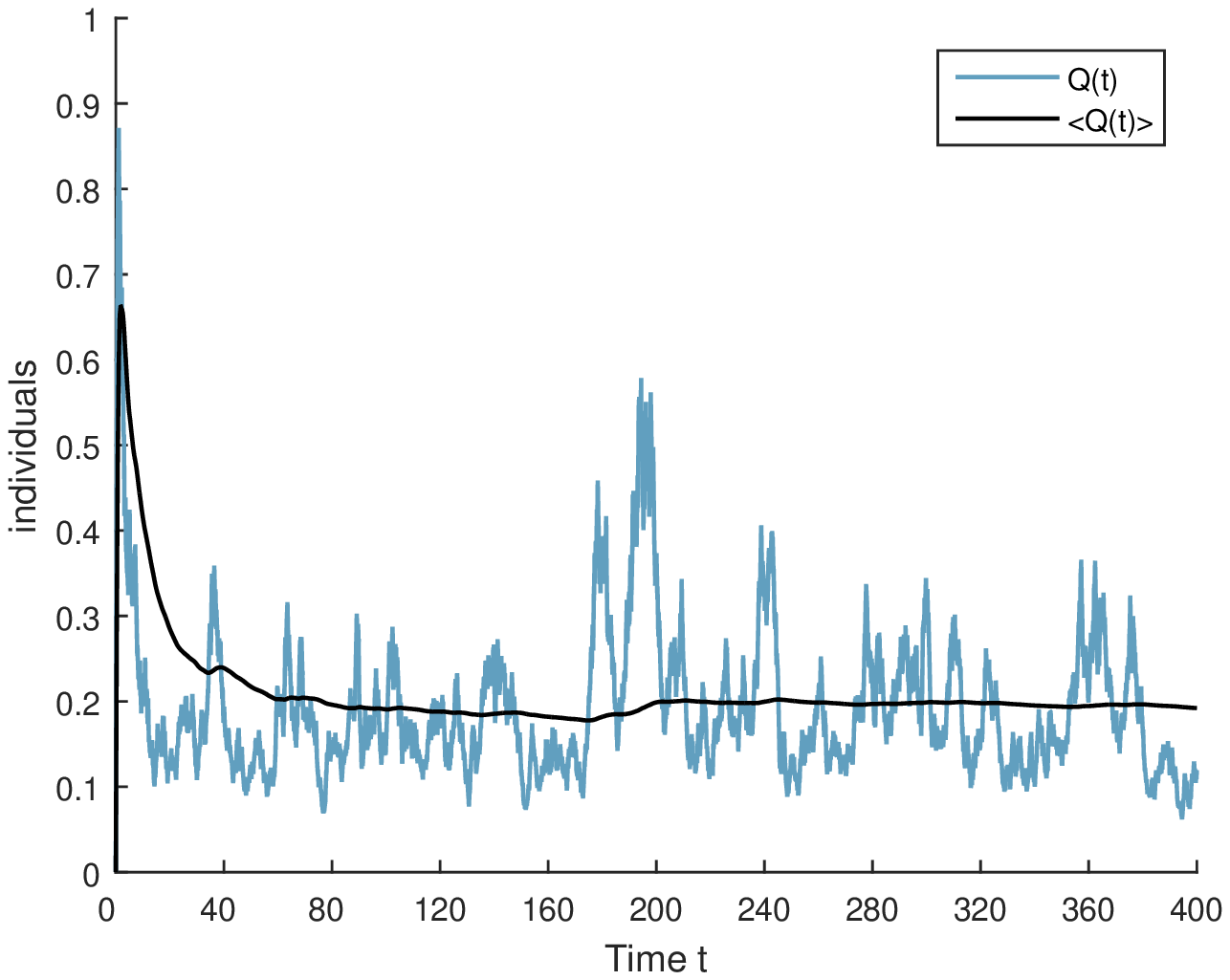}
  }\\[-2pt]
  \subfigure{
    \includegraphics[width=.4\linewidth]{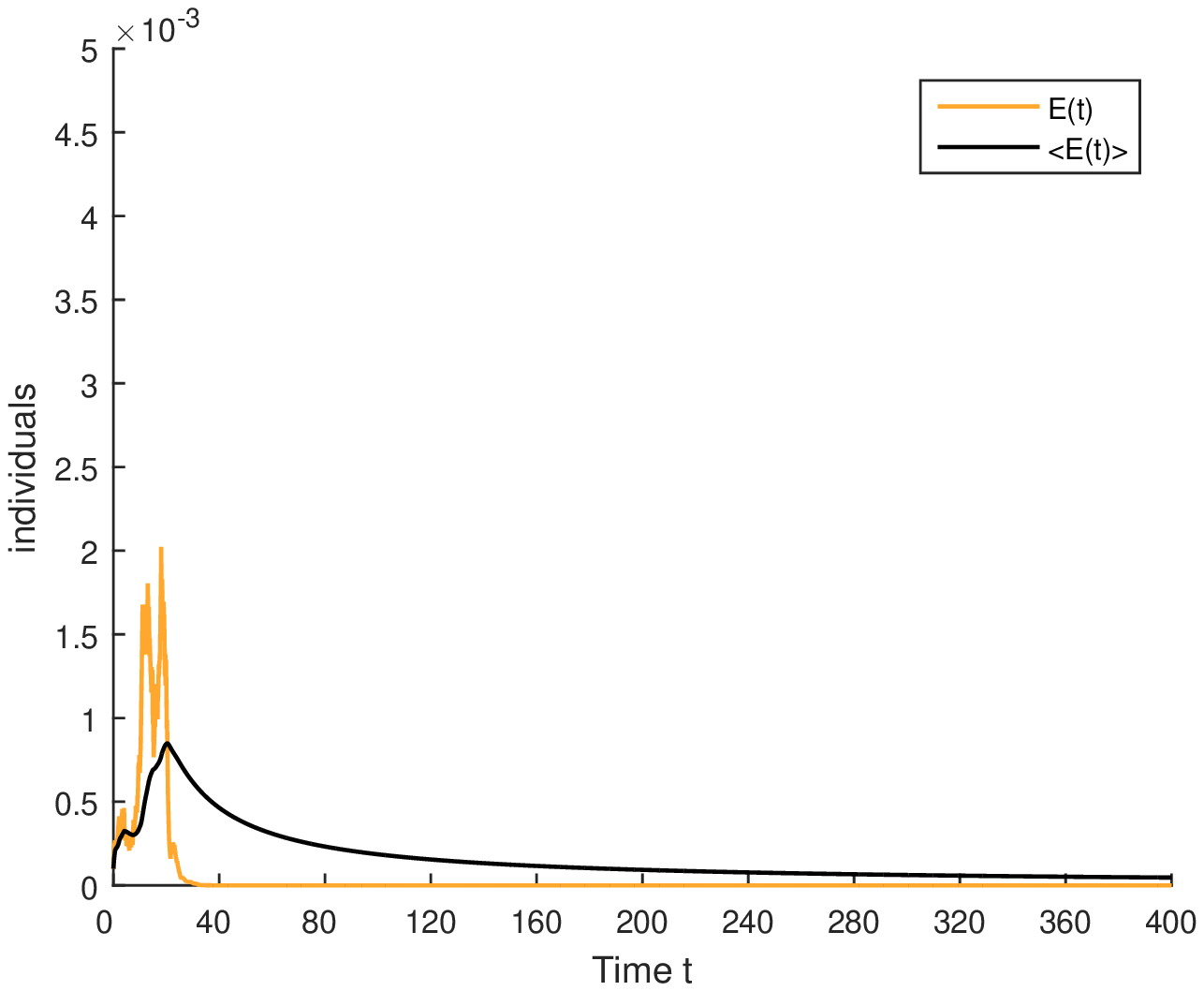}
  }%
\subfigure{
    \includegraphics[width=.4\linewidth]{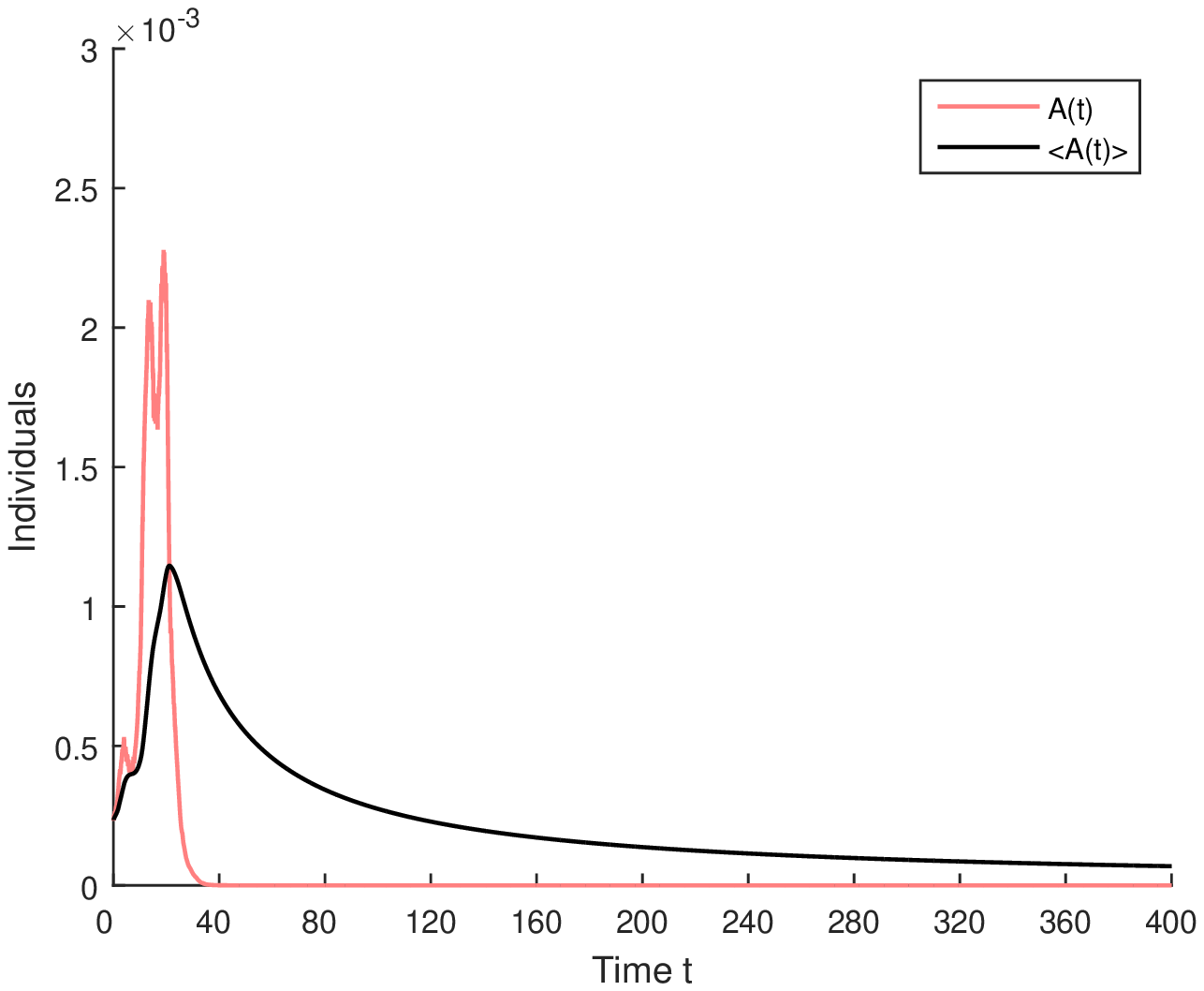}
  \centering
  }\\[-2pt]
    \subfigure{
    \includegraphics[width=.4\linewidth]{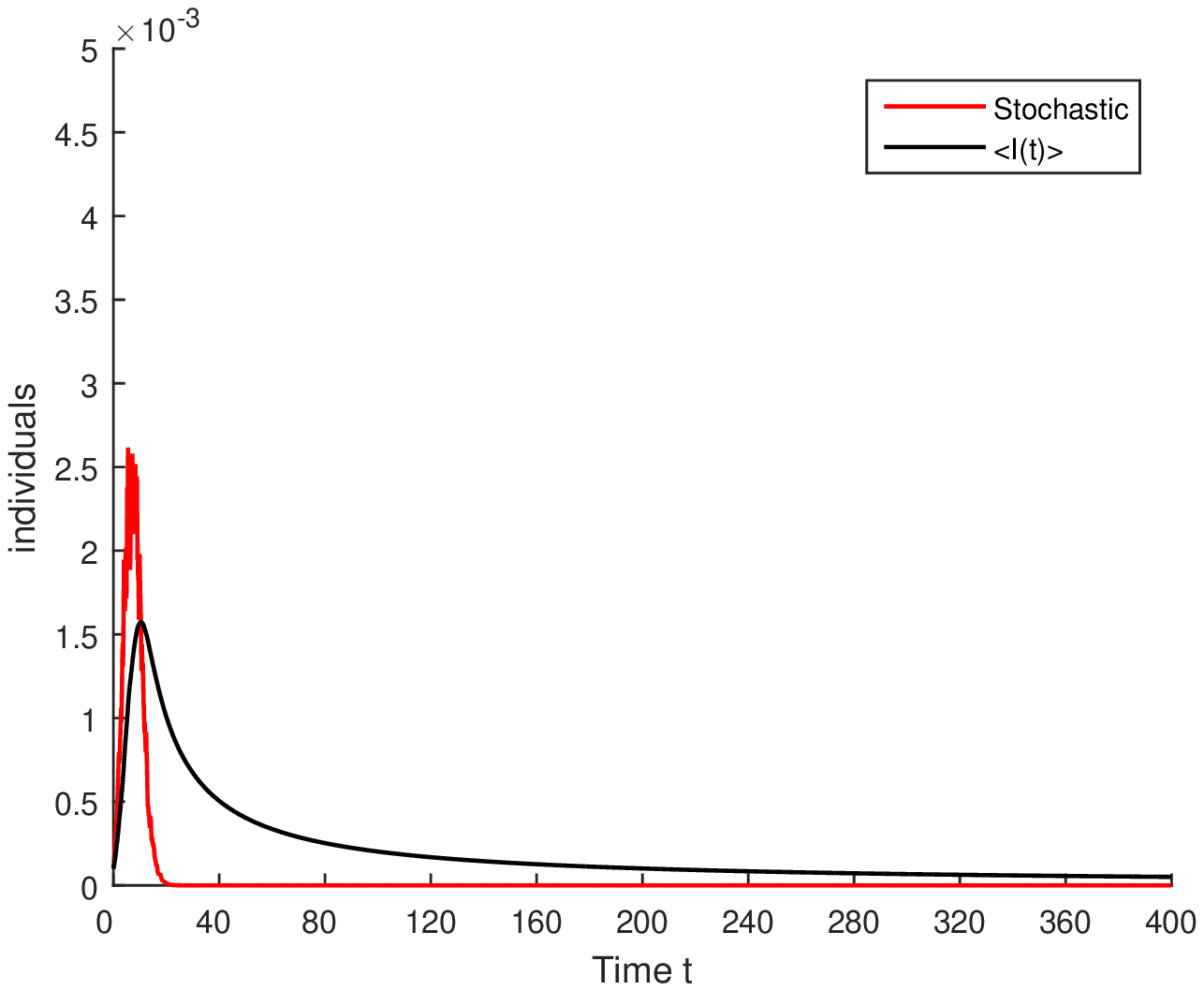}
  }%
\subfigure{
    \includegraphics[width=.4\linewidth]{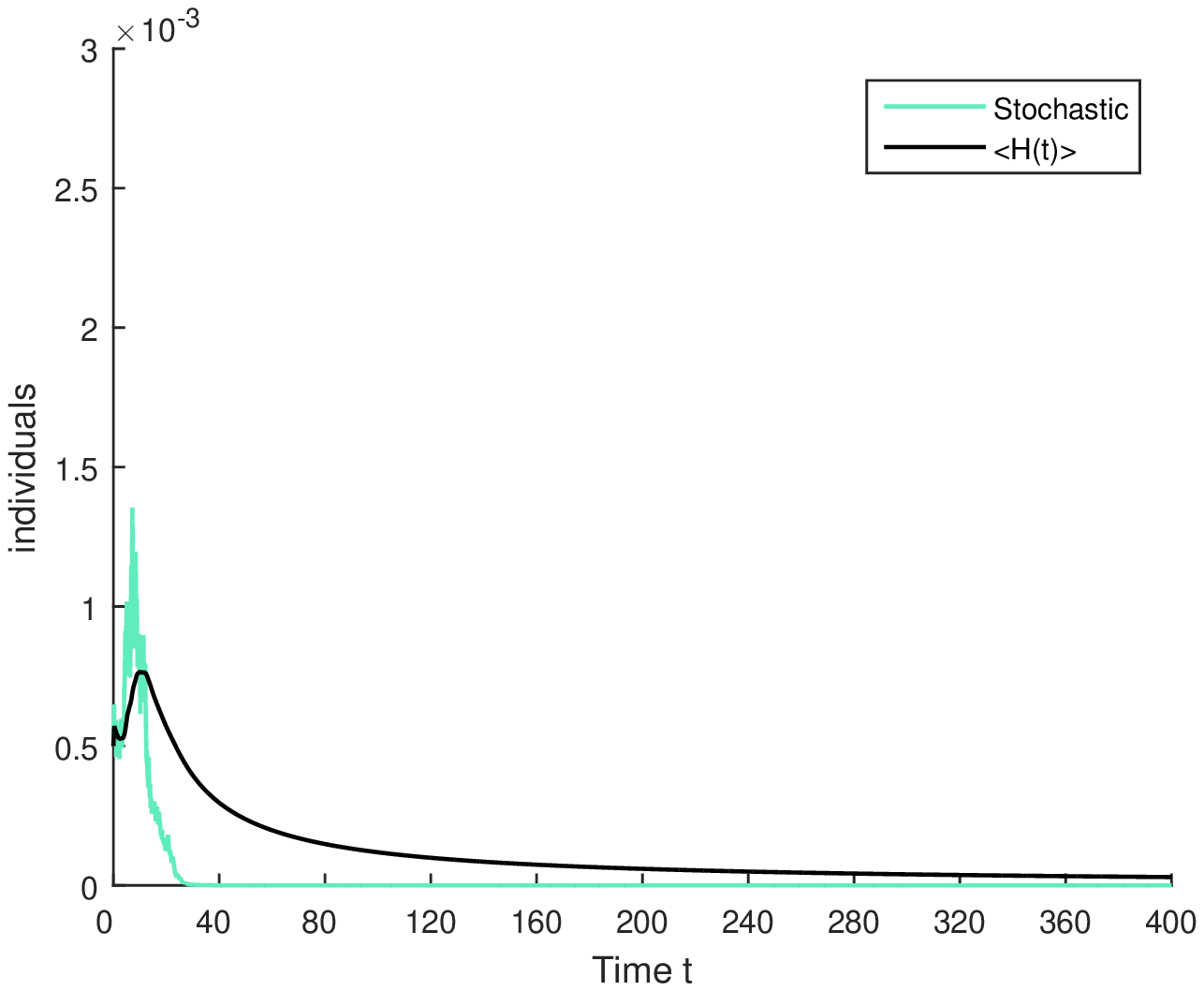}
  }\\[-2pt]
  \hspace*{2.8cm}\subfigure{
    \includegraphics[width=.4\linewidth]{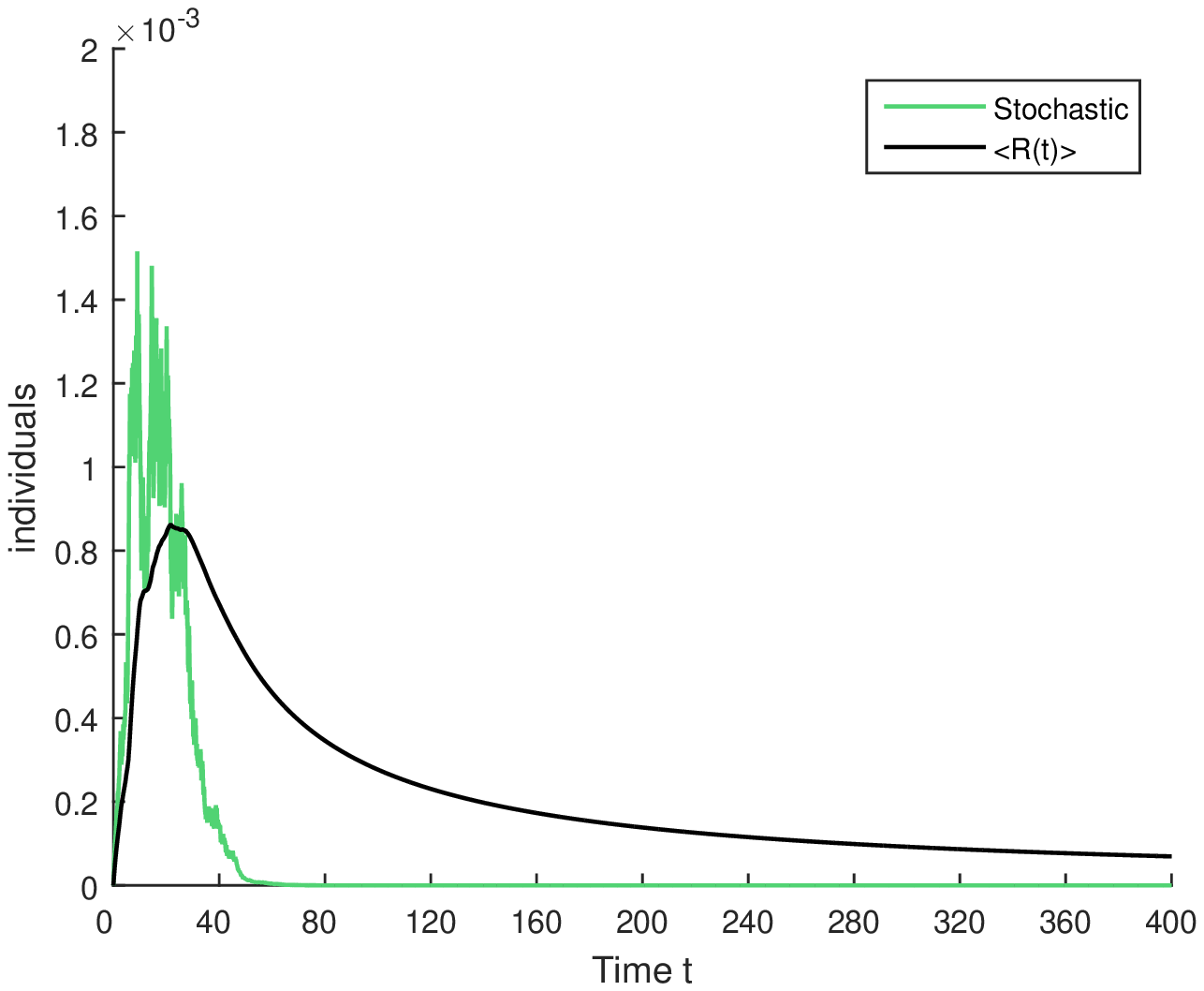}
  }\vspace*{-6pt}
 \caption{Trajectories of COVID-19 stochastic model \eqref{systo} taking $\beta_1=2.08\times 10^{-9}$, $\beta_2=0.6\hspace{-2pt}\times\hspace{-2pt}\beta_1$ and $p=0.6201$.}
 \label{Fig3}
\end{figure}
\begin{figure}[H]
\centering
\subfigure{
    \includegraphics[width=.4\linewidth]{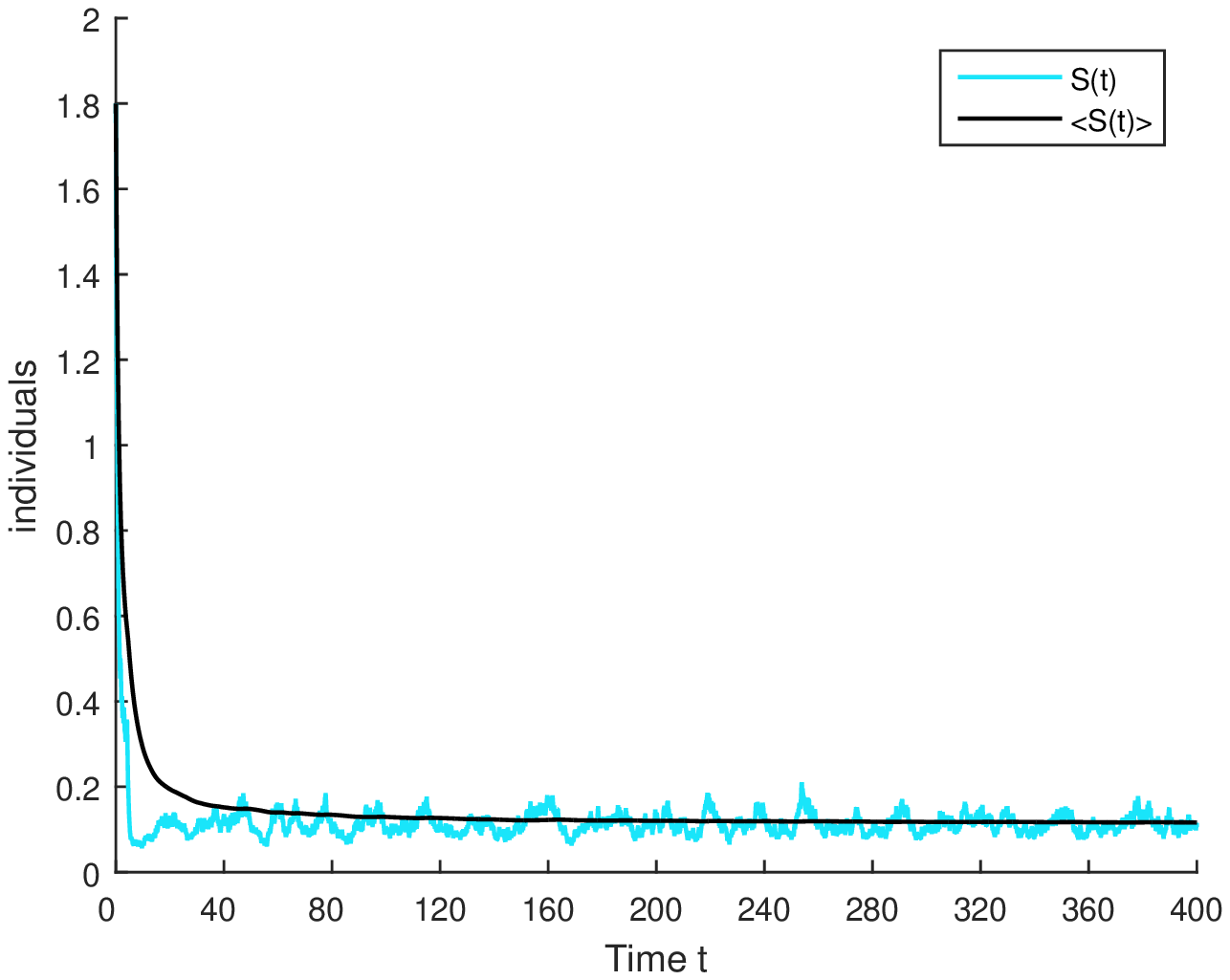}
  }%
\subfigure{
    \includegraphics[width=.4\linewidth]{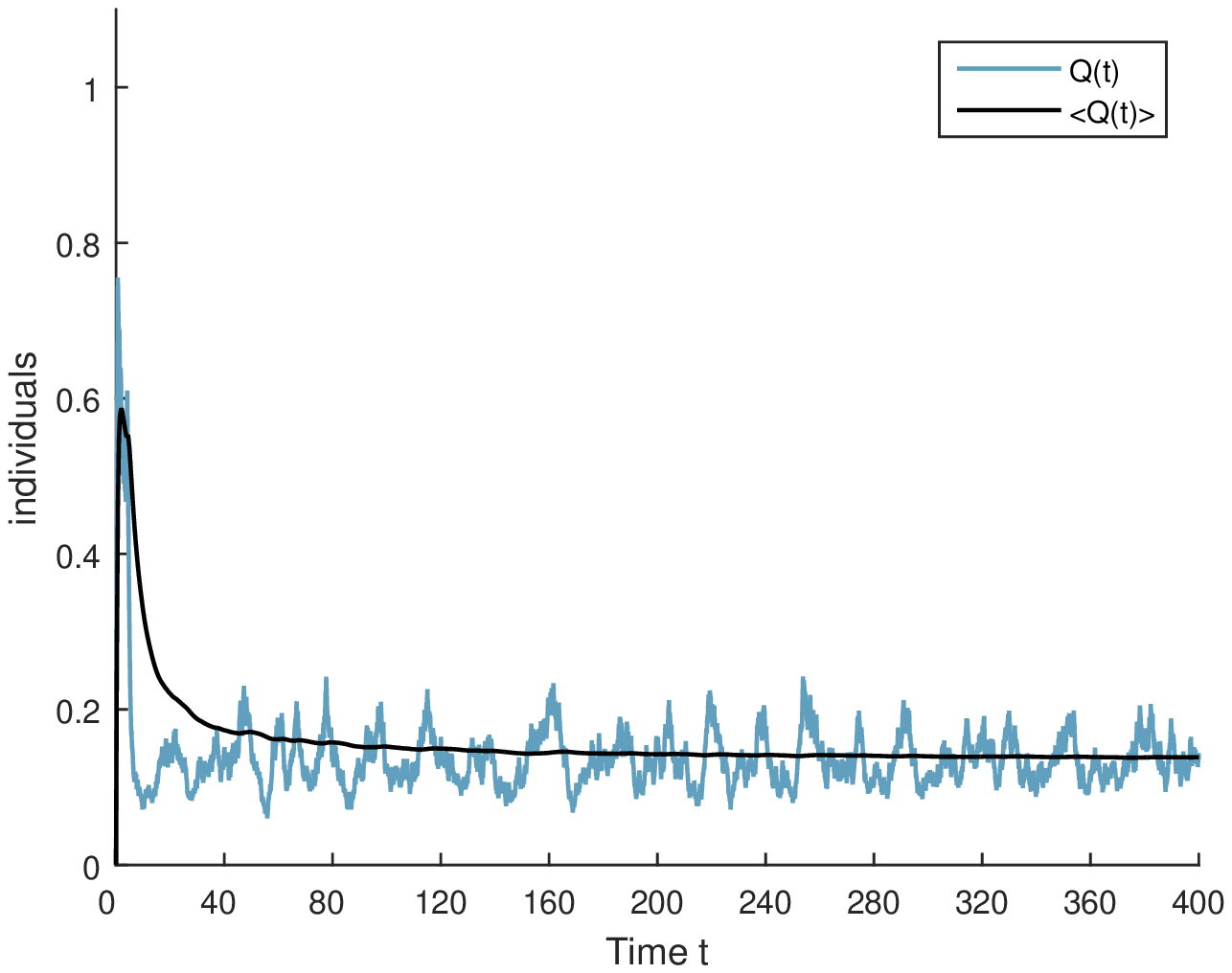}
  }\\[-2pt]
  \subfigure{
    \includegraphics[width=.4\linewidth]{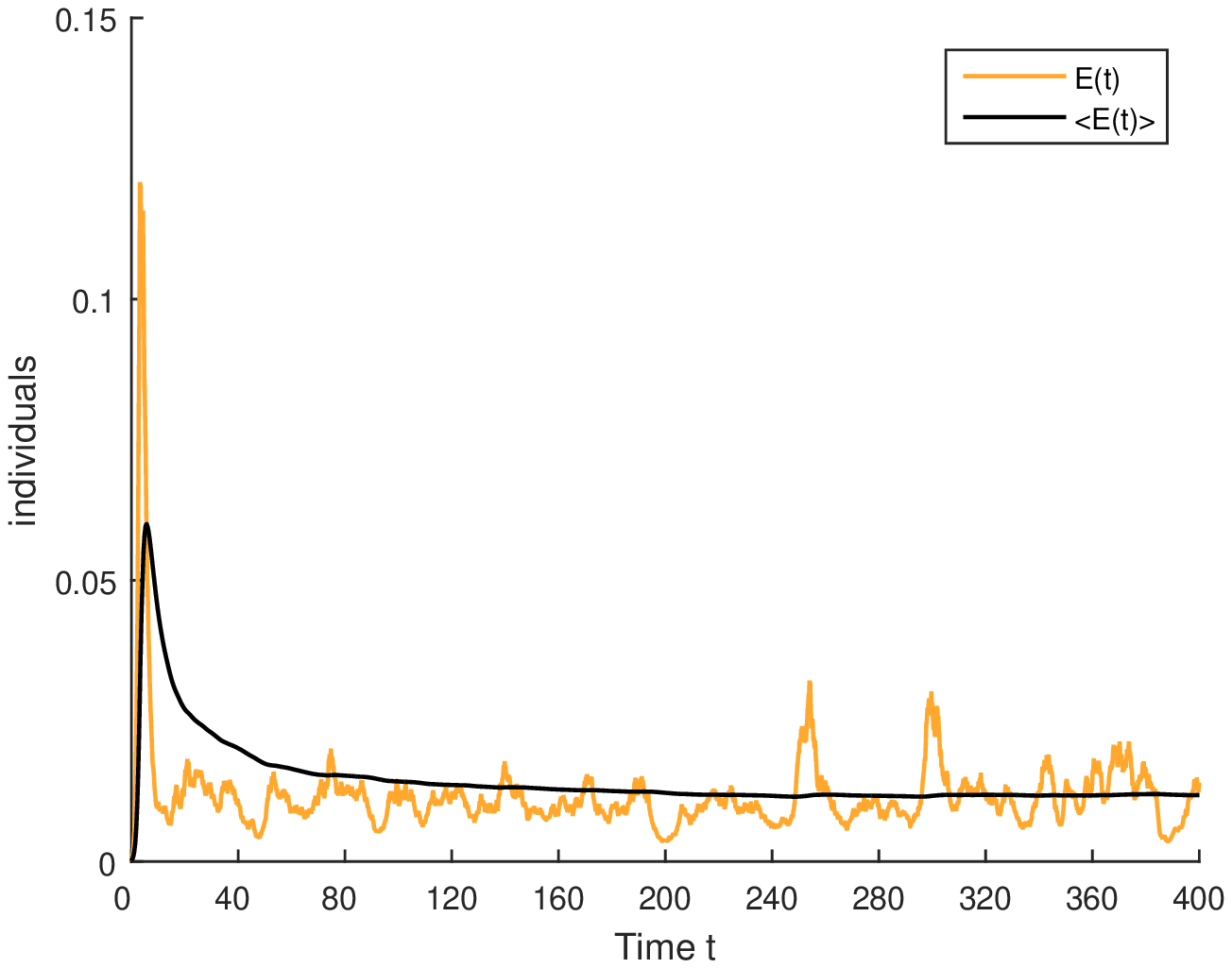}
  }%
\subfigure{
    \includegraphics[width=.4\linewidth]{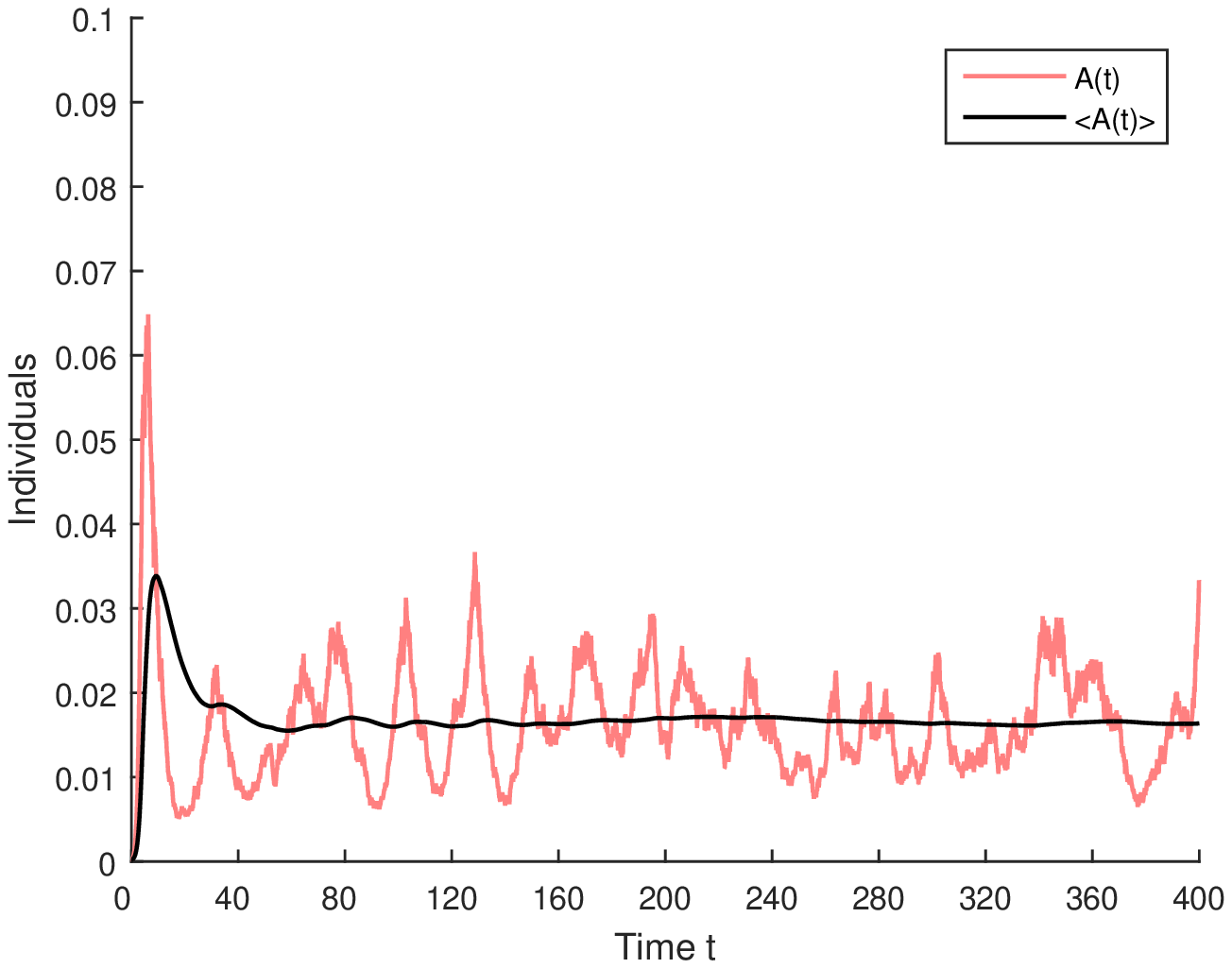}
  \centering
  }\\[-2pt]
    \subfigure{
    \includegraphics[width=.4\linewidth]{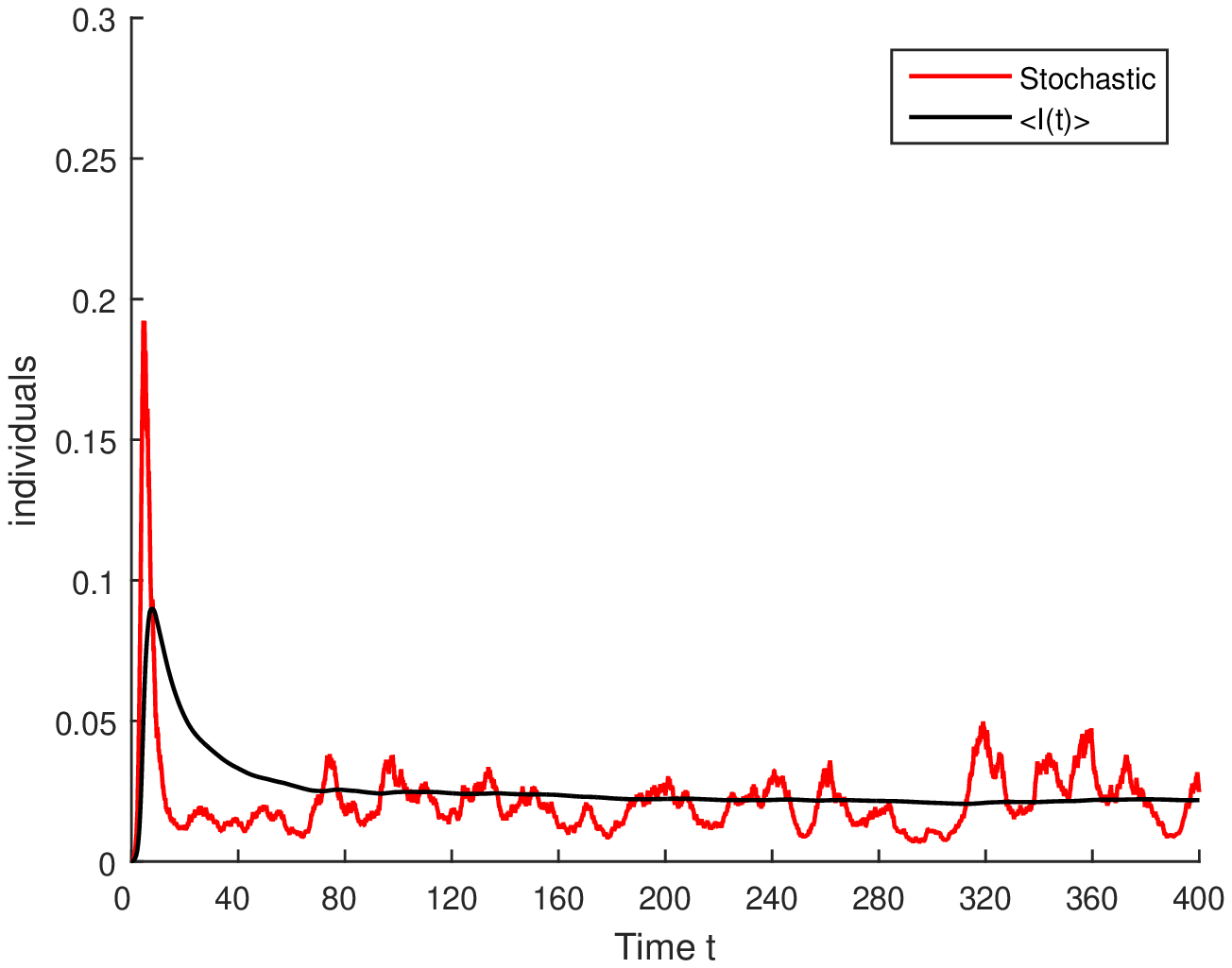}
  }%
\subfigure{
    \includegraphics[width=.4\linewidth]{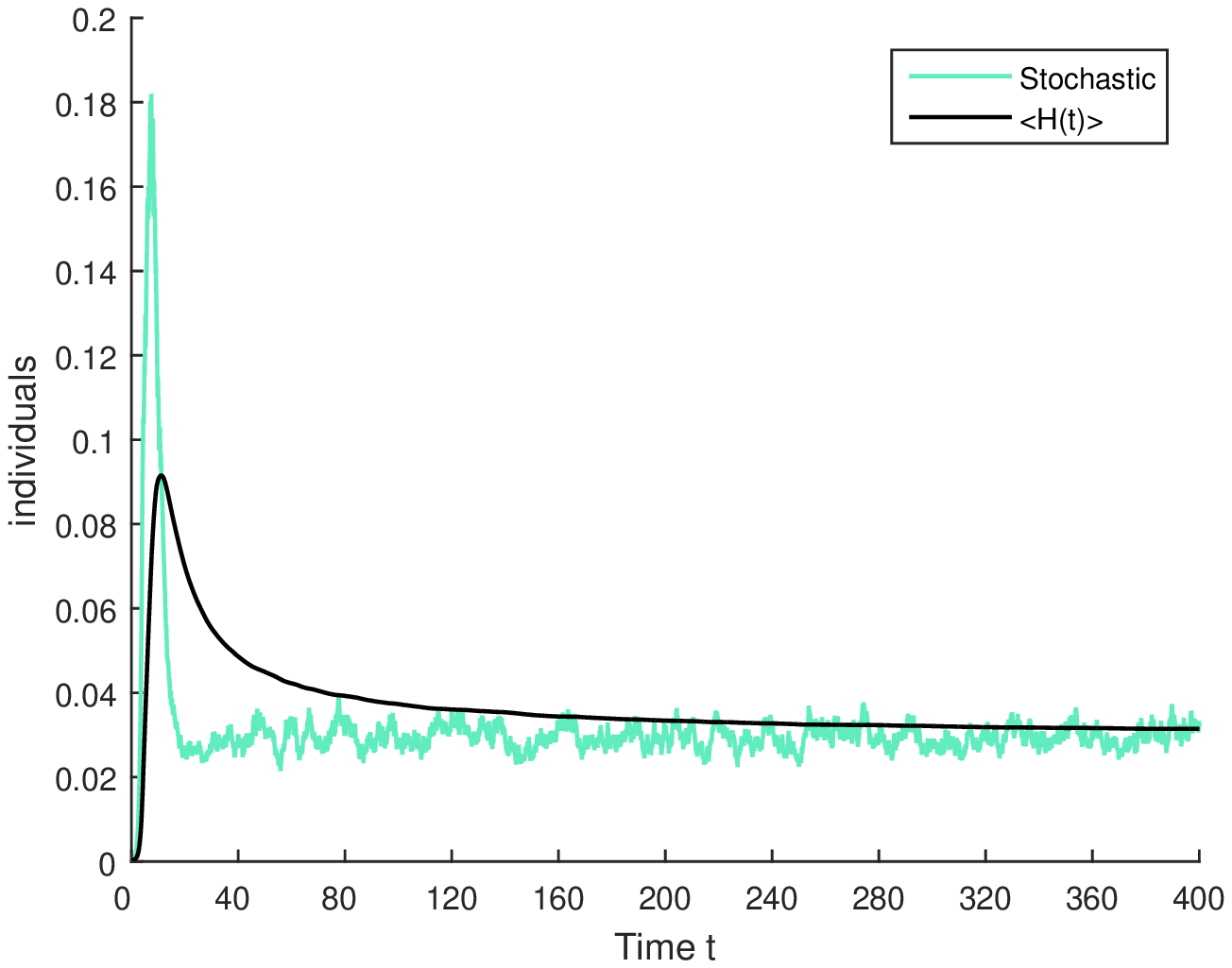}
  }\\[-2pt]
  \hspace*{2.8cm}\subfigure{
    \includegraphics[width=.4\linewidth]{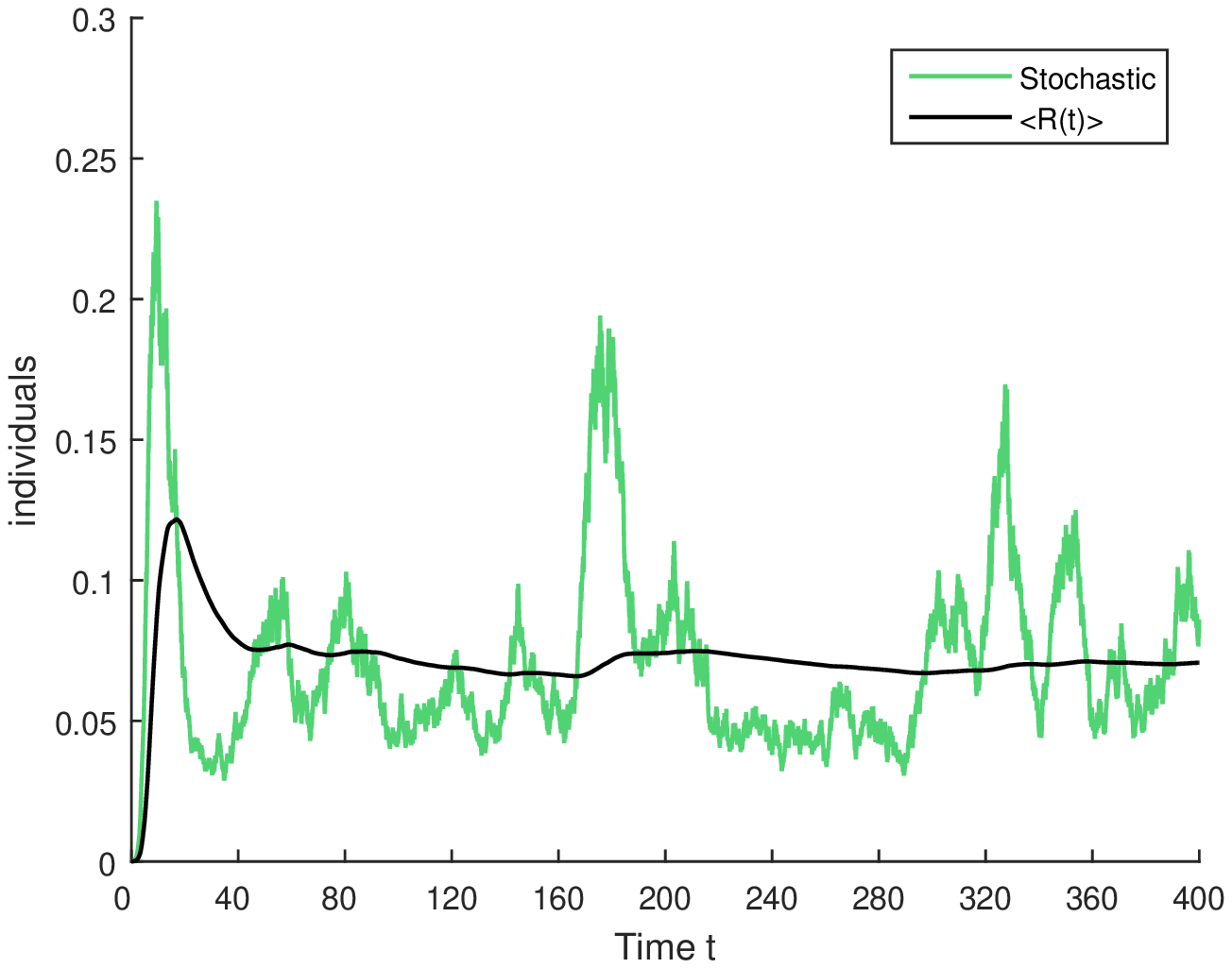}
  }\vspace*{-6pt}
 \caption{Trajectories of COVID-19 stochastic model \eqref{systo} taking $\beta_1=4.1\times 10^{-3}$, $\beta_2=0.1\hspace{-2pt}\times\hspace{-2pt}\beta_1$ and $p=0.6201$ ($\rho_1(\widehat{\alpha})=1.0266>0.9694
=\rho_2$) .}\label{Fig4}
\end{figure}
\begin{figure}[H]
\centering
\subfigure{
    \includegraphics[width=.4\linewidth]{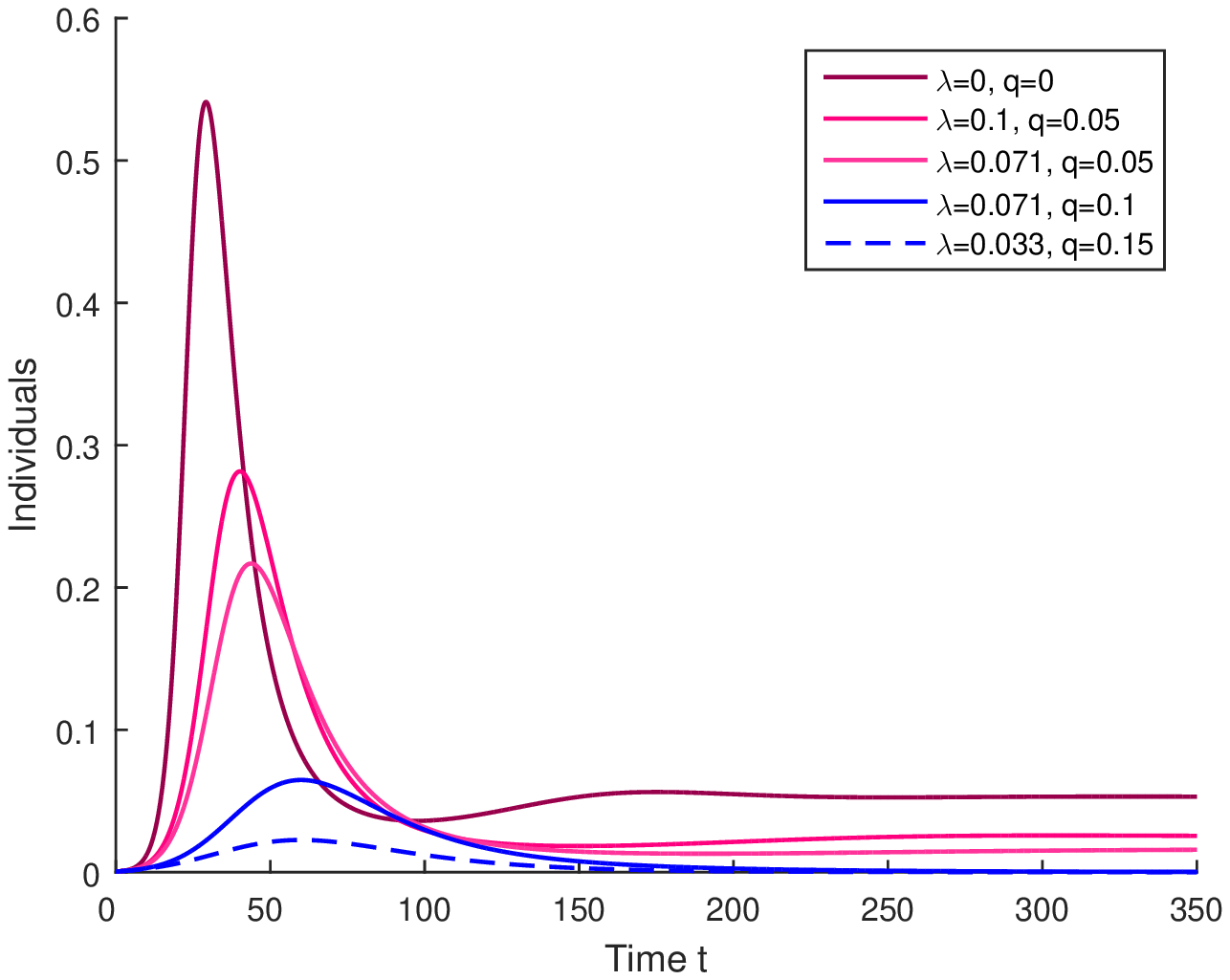}
      }%
\subfigure{
    \includegraphics[width=.4\linewidth]{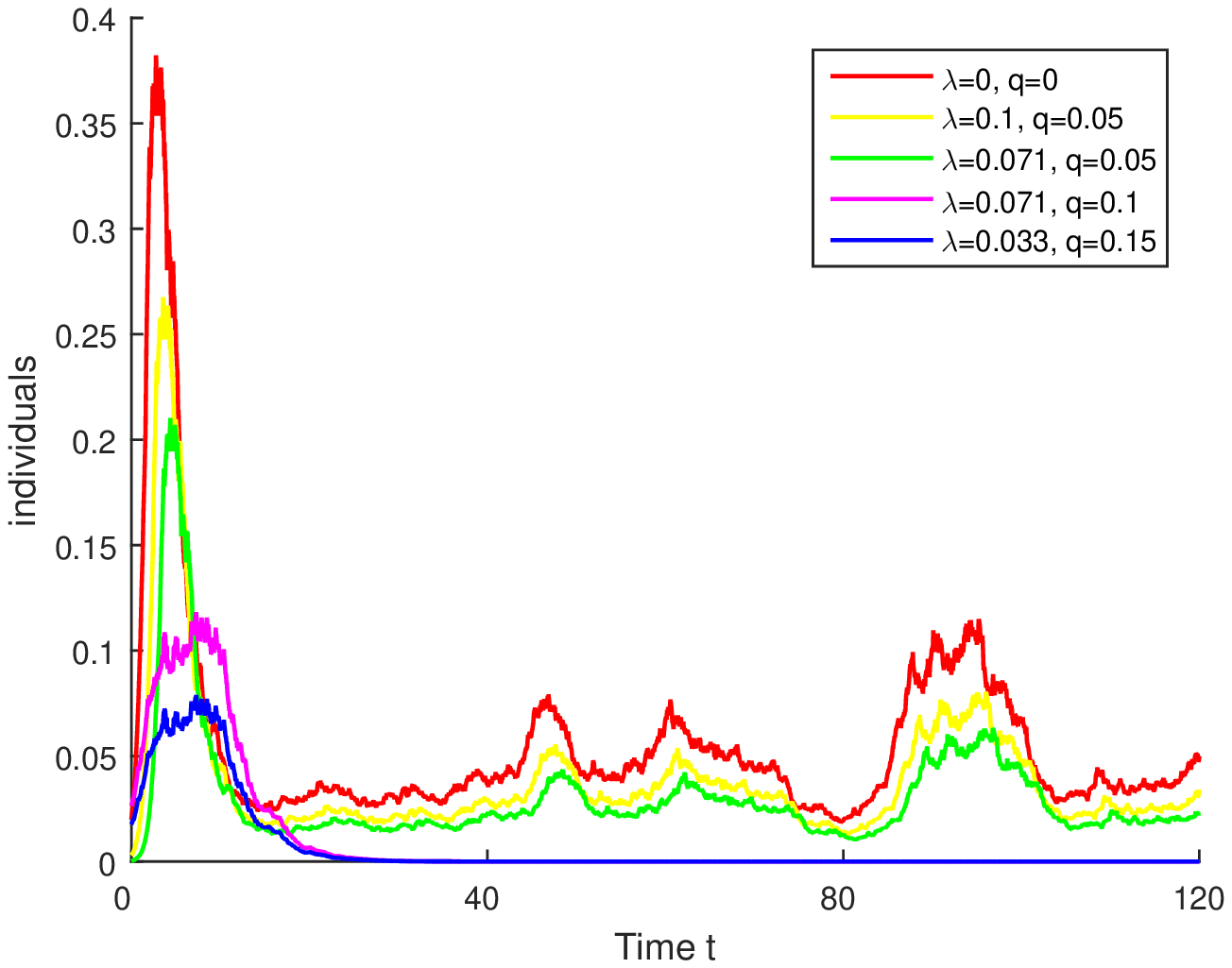}
     }  
\caption{The impact of the quarantine parameters $\lambda$ and $q$ on the deterministic and stochastic trajectories of the total  infected individuals $I_{\text{total}}(t):=E(t)+A(t)+I(t)$. The rest of the parameters is taken respectively as in Figure \ref{Fig2} and Figure \ref{Fig4}.} 
\label{Fig5} 
\end{figure}
\begin{figure}[H]
\centering
\subfigure{
    \includegraphics[width=.4\linewidth]{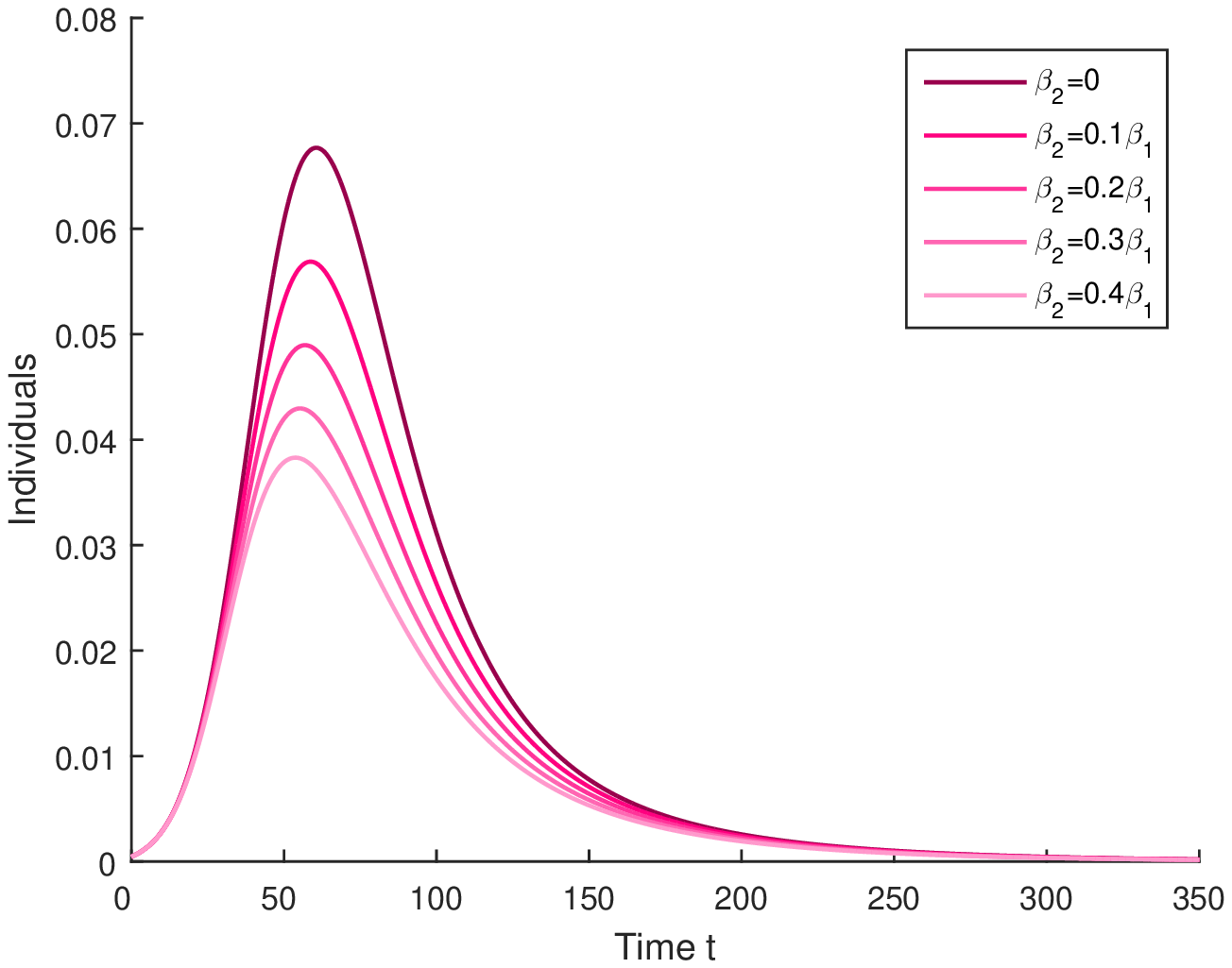}
    }%
\subfigure{
    \includegraphics[width=.4\linewidth]{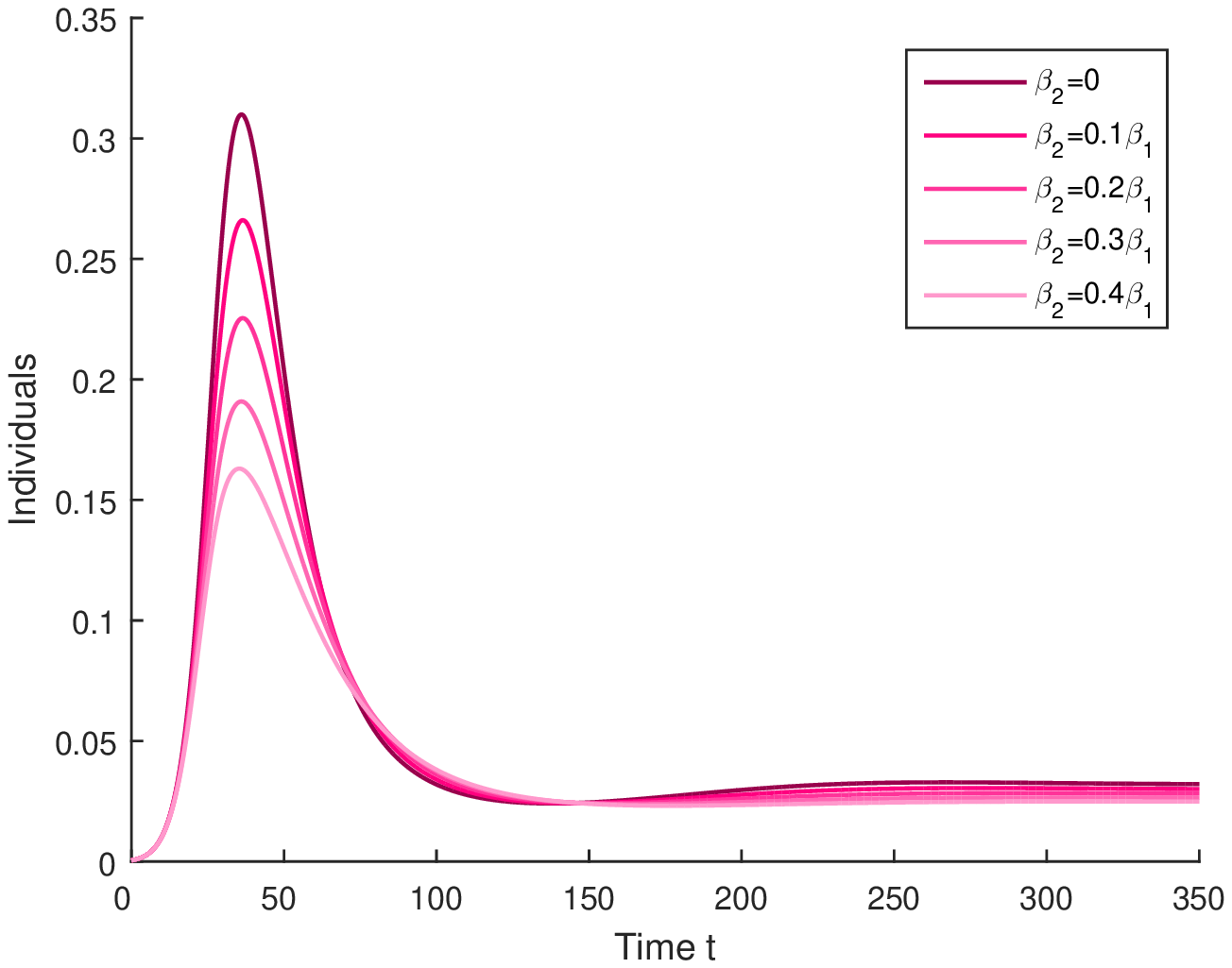}
   } 
 \caption{The effect of the awareness rate $\beta_2$ on the deterministic trajectories of the total infected individuals $I_{\text{total}}(t):=E(t)+A(t)+I(t)$ when parameters are taken  respectively as in  Figure \ref{Fig1} and Figure \ref{Fig2}.}\label{Fig6}   
\end{figure}
\begin{figure}[H]
\centering
\subfigure{
    \includegraphics[width=.4\linewidth]{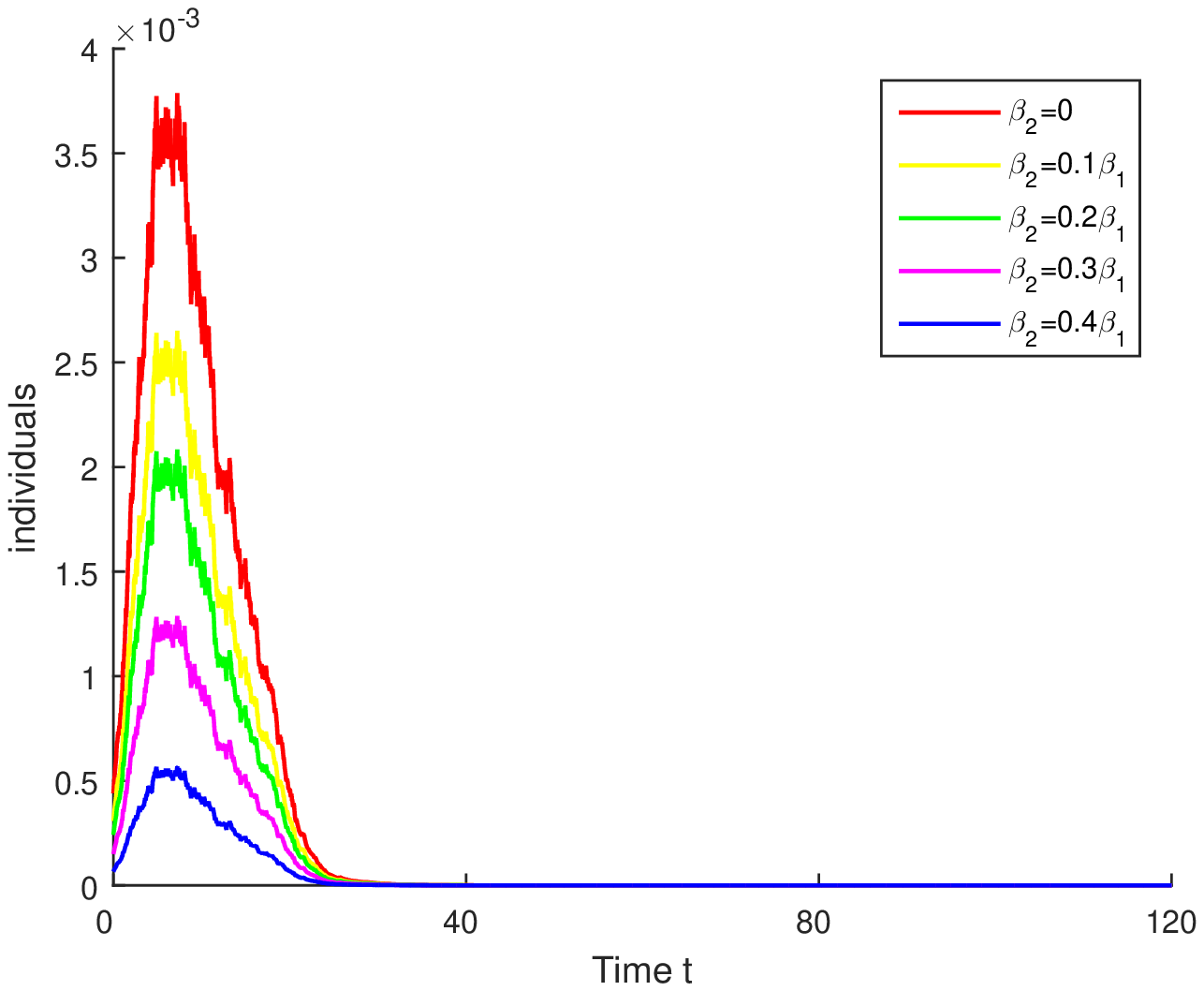}
  }%
\subfigure{
    \includegraphics[width=.4\linewidth]{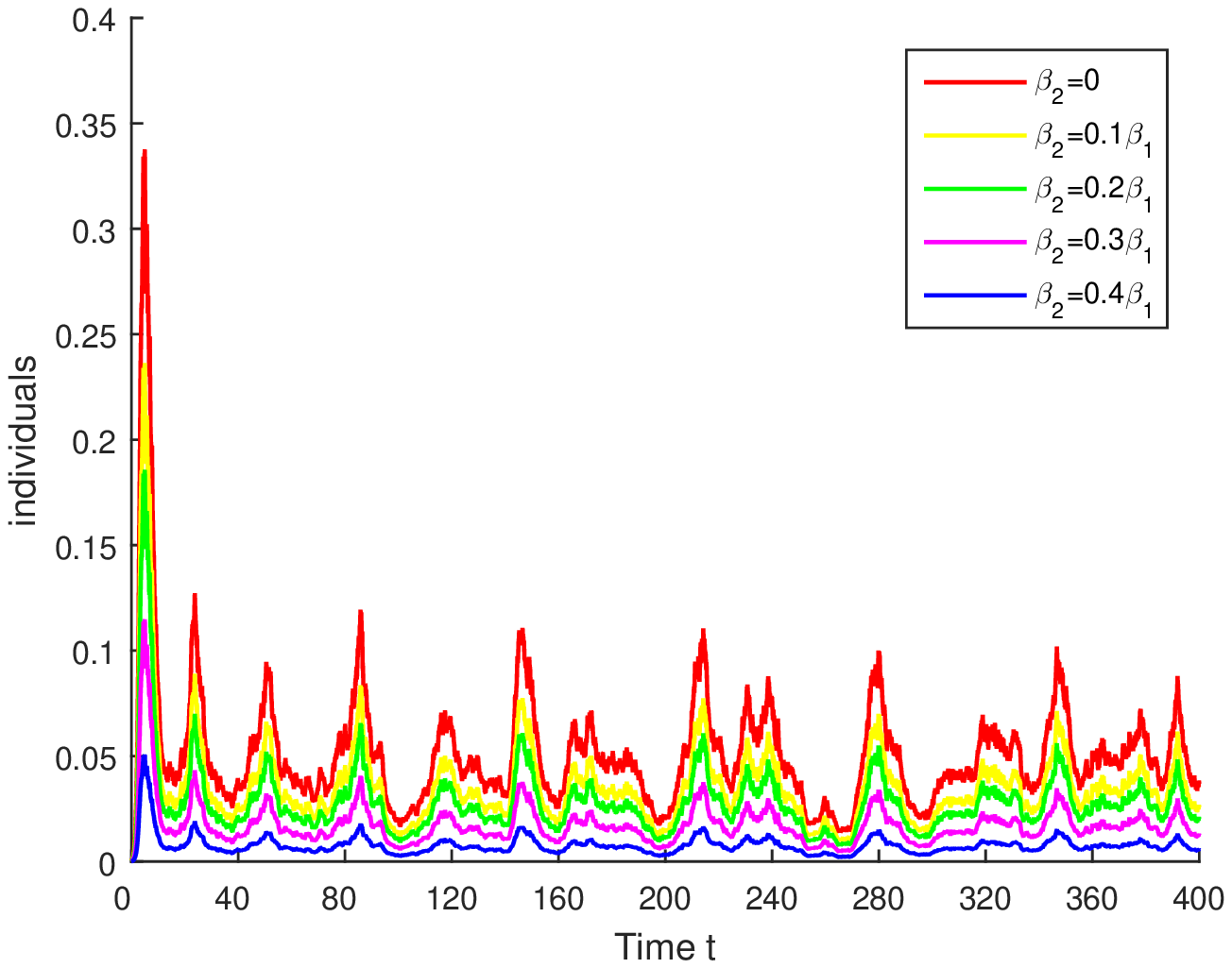}
    }  
  \caption{Stochastic paths of the  total infected individuals $I_{\text{total}}(t):=E(t)+A(t)+I(t)$ under different response intensities $\beta_2$. The other parameters are taken respectively as in Figure \ref{Fig3} and Figure \ref{Fig4}.}
  \label{Fig7}
 \end{figure} 
\section{Conclusion and discussion}\label{sec4}
The current Coronavirus disease is a major danger that threatens the whole world, and in this context, mathematical modelling is a very powerful tool for knowing more about how such a virus is transmitted within a host population of humans. In this regard, an SQEAIHR epidemic model that describe the COVID-19 dynamics under the application of quarantine and coverage media strategies is proposed on both deterministic and stochastic forms in this work. Moreover, a rigorous mathematical analysis of this model is performed to get an overview of  COVID-19 dissemination behaviour. The principal epidemiological and mathematical findings of our study are presented as follows:
\begin{itemize}
\item[$\bullet$] For the deterministic version of COVID-19 model, the basic reproduction number $\mathcal{R}_0$ is calculated by using the next-generation matrix approach and based on its expression, we have determined many dynamical properties of this version. More precisely,  when  $\mathcal{R}_0<1$, the COVID-19-free steady point   $\mathcal{E}^o$ is the unique equilibrium of system \eqref{detr}, and it is globally asymptotically stable in this case. On the other side, when $\mathcal{R}_0>1$, the disease-free equilibrium $\mathcal{E}^o$ is still present but it becomes unstable, and another endemic one $\mathcal{E}^\star$ appears this time, which makes our system \eqref{detr} uniformly persistent according to  Theorem \ref{persist}.  
\item[$\bullet$] For the stochastic version of COVID-19 model, we have demonstrated the existence and uniqueness of a global positive solution, and besides this, we have established that this latter is stochastically ultimately bounded, in other words, the probability of this solution exploding in the infinite time is very low (see Theorem \ref{bounded}). During our exploration of the perturbed system \eqref{systo}, we have derived the conditions for COVID-19 extinction and persistence, and we remarked that they are mainly depending on the magnitude of the noises intensities as well as the system parameters. 
\end{itemize}
Compared to the existing literature, the novelty of our work lies in new analysis techniques and improvements which are summarized in the following items: 
\begin{itemize}
\item[$\bullet$] Our article sheds some new light on the next-generation matrix method and applies the Varga's theorem \cite{salletbook,kamgang2008computation,varga1999matrix} to show the local stability of the disease-free state without making recourse to the Routh-Hurwitz criterion, which enabled us to avoid many long calculations.
\item[$\bullet$] By eliminating a hypothetical redundancy and bringing into play the notion of the positively invariant set, our work provides an improved and generalized version of Theorem 9.2 in \cite{brauer2013mathematical}, and use it to establish the global stability of the disease-free equilibrium $\mathcal{E}^o$.  
\item[$\bullet$] For the case of the stochastic COVID-19 model \eqref{systo}, in Lemma \ref{lass} and Theorem \ref{las}, we showed that the following inequality:
         $$\rho_1(\widehat{\alpha})>\rho_2,$$ 
is a sufficient condition for the non-disappearance of COVID-19 infective individuals.
\end{itemize} 
In order to support the theoretical results and clarify the role of quarantine and awareness strategies towards the COVID-19 spreading behaviour,  we have presented some numerical simulation examples.  From the curves appearing in these simulations, and more precisely those where we have gradually varied the value of $\beta_2$, $\lambda$ and $q$, we noticed that the quarantine and awareness strategies can effectively lower the infection and reduce the sizes of infected species. Despite its remarkable efficiency, the coverage media alone is unfortunately unable to prevent the COVID-19 from persisting. This last fact can be clearly remarked and confirmed for the deterministic case by observing that $\mathcal{R}_0$ does not contain any media coverage parameter. So, in short, we conclude that the first thing that must be done in the future during confronting a new and rapidly spreading disease like COVID-19 is to adopt quarantine and media intervention strategies,  pending the emergence of an appropriate and safe treatment. 
\par We believe that our article can be a rich basis for future studies especially after the recent discovery of a new and stronger variant of COVID-19, named COVID-19-VUI–202012/01, in the United Kingdom \cite{ENG}. 
\bibliographystyle{elsart-num}
\bibliography{bibl} 

\begin{thebibliography}{10}
\bibitem{schwartz2020infections}
D.~A. Schwartz, A.~Dhaliwal, Infections in pregnancy with covid-19 and other
  respiratory rna virus diseases are rarely, if ever, transmitted to the fetus:
  experiences with coronaviruses, hpiv, hmpv rsv, and influenza, Archives of
  pathology \& laboratory medicine.
\end{thebibliography}
\end{document}